\documentclass[10pt]{article}
\usepackage{amsfonts}
\usepackage{amsfonts,latexsym}
\usepackage{amsmath}
\usepackage{amsthm}
\usepackage{amssymb}
\usepackage{mathrsfs}
\usepackage{hyperref}
\usepackage{color}
\usepackage[]{cite}
\usepackage{dsfont}
\usepackage{graphicx}
\usepackage{appendix}
\usepackage{amsmath}
\allowdisplaybreaks[4]
\usepackage{indentfirst}
\usepackage{geometry}
\newtheorem{theorem}{Theorem}[section]

\newtheorem{definition}[theorem]{Definition}

\newtheorem{lemma}[theorem]{Lemma}
\newtheorem{remark}[theorem]{Remark}

\begin{document}
	
	\title{\bf Well-posedness, mean attractors and invariant measures of stochastic discrete long-wave-short-wave resonance equations driven by locally Lipschitz nonlinear noise \footnote{The research is  supported by National Natural Science Foundation of China (No. 12371198, 12031020)}}
	
	\author{ Xia Pan$^\text{a}$,\,\,Jianhua  Huang$^\text{a}$,\,\, Juntao Wu$^\text{b}$,\,\,  Jiangwei Zhang$^\text{c}$\thanks{Corresponding author:\,\,\;jwzhang0202@yeah.net}  \,  \,
		\\
		{\small\textsl{$^\text{a}$ College of Science, National University of Defense Technology,}} \\
		{\small \textsl{ Changsha, Hunan 410073, P.R. China}}\\
			{ \small\textsl{$^\text{b}$ School of Mathematics and Statistics, Wuhan University, }}\\ 
		{ \small \textsl{Wuhan, Hubei 430072, P.R. China}}\\
		{ \small\textsl{$^\text{c}$ Institute of Applied Physics and Computational Mathematics,  }}\\ 
		{ \small \textsl{Beijing, 100088,   P.R. China}}
	}
	\date{}
	
	
\maketitle

\begin{abstract} 
	This paper is devoted to investigating the random dynamics of stochastic discrete long-wave-short-wave resonance equation, which are characterized by the following features: $(1)$ the equation contains locally Lipschitz nonlinear coupling terms $u_mv_m$ and $(B(|u(t)|^2))_m$ for $m\in \mathbb{Z}$; $(2)$ the nonlinear coefficients of noises satisfy local Lipschitz conditions; and $(3)$ the system couples real and complex equations and is infinite-dimensional. These inherent structural properties prevent the analysis from being carried out in a standard product space of the same order and make it difficult to directly verify the tightness of the distribution family of solutions. To address these challenges, we adopt a high-order product space $L^4(\Omega,C([\tau,\tau+T],\ell_c^2))\times L^2(\Omega,C([\tau,\tau+T],\ell^2))$ as the phase space and employ the technique of uniform tail-end estimates. The main results include: establishing the global well-posedness of the nonautonomous stochastic discrete long-wave-short-wave resonance equations driven by nonlinear noise in $L^4(\Omega,C([\tau,\tau+T],\ell_c^2))\times L^2(\Omega,C([\tau,\tau+T],\ell^2))$; based on this, defining the mean random dynamical system and proving the existence and uniqueness of weak $\mathscr{D}$-pullback mean random attractors. When the external forcing terms are independent of time and sample, we investigate the existence of invariant measures for the corresponding autonomous system and examine the limiting behavior of the invariant measure as the noise intensity tends to zero.

\end{abstract}

{\hspace*{2mm}  MSC 2010: 35B40, 35B41, 37L55, 60H10.}

\linespread{1.2}
\section{Introduction}
\numberwithin{equation}{section}
In this paper, we are concerned with the following nonautonomous stochastic discrete lattice long-wave-short-wave resonance equation driven by nonlinear noise defined on the integer set $\mathbb{Z}$:
	\begin{eqnarray}\label{1.1}
	\begin{cases}
		idu_m(t)=(2u_m(t)-u_{m+1}(t)-u_{m-1}(t))dt-i \alpha u_m(t)dt+u_m(t)v_m(t)dt+f_m(t)dt \\
		\qquad \qquad +\varepsilon \sum\limits_{k=1}^{\infty}\left(b_{k,m}(t)+h_{k,m}(u_m(t))\right)dW_k(t),\,\, m\in\mathbb{Z},\,\, t>\tau, \,\, \tau \in \mathbb{R},\\
		dv_m(t)=-\beta v_m(t)dt-\lambda(B(|u(t)|^2))_mdt+g_m(t)dt\\
		\qquad \qquad+\varepsilon \sum\limits_{k=1}^{\infty}\left(\gamma_{k,m}(t)+{\sigma}_{k,m}(v  _m(t))\right)dW_k(t) ,\,\, m\in\mathbb{Z},\,\, t>\tau,\,\, \tau \in \mathbb{R},\\
	\end{cases}
	\end{eqnarray}
with initial data
	\begin{eqnarray}\label{1.2}
	u_m(\tau)=u_{m,\tau},\,\, v_m(\tau)=v_{m,\tau},\,\, m\in \mathbb{Z},\,\, \tau \in \mathbb{R},
	\end{eqnarray}		
where $u_m(t)\in \mathbb{C}$ and $v_m(t) \in \mathbb{R}$, for $m\in \mathbb{Z}$, denote the unknown functions, and here $\mathbb{R}$ and $\mathbb{C}$ represent the real numbers and complex numbers, respectively; the symbol $i$ denotes the imaginary unit such that $i^2=-1$, 
 $\alpha, \beta, \lambda >0$ are  the positive constants, $|u|^2=\left(|u_m|^2\right)_{m\in \mathbb{Z}}$, $ \varepsilon\in [0,\varepsilon_0] $ with $\varepsilon_0>0$ is the intensity of noise, $ f=(f_m)_{m\in \mathbb{Z}} $, $ g=(g_m)_{m\in \mathbb{Z}} $, $ b=(b_{k,m})_{k\in \mathbb{N},m\in \mathbb{Z}} $ and $ \gamma=(\gamma_{k,m})_{k\in \mathbb{N},m\in \mathbb{Z}} $ are progressively measurable time-dependent sequences, $ h_{k,m}$ and ${\sigma}_{k,m}$ are two sequences of locally Lipschitz continuous functions, $\{W_k\}_{k\in \mathbb{N}}$ is a sequence of independent
two-side real-valued Wiener processes defined on a complete filtered probability space $ (\Omega,\mathscr{F} ,\{\mathscr{F}_t \}_{t\in \mathbb{R}},\mathbb{P}) $ satisfying the usual conditions. 

The equation \eqref{1.1}, as a lattice model, can be regarded as a discretization in the one-dimensional real line $\mathbb{R}$ with respect to the spatial variable $x$ of the following non-autonomous long-wave–short-wave resonance equation:
	\begin{eqnarray}\label{001.1}
	\begin{cases}
		idu(t)+u_{xx}(t)dt+i \alpha u(t)dt=u(t)v(t)dt+f(t)dt +\varepsilon \sum\limits_{k=1}^{\infty}\left(b_{k}(t)+h_{k}(x,u(t))\right)dW_k(t),\\
		dv(t)+\beta v(t)dt+\lambda(|u(t)|^2)_xdt=g(t)dt+\varepsilon \sum\limits_{k=1}^{\infty}\left(\gamma_{k}(t)+{\sigma}_{k}(x,v (t))\right)dW_k(t),
	\end{cases}
\end{eqnarray}
The long-wave-short-wave resonance equation \eqref{001.1} of the above type originate from fluid dynamics and plasma physics \cite{Benney-1977}. They describe a physical mechanism in which resonant interactions occur between long and short waves. Under such resonance, energy can be transferred between the two, leading to the formation of coupled wave structures. There is an extensive body of literature investigating this class of problems, see, e.g., \cite{Li-2006,Zhao-2008,Guo-1998,Liu-2020,Wang-2018} and the references therein. 

In this work, our main objective is to investigate the discrete version of \eqref{001.1}. Lattice systems arise in a wide range of applications, including biology, chemical reactions, pattern formation, nerve-pulse propagation, electrical circuits, and physics, see, e.g., \cite{Bell-1984, Chow-1996, Kapval-1991, Keener-1987, Chow-2003, Varadhan1966Asymptotic,Keener1987}, and the references therein. Over the past few decades, substantial research has been devoted to various classes of deterministic and stochastic (delay) lattice systems, with major efforts directed toward well-posedness and the long-time behavior of solutions, including global and random attractors \cite{Bates2006sd, Caraballo2008fmc, Pereira-2014, Wang-2006, ZSfan-jde, Caraballo2012jde, Wang2016jdde}, chaotic properties of solutions \cite{Chow-MP-1995}, traveling waves \cite{Bates-2003, Elmer-1999, Elmer-2001}, and invariant measures \cite{Wang-2019 JMAA,Wang2020saa,Wang2020spa, Chen2023jdde, Li2021jde, Chenpy_JGA_2023,chen2022asymptotic}. In recent years, there has been growing attention to and research on stochastic PDEs driven by state-dependent noise, also referred to as nonlinear noise. Unlike stochastic PDEs with additive or linear multiplicative noise, such systems generally cannot be transformed via the conventional Ornstein-Uhlenbeck approach into pathwise deterministic equations with random coefficients. This has motivated the development of several frameworks for studying this class of problems.

Regarding the attractors for stochastic partial differential equations driven by nonlinear noise, B. Wang developed a more general framework based on the results in \cite{CKloeden2012}, enabling the study of weak mean random attractors for the corresponding stochastic models \cite{WBX2019, Wang2019pullbackattractors}. Recently, several theoretical approaches have been developed to analyze the statistical properties of stochastic PDEs on unbounded domains and stochastic lattice systems defined on the integer set $\mathbb{Z}$ (or $\mathbb{Z}^n$), including measure attractors, invariant measures, ergodicity, and large deviation principles, providing tools to describe the asymptotic behavior of their solutions, see e.g., \cite{Wang-2019,Wang2020saa,Wang2020spa,Chen-2021,Chen-2023,Li-2022,LDS-JDE-2024,Chen2023jdde,Li2021jde,Chenpy_JGA_2023,WBX_LDP_JDE,WBX_DCDSB_LDP,LDS-JDE-2024,LZH-AML,MLZ-Procd-2024,WZH-SD-2025,ZWH-SD-2025.8}. The previously published literature has investigated various types of lattice systems, addressing issues such as weak mean random attractors, invariant measures, and ergodicity, including (fractional) reaction-diffusion lattice systems, $p$-Laplace lattice systems, reversible Selkov lattice systems, Schr\"{o}dinger lattice systems, and Klein-Gordon-Schr\"{o}dinger lattice systems. However, for the target equation \eqref{1.1} to be studied in this paper, no results are available. Therefore, this paper investigates the nonautonomous stochastic discrete long-wave-short-wave resonance equation driven by locally Lipschitz nonlinear noise. It should be noted that the analysis of this equation is not a straightforward extension of existing model results, mainly due to the following reasons:

  \begin{enumerate}
	\item[$\bullet$] 
	\textbf{Structural Complexity.} The special structure of the equation itself gives rise to some estimation difficulties. Specifically, the nonlinear coupling terms between the complex and real equations prevent the existence and asymptotic behavior of solutions from being considered in the phase space $L^4(\Omega,C([\tau,\tau+T],\ell_c^2))\times L^2(\Omega,C([\tau,\tau+T],\ell^2))$.
	
	\item[$\bullet$]
    \textbf{Local Lipschitz Continuity.} Since the first equation in the coupled system \eqref{1.1} is a complex equation and the second is a real equation, the local Lipschitz continuity of the nonlinear terms makes it impossible to directly use classical truncation in \cite{WBX-JMAA-2019} to globally approximate the nonlinear terms in the process of proving the global existence of the solution.

	\item[$\bullet$] 
     \textbf{Tightness.} The proof of the tightness of the solution's probability distribution requires the use of uniform tail estimates. However, this result cannot be directly obtained by applying It\^{o} formula to $\|\rho_n u^\varepsilon\|^4$ and $\|\rho_n v^\varepsilon\|^2$, as such an approach would not lead to the desired estimate.
	
	\item[$\bullet$] 
    \textbf{High-Order Phase Space.} Since the phase space is no longer the classical $L^2(\Omega,\ell_c^2) \times L^2(\Omega,\ell^2)$, but rather the high-order Bochner space $L^4(\Omega,\ell_c^2) \times L^2(\Omega,\ell^2)$, proving the tightness of the distribution family of the solution and the Feller property of the semigroup is no longer trivial.

	\item[$\bullet$] 
     \textbf{Continuous Dependence and Convergence.} When proving the continuous dependence of the solution and the convergence of the invariant measure with respect to the noise intensity in the high-order Bochner space, there will be some technical difficulties in obtaining the necessary stability estimates for the solution.
\end{enumerate}

The arguments mentioned above, which differ from those used in the study of global well-posedness and long-time behavior of solutions to random lattice systems driven by nonlinear noise of other types, lead to the main difficulties of this paper. We will focus on addressing these distinct challenges, with three main innovations as follows:
$(i)$ For the system of stochastic equations driven by locally Lipschitz nonlinear noise and coupled by complex and real equations, we propose a new truncation method, which includes the technique of approximating with global Lipschitz functions, as well as truncation function estimation techniques when performing uniform tail estimates for the solution. $(ii)$ For the stochastic discrete long-wave-short-wave resonance equation driven by locally Lipschitz nonlinear noise, we have explored the use of high-order product space $L^4(\Omega,C([\tau,\tau+T],\ell_c^2))\times L^2(\Omega,C([\tau,\tau+T],\ell^2))$ as the phase space to study the global well-posedness and long-time dynamics of the solution. $(iii)$ For the stochastic discrete long-wave-short-wave resonance equation, regarding the continuous dependence of the solution in $L^4(\Omega,C([\tau,\tau+T],\ell_c^2))\times L^2(\Omega,C([\tau,\tau+T],\ell^2))$ and the estimation of solution convergence with respect to the noise intensity, both fourth-order moment and second-order moment estimates must be employed to overcome the difficulties introduced by the nonlinear terms.

Inspired by the aforementioned methods, the main results of the current article are the global well-posedness,
the existence and uniqueness of weak $\mathscr{D}$-pullback mean random attractors, the existence of invariant measures, and the limiting behavior of the invariant measures as the noise intensity tends to zero for the nonautonomous stochastic discrete long-wave-short-wave resonance equation \eqref{1.1}-\eqref{1.2} driven by nonlinear noise in high-order Bochner space.

The article is organized as follows. In Section 2, we introduce some notations and transform the original equation into an abstract non-autonomous stochastic differential equation. In Section 3, we prove the existence and uniqueness of solutions to equation \eqref{1.1}-\eqref{1.2}. In Section 4, we prove the existence and uniqueness of weak pullback mean random attractors in high-order Bochner space. In Section 5, we establish the existence of invariant measures, and in the final section, we study the limiting behavior of invariant measures as the noise intensity $\varepsilon \to 0$.

\section{Preliminaries}\label{PrelSLWSW2}
In this section, we first define some Hilbert spaces consisting of real-valued and complex-valued summable bi-infinite sequences and then introduce several operators. Furthermore, we make some basic assumptions with respect to time-dependent external force terms and nonlinear noise coefficients. Finally, we reformulate the system \eqref{1.1}–\eqref{1.2} as an abstract equation.  Throughout this paper, we denote by $\textbf{c}$ a generic positive constant, which is allowed to vary in different line, and let $\mathbb{R}^\tau:=[\tau,\infty)$ and $\mathbb{R}^+:=[0,\infty)$.

Since the problem \eqref{1.1}–\eqref{1.2} is studied in the phase space  $L^4(\Omega,\ell_c^2)\times L^2(\Omega, \ell^2)$, it is necessary to define the spaces $\ell^2$ and $\ell_c^2$. They are explicitly described as follows:
\begin{equation*}
	\ell_c^2=\left\{u=\left(u_{m}\right)_{m \in \mathbb{Z}} \big| u_{m} \in \mathbb{C} \text { and } \sum_{m \in \mathbb{Z}}\left|u_{m}\right|^{2}<+\infty\right\},
\end{equation*} 
\begin{equation*}
	\ell^2=\left\{v=\left(v_{m}\right)_{m \in \mathbb{Z}}\big|v_{m} \in \mathbb{R} \text { and }  \sum_{m \in \mathbb{Z}} |v_{m}|^{2}<+\infty\right\},
\end{equation*}
where both $\ell_c^2$ and $ \ell^2$ are equipped with the following inner products and norms:
$$
(u^1, u^2)=\sum_{m \in \mathbb{Z}} u^1_{m} \overline{u^2_{m}},\quad \|u^1\|=\sqrt{(u^1, u^1)},
\qquad \forall u^1=\left(u^1_{m}\right)_{m \in \mathbb{Z}}, u^2=\left(u^2_{m}\right)_{m \in \mathbb{Z}} \in \ell_c^2,
$$
$$
(v^1, v^2)=\sum_{m \in \mathbb{Z}} v^1_{m} {v}^2_{m},\quad\|v^1\|=\sqrt{(v^1, v^1)},
\qquad \forall v^1=\left(v^1_{m}\right)_{m \in \mathbb{Z}}, v^2=\left(v^2_{m}\right)_{m \in \mathbb{Z}} \in \ell^2,
$$
where $\overline{u^2_{m}}$ denotes the conjugate of $u^2_m$.

For any $u=\left(u_{m}\right)_{m \in \mathbb{Z}}$, the linear operators $A$ and $B$ are defined by
\begin{eqnarray*}
(A u)_m=-u_{m-1}+2u_m-u_{m+1} \text{~~ and ~~} (B u)_m=u_{m+1}-u_m,\quad \forall u=\left(u_{m}\right)_{m \in \mathbb{Z}}\in \ell^2 \text{ or } \ell_c^2.
\end{eqnarray*}
We define the adjoint operator $B^*$ of $B$ by
 \begin{eqnarray*}
 	(B^* u)_m=u_{m-1}-u_m,\quad \forall u=\left(u_{m}\right)_{m \in \mathbb{Z}}\in \ell^2 \text{ or } \ell_c^2.
 \end{eqnarray*}
By simple calculation, we can verify the following facts:
\begin{eqnarray*}
	(A u, v)=(B^*Bu,v)=(B u, B v),  \quad (B u, v)=(u, B^*v), \quad \forall\, v=\left(v_{m}\right)_{m \in \mathbb{Z}}, u=\left(u_{m}\right)_{m \in \mathbb{Z}} \in \ell^2 \text{ or } \ell_c^2.
\end{eqnarray*}
Obviously, we can deduce that 
\begin{eqnarray*}
(Au, u)=\left\|Bu\right\|^2\ge 0, \quad \forall u=\left(u_{m}\right)_{m \in \mathbb{Z}} \in \ell^2 \text{ or } \ell_c^2.
\end{eqnarray*}
In particular, for $u=\left(u_{m}\right)_{m \in \mathbb{Z}} \in \ell^2 \text{ or } \ell_c^2$, we can get
$\|B u\|^{2} \leq 4\|u\|^{2}$.

To discuss the existence, uniqueness and long-time uniform estimates of solutions, we impose the following assumptions. First, for the time-dependent external force terms, the following condition must be satisfied.

{$(H_1)$} Assume that $ f=(f_m)_{m\in \mathbb{Z}} $ and $ b_k=(b_{k,m})_{m\in \mathbb{Z}} $ belong to $ L^4_{loc}(\mathbb{R},L^4(\Omega,\ell_c^2)) $, and $ g=(g_m)_{m\in \mathbb{Z}} $, $ \gamma_k=(\gamma_{k,m})_{m\in \mathbb{Z}} $ belong to $ L^4_{loc}(\mathbb{R},L^4(\Omega,\ell^2)) $. Namely, for any $ \tau \in \mathbb{R}$ and $ T>0 $,
\begin{equation*}\label{A1}
	\int_\tau ^{\tau  + T}  \mathbb{E}\left[ \left\|f(t)\right\|^4 + \left\|g(t)\right\|^4 + \left(\sum\limits_{k \in \mathbb{N}} {\left\| {{b_k}(t)} \right\|^2}\right)^2  + \left(\sum\limits_{k \in \mathbb{N}} {\left\| {{\gamma_k}(t)} \right\|^2}\right)^2  \right]dt < \infty,
\end{equation*}
where $\mathbb{E}\left[\cdot\right]$ denotes the mathematical expectation of a random variable.
	For the sake of simplicity, we also denote 
	$$ \left\| b(t)\right\|^2=\sum\limits_{k \in \mathbb{N}} {\left\| {{b_k}(t)} \right\|_{}^2}  \,\, \,\,\text{and }\,\, \left\| \gamma(t)\right\|^2=\sum\limits_{k \in \mathbb{N}} {\left\| {{\gamma_k}(t)} \right\|_{}^2},\quad \forall t\in \mathbb{R}.  
	$$

Moreover, for the nonlinear noise coefficients, we give the following assumptions.
	
$(H_2)$ Assume that $h_{k, m}: \mathbb{C}\rightarrow \mathbb{C}$ and $\sigma_{k, m}: \mathbb{R}\rightarrow \mathbb{R}$ are locally Lipschitz continuous  uniformly with respect to  $ k\in \mathbb{N} $ and $m\in \mathbb{Z}$. More  precisely, for any compact subsets $\mathscr{I}_1\subset \mathbb{C}$ and $\mathscr{I}_2\subset \mathbb{R}$, there exist the constants $\mathcal{L}^j_{k,m,\mathscr{I}_j}>0$ $(j=1,2)$ such that for all $ k\in \mathbb{N} $ and $m\in \mathbb{Z}$,
\begin{equation*}\label{G01}
	\left|h_{k, m}\left(s^1\right)-h_{k, m}\left(s^2\right)\right|  \leq \mathcal{L}^1_{k,m,\mathscr{I}_1}\left|s^1-s^2\right|,\quad  \forall s^1,s^2 \in \mathscr{I}_1,
\end{equation*}
and
\begin{equation*}\label{G02}
\left|\sigma_{k, m}\left(s^1\right)-\sigma_{k, m}\left(s^2\right)\right| \leq \mathcal{L}^2_{k,m,\mathscr{I}_2}\left|s^1-s^2\right|,\quad  \forall s^1,s^2 \in \mathscr{I}_2,
\end{equation*}
where we set $\mathcal{L}^j=\left({\mathcal{L}}_{k,m,\mathscr{I}_j}^j\right)_{k\in \mathbb{N},m\in \mathbb{Z}}\in \ell^2$ with $ \left\|\mathcal{L}^j\right\|^2 = \sum\limits_{k\in \mathbb{N}}\sum\limits_{m\in \mathbb{Z}}|{\mathcal{L}}^j_{k,m,\mathscr{I}_j}|^2<\infty$ for each $j=1,2$.

$(H_3)$ For any $ k\in \mathbb{N} $, $m \in \mathbb{Z} $, there exists a positive sequence $ \delta_{k,m}$ such that for $ s \in \mathbb{C} \,\text{ or } \mathbb{R},$
\begin{equation*}\label{G3}
	\left|h_{k, m}(s)\right| \vee|\sigma_{k, m}(s)| \leq \delta_{k, m}(1+|s|),
\end{equation*}
where we set $\mathcal{\delta}=({\mathcal{\delta}}_{k,m})_{k\in \mathbb{N},m\in \mathbb{Z}}\in \ell^2$ with $ \left\|\delta\right\|^2= \sum\limits_{k\in \mathbb{N}}\sum\limits_{m\in \mathbb{Z}}|\delta_{k,m}|^2<\infty$.

For each $k\in \mathbb{N}$, we define the operators $h_k: \ell_c^2 \rightarrow \ell_c^2$ and $\sigma_{k}:\ell^2 \rightarrow \ell^2$ by
$$
h_k(u)=\left(h_{k,m}(u_m)\right)_{m\in \mathbb{Z}}, \quad \forall u=(u_m)_{m\in \mathbb{Z}}\in\ell_c^2,
$$
and
$$
\sigma_k(v)=\left(\sigma_{k,m}(v_m)\right)_{m\in \mathbb{Z}}, \quad \forall v=(v_m)_{m\in \mathbb{Z}}\in \ell^2.
$$
By $(H_3)$ we can derive that for every $u=(u_m)_{m\in \mathbb{Z}}\in\ell_c^2$ and $v=(v_m)_{m\in \mathbb{Z}}\in \ell^2$, 
\begin{align}\label{opformgrow}
\sum_{k\in \mathbb{N}} \|h_k(u)\|^2\leq 2\|\delta\|^2\left(1+\|u\|^2\right)
\text{~ and ~} \sum_{k\in \mathbb{N}} \|\sigma_{k}(v)\|^2\leq 2\|\delta\|^2\left(1+\|v\|^2\right),
\end{align}
which implies that $h_k: \ell_c^2 \rightarrow \ell_c^2$ and $\sigma_{k}:\ell^2 \rightarrow \ell^2$ are well-defined. Furthermore, it is easy to verify that $h_k: \ell_c^2 \rightarrow \ell_c^2$ and $\sigma_{k}:\ell^2 \rightarrow \ell^2$ are locally Lipschitz continuous, that is, for every $n>0$, there exist $\textbf{c}_1=\textbf{c}_1(n)>0$ and  $\textbf{c}_2=\textbf{c}_2(n)>0$ such that for all $u^1,u^2\in \ell_c^2$ and $v^1,v^2\in \ell^2$  with $\|u^1\|, \|u^2\|\leq n$ and $\|v^1\|, \|v^2\|\leq n$, it holds
\begin{align}\label{opformLip}
\sum_{k\in \mathbb{N}} \|h_k(u^1)-h_k(u^2)\|^2\leq \textbf{c}_1 \|u^1-u^2\|^2
\text{~ and ~} \sum_{k\in \mathbb{N}} \|\sigma_{k}(v^1)-\sigma_{k}(v^2)\|^2\leq \textbf{c}_2 \|v^1-v^2\|^2.
\end{align}

\begin{remark}
	In assumption $(H_1)$, we can also suppose that $ f=(f_m)_{m\in \mathbb{Z}} $ and $ b_k=(b_{k,m})_{m\in \mathbb{Z}} $ belong to $ L^4_{loc}(\mathbb{R},L^4(\Omega,\ell_c^2)) $, and $ g=(g_m)_{m\in \mathbb{Z}} $, $ \gamma_k=(\gamma_{k,m})_{m\in \mathbb{Z}} $ belong to $ L^2_{loc}(\mathbb{R},L^2(\Omega,\ell^2)) $. It should be emphasized that the assumptions regarding the time-dependent external force sequences do not affect the subsequent conclusions on the existence and uniqueness of solutions or the random dynamics. 
\end{remark}

$\bullet$ We define two operators $F:\ell_c^2\times \ell^2 \rightarrow \ell_c^2$ and $G:\ell_c^2 \rightarrow \ell^2$ by $F(u,v)=(u_mv_m)_{m\in\mathbb{Z}}$ and $G(u)=\lambda\left(\left(B(|u|^2)\right)_m\right)_{m\in\mathbb{Z}}$ for any $u=\left(u_{m}\right)_{m \in \mathbb{Z}} \in \ell^2_c$ and $v=\left(v_{m}\right)_{m \in \mathbb{Z}} \in \ell^2$. It is easy to obtain that
for any $u^1, u^2\in \ell_c^2$ and $v^1, v^2\in \ell^2$,
\begin{align}\label{Flojo}
	\begin{split}
		\|F(u^1,v^1)-F(u^2,v^2)\|^2
		&=\sum_{m \in \mathbb{Z}} |u^1_m(v_m^1-v_m^2)+v_m^2(u^1_m-u^2_m)|^2\\
		&\leq 2\sum_{m \in \mathbb{Z}}|u^1_m|^2 |v_m^1-v_m^2|^2
		+2\sum_{m \in \mathbb{Z}}|v_m^2|^2|u^1_m-u^2_m|^2\\
		&\leq 2\left(1+\|u^1\|^2+\|v^2\|^2\right)\left(\|u^1-u^2\|^2+\|v^1-v^2\|^2\right).
	\end{split}
\end{align}
Moreover, for any $u^1, u^2\in \ell_c^2$, it holds
\begin{align}\label{Glojo}
	\begin{split}
		\|G(u^1)-G(u^2)\|^2
		&\leq 4\lambda^2\||u^1|^2-|u^2|^2\|^2\leq 8\lambda^2 \left(\|u^1\|^2+\|u^2\|^2\right)\|u^1-u^2\|^2.
	\end{split}
\end{align}
It follows from \eqref{Flojo}-\eqref{Glojo} that $F(u,v)$ and $G(u)$ satisfy  locally Lispchiz conditions. In other words, for every $n\in \mathbb{N}$, there exist $\textbf{c}_3(n)>0$ and $\textbf{c}_4(n)>0$ such that for any $u^1, u^2\in \ell_c^2$, $v^1, v^2\in \ell^2$  and $\|u^1\|\leq n$, $\|u^2\|\leq n$, $\|v^1\|\leq n$, $\|v^2\|\leq n$,
\begin{align}\label{FGlojo}
	\begin{split}
		\|F(u^1,v^1)-F(u^2,v^2)\|^2&\leq \textbf{c}_3(n)\left(\|u^1-u^2\|^2+\|v^1-v^2\|^2\right),\\
		\|G(u^1)-G(u^2)\|^2&\leq \textbf{c}_4(n)\|u^1-u^2\|^2.
	\end{split}
\end{align}

With the help of the aforementioned notations, we can rewrite the system \eqref{1.1}-(\ref{1.2}) in $ \ell_c^2\times \ell^2 $ as follows:
\begin{eqnarray}\label{2.1}
	\begin{aligned}
	\begin{cases}
		idu(t)=\left(Au(t)-i\alpha u(t)+F(u(t),v(t))+f(t)\right)dt+\varepsilon \sum\limits_{k=1}^{\infty}(h_{k}(u(t))+b_{k}(t))dW_k(t),\,\,  t>\tau,\\
		dv(t)=\left(-\beta v(t)-G(u(t))+g(t)\right)dt+\varepsilon \sum\limits_{k=1}^{\infty}(\sigma_{k}(v(t))+\gamma_{k}(t))dW_k(t) ,\,\, t>\tau,\\
		(u(\tau),v(\tau))=(u_{\tau},{v}_{\tau}),\,\, \tau \in \mathbb{R}.
	\end{cases}
	\end{aligned}
\end{eqnarray}
In order to represent the above system \eqref{2.1} in an abstract form, we introduce some addition notations. Let
$$
\varphi=(u,v)^\mathrm{T},\quad  \mathscr{G}(\varphi,t)=\left(-i F(u(t),v(t))-if(t), -G(u(t))+g(t) \right)^\mathrm{T},
$$
and
\[
\mathfrak{A}=\begin{pmatrix}
	iA+\alpha I  & 0\\
	0 & \beta I 
\end{pmatrix},
\quad 
\mathscr{H}_k(\varphi,t)=
\begin{pmatrix}
	-ih_{k}(u(t))-ib_{k}(t)\\
	\sigma_{k}(v(t))+\gamma_{k}(t) 
\end{pmatrix}.
\]
Then, we have the following abstract non-autonomous stochastic differential equations:
\begin{eqnarray}\label{2.002}
	\begin{aligned}
		\begin{cases}
    d\varphi+\mathfrak{A}\varphi dt=	\mathscr{G}(\varphi,t)dt +\varepsilon\sum\limits_{k=1}^{\infty}\mathscr{H}_k(\varphi,t)dW_k(t) ,\,\, t>\tau,\,\, \tau \in \mathbb{R},\\
	\varphi(\tau)=\varphi_{\tau},\,\, \tau \in \mathbb{R},
\end{cases}
\end{aligned}
\end{eqnarray}
where $\varphi_{\tau} \in \ell_c^2\times \ell^2$.

\section{Existence and uniqueness of solutions}
This section mainly studies the existence and uniqueness of solutions to problem \eqref{2.002}. 
To this end, based on the preparatory work in Section \ref{PrelSLWSW2}, we consider the solutions of problem  \eqref{2.002} in the following sense:
\begin{definition}\label{Definition2.1}
	For any $\tau\in \mathbb{R}$ and $\mathscr{F}_\tau$-measurable  $\varphi_\tau \in L^4(\Omega,\ell_c^2)\times L^2(\Omega, \ell^2)$, an $ \ell_c^2 \times \ell^2$-valued $ \mathscr{F}_t $-adapted stochastic process $ \varphi(t)= (u(t),v(t))$ is called a solution of system \eqref{2.002} if
 $ \varphi(t)\in L^2(\Omega,C([\tau,\tau+T],\ell_c^2))\times L^2(\Omega,C([\tau,\tau+T], \ell^2))$ for any $T>0$, and for almost surely $ \omega\in\Omega $,
	 \begin{eqnarray}\label{2.2}
	 	\begin{aligned}
	 			\varphi(t)+\int_{\tau}^{t} \mathfrak{A}\varphi(s)ds=\varphi_{\tau}
	 			+\int_{\tau}^{t}\mathscr{G}(\varphi(s),s)ds+\varepsilon \sum\limits_{k=1}^{\infty}\displaystyle\int_{\tau}^{t}\mathscr{H}_k(\varphi(s),s)dW_k(s)
	 	\end{aligned}
	 \end{eqnarray}
	 in $\ell_c^2\times \ell^2$ for all $ t\geq \tau $.
\end{definition}

By assumption $(H_3)$, we know that the nonlinear drift terms $F$, $G$ and the nonlinear diffusion coefficients $h_k$, $\sigma_{k}$ are locally Lipschitz continuous. 
To establish the global existence of a solution to system \eqref{2.1} in the sense of Definition \ref{Definition2.1},
it is necessary to approximate the locally Lipschitz nonlinearities by globally Lipschitz continuous ones. For this purpose, for every $n \in \mathbb{N}$, we introduce two cut-off functions $\rho_{n}^{\mathbb{C}} : \mathbb{C} \rightarrow \mathbb{C}$ and $\rho_{n}^{\mathbb{R}} : \mathbb{R} \rightarrow \mathbb{R}$ defined by
\begin{align}\label{2.C6}
	\rho_{n}^{\mathbb{C}}(z)=\left\{\begin{array}{ll}
		z,\quad&\textrm{if}\quad |z|\leq n,\\
		n\frac{z}{|z|},\quad&\textrm{if}\quad |z|>n,
	\end{array} \right.
	\text{~~~and~~~}
	\rho_{n}^{\mathbb{R}}(s)=\left\{\begin{array}{ll}
		s,\quad&\textrm{if}\quad |s|\leq n,\\
		n\frac{s}{|s|},\quad&\textrm{if}\quad |s|>n.
	\end{array} \right.
\end{align}
We observe that $\rho_{n}^{\mathbb{C}} : \mathbb{C} \rightarrow \mathbb{C}$ and $\rho_{n}^{\mathbb{R}} : \mathbb{R} \rightarrow \mathbb{R}$ are both increasing and satisfy the global Lipschitz condition:
\begin{align}\label{2.C27}
	|&\rho_{n}^{\mathbb{C}}(z_{1})-\rho_{n}^{\mathbb{C}}(z_{2})|\leq 2|z_{1}-z_{2}|,\quad \forall z_{1},z_{2}\in \mathbb{C},\\
	&|\rho_{n}^{\mathbb{R}}(s_{1})-\rho_{n}^{\mathbb{R}}(s_{2})|\leq|s_{1}-s_{2}|,\quad \forall s_{1},s_{2}\in \mathbb{R}.
\end{align}
In addition, we have 
\begin{align}\label{2.28}
	\rho_{n}^{\mathbb{C}}(0)=0,\;|\rho_{n}^{\mathbb{C}}(z)|\leq n,\quad \forall z\in \mathbb{C}
	\text{~~~and~~~}
	\rho_{n}^{\mathbb{R}}(0)=0,\;|\rho_{n}^{\mathbb{R}}(s)|\leq n,\quad \forall s\in \mathbb{R}.
\end{align}
Particularly, let $\arg(\rho_{n}^{\mathbb{C}}(z))=\arg(z)$ for $z\neq 0$.

For every $k,n\in \mathbb{N}$ and any $u=(u_{m})_{m\in \mathbb{Z}}\in \ell_c^{2}$ and $v=(v_{m})_{m\in \mathbb{Z}}\in \ell^{2}$, let 
\begin{align}\label{2.CRR27}
	\begin{split}
		F^n(u,v)=\left((\rho_{n}^{\mathbb{C}} u_m)(\rho_{n}^{\mathbb{R}} v_m)\right)_{m\in \mathbb{Z}},\quad G^n(u)=\lambda\left(\left(B\left(|\rho_{n}^{\mathbb{C}} u|^2\right)\right)_m\right)_{m\in \mathbb{Z}},
	\end{split}
\end{align}
and
\begin{align}\label{2.29}
h_{k}^{n}(u)=h_k(\rho_{n}^{\mathbb{C}}(u))=(h_{k,m}(\rho_{n}^{\mathbb{C}}(u_{m})))_{m\in \mathbb{Z}}, \quad \sigma_{k}^{n}(v)=\sigma_k(\rho_{n}^{\mathbb{R}}(v))=(\sigma_{k,m}(\rho_{n}^{\mathbb{R}}(v_{m})))_{m\in \mathbb{Z}}.
\end{align}

By \eqref{opformLip}, \eqref{FGlojo} and \eqref{2.C27}-\eqref{2.CRR27} we know that $F^n:\ell_c^2\times \ell^2 \rightarrow \ell^2_c$, $G^n:\ell_c^2 \rightarrow \ell^2$, $h^n_k:\ell_c^2\rightarrow \ell_c^2$  and $\sigma_{k}^n:\ell^2\rightarrow \ell^2$ are globally Lipschitz continuous; that is, for every $k,n\in \mathbb{N}$, there exist $\mathcal{Q}_1(n)>0$, $\mathcal{Q}_2(n)>0$, $\mathcal{Q}_3(n)>0$, $\mathcal{Q}_4(n)>0$ depending only on $n$ such that for all $u, u^1, u^2\in \ell_c^2$ and $v, v^1, v^2\in \ell^2$, it holds
\begin{align}\label{TeFG}
	\begin{split}
		\|F^n(u^1,v^1)-F^n(u^2,v^2)\|^2&\leq \mathcal{Q}_1(n)\left(\|u^1-u^2\|^2+\|v^1-v^2\|^2\right),\\
		\|G^n(u^1)-G^n(u^2)\|^2&\leq \mathcal{Q}_2(n)\|u^1-u^2\|^2,
	\end{split}
\end{align}
and
\begin{align}\label{Tehsuif}
	\begin{split}
		\sum_{k\in \mathbb{N}}\|h^n_k(u^1)-h^n_k(u^2)\|^2&\leq \mathcal{Q}_3(n)\|u^1-u^2\|^2,\\
		\sum_{k\in \mathbb{N}}\|\sigma_{k}^n(v^1)-\sigma_{k}^n(v^2)\|^2&\leq \mathcal{Q}_4(n)\|v^1-v^2\|^2.
	\end{split}
\end{align}
Moreover, we have $F^n(0,0)=0$, $G^n(0)=0$ and
\begin{align}\label{TeFGkoko}
	\begin{split}
		\sum_{k\in \mathbb{N}}\|h^n_k(u)\|^2\leq 2\|\delta\|^2(1+\|u\|^2),\quad
		\sum_{k\in \mathbb{N}}\|\sigma^n_k(v)\|^2 \leq 2\|\delta\|^2(1+\|v\|^2).
	\end{split}
\end{align}

For every $n\in \mathbb{N}$, we consider the following approximate stochastic system in $\ell_c^2\times \ell^{2}$,  
\begin{eqnarray}\label{2.appro2}
	\begin{aligned}
		\begin{cases}
			d\varphi^n(t)+\mathfrak{A}\varphi^n(t) dt=	\mathscr{G}^n(\varphi^n(t),t)dt +\varepsilon\sum\limits_{k=1}^{\infty}\mathscr{H}^n_k(\varphi^n(t),t)dW_k(t),\,\, t>\tau, \tau\in \mathbb{R},\\
			\varphi^n(\tau)=\varphi_{\tau},\,\, \tau \in \mathbb{R},
		\end{cases}
	\end{aligned}
\end{eqnarray}
where $\varphi^n(t)=(u^n(t),v^n(t))^\mathrm{T}$,
\[
\mathfrak{A}\varphi^n(t)=\begin{pmatrix}
	iAu^n(t)+\alpha u^n(t) \\
	\beta v^n(t) 
\end{pmatrix},
\quad 
\mathscr{G}^n(\varphi^n(t),t)=\begin{pmatrix}
	-i F^n(u^n(t),v^n(t))-if(t)\\
	-G^n(u^n(t))+g(t) 
\end{pmatrix},
\]
and
\[\mathscr{H}_k^n(\varphi^n(t),t)=
\begin{pmatrix}
	-ih_{k}^n(u^n(t))-ib_{k}(t)\\
	\sigma_{k}^n(v^n(t))+\gamma_{k}(t) 
\end{pmatrix}.
\]

Following the standard theory of the existence of solutions for stochastic differential equations on the entire space $\mathbb{R}^{n}$ (see e.g., \cite{Arnold-1974}), it follows from \eqref{TeFG}-\eqref{TeFGkoko} that for all $n \in \mathbb{N}$, $\tau \in \mathbb{R}$, and any $\mathscr{F}_{\tau}$-measurable $\varphi_{\tau} \in L^{2}(\Omega, \ell^{2}_c\times \ell^2)$, the approximate system \eqref{2.appro2} admits a unique solution $\varphi^{n} \in L^{2}(\Omega, C([\tau, \infty), \ell^{2}_c\times \ell^2))$ in the sense of Definition \ref{Definition2.1}, with $F$, $G$ and $\mathscr{H}_{k}$ replaced by $F^{n}$, $G^n$ and $\mathscr{H}^{n}_{k}$, respectively.

We now establish the existence and uniqueness of solutions to system \eqref{2.002} in the sense of Definition \ref{Definition2.1} via a limiting procedure. Specifically, we will examine the limiting behavior of the solution sequence $\{\varphi^{n}\}_{n=1}^\infty$ of the approximate system \eqref{2.appro2} as $n \rightarrow \infty$. As a key step, for every $n \in \mathbb{N}$, $t \geq \tau$, and $T > 0$, we introduce the stopping time $\tau_n$ defined by
\begin{align}\label{2.37}
	\tau_{n}=\inf\{t\geq \tau:\|u^{n}(t)\|+\|v^{n}(t)\|>n\},
\end{align}
where, as usual, $\tau_{n}=+\infty$ if $\{t\geq \tau:\|u^{n}(t)\|+\|v^{n}(t)\|>n\}=\emptyset$.
For convenience, we let $\varphi^n_{\tau_n}=\varphi^n(t\wedge\tau_{n})$, then the sequence $\{\varphi^n_{\tau_n}\}_{n=1}^\infty$ is consistent in the following sense.
\begin{lemma}\label{lemma2.2}
	Suppose that $(H_1)-(H_3)$ hold, and let $\varphi^{n}(t)=(u^n(t),v^n(t))^\mathrm{T}$ be the solution of system \eqref{2.appro2}.
	Then, we have the following conclusions
	\begin{align}\label{2.38}
		\varphi^{n+1}_{\tau_n}=\varphi^n_{\tau_n}\quad \text{and}\quad\tau_{n+1}\geq\tau_{n},\quad a.s.,\quad \forall t\geq\tau,n\in \mathbb{N},
	\end{align}
	where $\tau_{n}$ is the stopping time given by \eqref{2.37}.
\end{lemma}
\begin{proof}
	Applying Ito's formula and combining with \eqref{2.appro2}, by taking the real part we have that a.s., 
	\begin{align}\label{LSWs2.23}
		\begin{split}
			&~~\|u^{n+1}(t\wedge\tau_{n})-u^{n}(t\wedge\tau_{n})\|^2
			+2\alpha \int_{\tau}^{t\wedge\tau_{n}}\|u^{n+1}(r)-u^{n}(r)\|^2dr\\
			&=2\int_{\tau}^{t\wedge\tau_{n}}\textbf{Im}\left(F^{n+1}\left({u^{n+1}(r),v^{n+1}(r)}\right)-F^n\left(u^{n}(r),v^{n}(r)\right),u^{n+1}(r)-u^{n}(r)\right)dr\\
			&+\varepsilon^2\sum_{k=1}^\infty \int_{\tau}^{t\wedge\tau_{n}}\|h^{n+1}_{k}\left(u^{n+1}(r)\right)
			-h_{k}^{n}\left(u^{n}(r)\right)\|^{2}dr\\
			&+2\varepsilon \sum_{k=1}^\infty \int_{\tau}^{t\wedge\tau_{n}}\textbf{Im}\left(\left(u^{n+1}(r)-u^{n}(r)\right)\left(h^{n+1}_{k}\left(u^{n+1}(r)\right)
			-h_{k}^{n}\left(u^{n}(r)\right)\right)\right)dW_k(r).
		\end{split}
	\end{align}
	By \eqref{2.appro2} and applying Ito's formula again, we can derive that a.s.,
	\begin{align}\label{LSWs2.24}
		\begin{split}
			&~~\|v^{n+1}(t\wedge\tau_{n})-v^{n}(t\wedge\tau_{n})\|^2
			+2\beta \int_{\tau}^{t\wedge\tau_{n}}\|v^{n+1}(r)-v^{n}(r)\|^2dr\\
			&+2\int_{\tau}^{t\wedge\tau_{n}}\left(G^{n+1}\left(u^{n+1}(r)\right)-G^n\left(u^{n}(r)\right),v^{n+1}(r)-v^{n}(r)\right)dr\\
			&=\varepsilon^2\sum_{k=1}^\infty \int_{\tau}^{t\wedge\tau_{n}}\|\sigma^{n+1}_{k}\big(v^{n+1}(r)\big)
			-\sigma_{k}^{n}\big(v^{n}(r)\big)\|^{2}dr\\
			&+2\varepsilon \sum_{k=1}^\infty \int_{\tau}^{t\wedge\tau_{n}}\left(v^{n+1}(r)-v^{n}(r)\right)\left(\sigma^{n+1}_{k}\big(v^{n+1}(r)\big)
			-\sigma_{k}^{n}\big(v^{n}(r)\big)\right)dW_k(r).
		\end{split}
	\end{align}
	According to the Riesz
	representation theorem, we know that the last terms $u^{n+1}(r)-u^{n}(r)$ and $v^{n+1}(r)-v^{n}(r)$ in \eqref{LSWs2.23} and \eqref{LSWs2.24} are identified with the elements in the dual space of $\ell^2_c$ and $\ell^2$, respectively.
	
	By \eqref{2.37} we find that for any $r\in [\tau,\tau_{n})$, it holds $\|u^{n}(r)\| \leq n$ and $\|v^{n}(r)\| \leq n$. Hence, for all $n\in \mathbb{N}$ and $r\in [\tau,\tau_{n})$, we have
	\begin{align}\label{LSWs2.25}
		F^{n+1}\left({u^{n}(r),v^{n}(r)}\right)=F^{n}\left({u^{n}(r),v^{n}(r)}\right),\quad G^{n+1}\left(u^{n}(r)\right)=G^{n}\left(u^{n}(r)\right),
	\end{align}
	and
	\begin{align}\label{LSWs2.26}
		h_{k}^{n+1}(u^{n}(r))=h_{k}^{n}(u^{n}(r)),\quad \sigma_{k}^{n+1}(v^{n}(r))=\sigma_{k}^{n}(v^{n}(r)).
	\end{align}
	We now deal with the first term on the right-hand side of \eqref{LSWs2.23} and the third term on the left-hand side of \eqref{LSWs2.24}.
	By \eqref{2.CRR27}, \eqref{TeFG}, \eqref{LSWs2.25}  and Young's inequality we can derive that for all $r\in [\tau,\tau_{n})$,
	\begin{align}\label{LSWs2.27}
			&~~2\left|\int_{\tau}^{t\wedge\tau_{n}}\textbf{Im}\left(F^{n+1}\left({u^{n+1}(r),v^{n+1}(r)}\right)-F^n\left(u^{n}(r),v^{n}(r)\right),u^{n+1}(r)-u^{n}(r)\right)dr\right|\notag \\
			&= 2\left|\int_{\tau}^{t\wedge\tau_{n}}\textbf{Im}\left(F^{n+1}\left({u^{n+1}(r),v^{n+1}(r)}\right)-F^{n+1}\left(u^{n}(r),v^{n}(r)\right),u^{n+1}(r)-u^{n}(r)\right)dr\right|\notag \\
			&\leq  \int_{\tau}^{t\wedge\tau_{n}}\left(\|F^{n+1}\left({u^{n+1}(r),v^{n+1}(r)}\right)-F^{n+1}\left(u^{n}(r),v^{n}(r)\|^2\right)+\|u^{n+1}(r)-u^{n}(r)\|^2\right)dr\notag \\
			&\leq 2\left(\mathcal{Q}_1(n+1)+1\right)\int_{\tau}^{t\wedge\tau_{n}}\left(\|u^{n+1}(r)-u^{n}(r)\|^2+\|v^{n+1}(r)-v^{n}(r)\|^2\right)dr,
	\end{align}
	and
	\begin{align}\label{LSWs2.28}
		\begin{split}
			&~~\left|2\int_{\tau}^{t\wedge\tau_{n}}\left(G^{n+1}\left(u^{n+1}(r)\right)-G^n\left(u^{n}(r)\right),v^{n+1}(r)-v^{n}(r)\right)dr\right|\\
			&\leq \left(\mathcal{Q}_2(n+1)+1\right)\int_{\tau}^{t\wedge\tau_{n}}\left(\|u^{n+1}(r)-u^{n}(r)\|^2+\|v^{n+1}(r)-v^{n}(r)\|^2\right)dr.
		\end{split}
	\end{align}
	Furthermore, in light of \eqref{2.29} and \eqref{Tehsuif}, it can be deduced from \eqref{LSWs2.26} that the second term on the right-hand side of \eqref{LSWs2.23} and the  first term on the right-hand side of \eqref{LSWs2.24} satisfy
	\begin{align}\label{LSWs2.29}
		\begin{split}
			&\varepsilon^2\sum_{k=1}^\infty \int_{\tau}^{t\wedge\tau_{n}}\|h^{n+1}_{k}\left(u^{n+1}(r)\right)
			-h_{k}^{n}\left(u^{n}(r)\right)\|^{2}dr\\
			\leq& \varepsilon_0^2\mathcal{Q}_3(n+1)\int_{\tau}^{t\wedge\tau_{n}}\|u^{n+1}(r)-u^{n}(r)\|^{2}dr\leq \mathcal{Q}_3(n+1)\int_{\tau}^{t\wedge\tau_{n}}\|u^{n+1}(r)-u^{n}(r)\|^{2}dr,
		\end{split}
	\end{align}
	and
	\begin{align}\label{LSWs2.30}
		\varepsilon^2\sum_{k=1}^\infty \int_{\tau}^{t\wedge\tau_{n}}\|\sigma^{n+1}_{k}\big(v^{n+1}(r)\big)
		-\sigma_{k}^{n}\big(v^{n}(r)\big)\|^{2}dr
		\leq\mathcal{Q}_4(n+1)\int_{\tau}^{t\wedge\tau_{n}}\|v^{n+1}(r)-v^{n}(r)\|^{2}dr,
	\end{align}
	respectively, where we set $\varepsilon_0\leq 1$ for convenience.
	
	Together with \eqref{LSWs2.23}-\eqref{LSWs2.24} and \eqref{LSWs2.27}-\eqref{LSWs2.30}, we can infer that there exists $\mathcal{Q}_5(n)>0$ such that
	\begin{align}\label{LSWs2.31}
		\begin{split}
			&~~\|u^{n+1}(t\wedge\tau_{n})-u^{n}(t\wedge\tau_{n})\|^2
			+\|v^{n+1}(t\wedge\tau_{n})-v^{n}(t\wedge\tau_{n})\|^2\\
			&\leq \mathcal{Q}_5(n)\int_{\tau}^{t\wedge\tau_{n}}\left(\|u^{n+1}(r)-u^{n}(r)\|^2+\|v^{n+1}(r)-v^{n}(r)\|^2\right)dr\\
			&+2\varepsilon \sum_{k=1}^\infty \int_{\tau}^{t\wedge\tau_{n}}\textbf{Im}\left(\left(u^{n+1}(r)-u^{n}(r)\right)\left(h^{n+1}_{k}\left(u^{n+1}(r)\right)
			-h_{k}^{n}\left(u^{n}(r)\right)\right)\right)dW_k(r)\\
			&+2\varepsilon \sum_{k=1}^\infty \int_{\tau}^{t\wedge\tau_{n}}\left(v^{n+1}(r)-v^{n}(r)\right)\left(\sigma^{n+1}_{k}\big(v^{n+1}(r)\big)
			-\sigma_{k}^{n}\big(v^{n}(r)\big)\right)dW_k(r).
		\end{split}
	\end{align}
	Taking the supremum of both sides of inequality \eqref{LSWs2.31} over the interval $[\tau,t]$ and then computing the expectation of the resulting inequality, by \eqref{LSWs2.26} we obtain
	\begin{align}\label{LSWs2.32}
		\begin{split}
			&~~\mathbb{E}\left[\sup_{\tau\leq r\leq t}\left(\|u^{n+1}(r\wedge\tau_{n})-u^{n}(r\wedge\tau_{n})\|^2
			+\|v^{n+1}(r\wedge\tau_{n})-v^{n}(r\wedge\tau_{n})\|^2\right)\right]\\
			&\leq \mathcal{Q}_5(n)\int_{\tau}^{t}\mathbb{E}\left[\sup_{\tau\leq r\leq s}\left(\|u^{n+1}(r\wedge\tau_{n})-u^{n}(r\wedge\tau_{n})\|^2+\|v^{n+1}(r\wedge\tau_{n})-v^{n}(r\wedge\tau_{n})\|^2\right)\right]ds\\
			&+2\varepsilon_0 \mathbb{E}\left[\sup_{\tau\leq r\leq t\wedge\tau_{n}}\left| \sum_{k=1}^\infty \int_{\tau}^{r}\textbf{Im}\left(\left(u^{n+1}(s)-u^{n}(s)\right)\left(h^{n+1}_{k}\left(u^{n+1}(s)\right)
			-h_{k}^{n+1}\left(u^{n}(s)\right)\right)\right)dW_k(s)\right|\right]\\
			&+2\varepsilon_0 \mathbb{E}\left[\sup_{\tau\leq r\leq t\wedge\tau_{n}}\left|\sum_{k=1}^\infty \int_{\tau}^{r}\left(v^{n+1}(s)-v^{n}(s)\right)\left(\sigma^{n+1}_{k}\big(v^{n+1}(s)\big)
			-\sigma_{k}^{n+1}\big(v^{n}(s)\big)\right)dW_k(s)\right|\right].
		\end{split}
	\end{align}
	We handle the second and third terms on the right-hand side of \eqref{LSWs2.32}. By \eqref{Tehsuif} and the Burkholder-Davis-Gundy (BDG) inequality we obtain there exist $\mathcal{Q}_6(n)>0$ and $\mathcal{Q}_7(n)>0$ such that
	{\small
	\begin{align}\label{LSWs2.33}
			&~~2\varepsilon_0 \mathbb{E}\left[\sup_{\tau\leq r\leq t\wedge\tau_{n}}\left| \sum_{k=1}^\infty \int_{\tau}^{r}\textbf{Im}\left(\left(u^{n+1}(s)-u^{n}(s)\right)\left(h^{n+1}_{k}\left(u^{n+1}(s)\right)
			-h_{k}^{n+1}\left(u^{n}(s)\right)\right)\right)dW_k(s)\right|\right]\notag \\
			&\leq 8\sqrt{2}\varepsilon_0 \mathbb{E}\left[\left(\int_{\tau}^{t\wedge\tau_{n}}\left(\|u^{n+1}(s)-u^{n}(s)\|^{2}
			\sum_{k=1}^{\infty}\|h^{n+1}_{k}\left(u^{n+1}(s)\right)
			-h_{k}^{n+1}\left(u^{n}(s)\right)\|^{2}\right)ds
			\right)^{\frac{1}{2}}\right]\notag \\
			&\leq 8\sqrt{2}\varepsilon_0\mathbb{E}\left[\left(\sup_{\tau\leq s\leq t}\|u^{n+1}(s\wedge\tau_{n})-u^{n}(s\wedge\tau_{n})\|\right)\left(\int_{\tau}^{t\wedge\tau_{n}}\sum_{k=1}^{\infty}
			\|h^{n+1}_{k}\left(u^{n+1}(s)\right)
			-h_{k}^{n+1}\left(u^{n}(s)\right)\|^{2}ds
			\right)^{\frac{1}{2}}\right]\notag \\
			&\leq 16\varepsilon_0\sqrt{\mathcal{Q}_3(n+1)}\mathbb{E}\left[\left(\sup_{\tau\leq r\leq t}\|u^{n+1}(r\wedge\tau_{n})-u^{n}(r\wedge\tau_{n})\|\right)\left(\int_{\tau}^{t\wedge\tau_{n}}
			\|u^{n+1}(r)-u^{n}(r)\|^{2}dr\right)^{\frac{1}{2}}\right]\notag \\
			&\leq\frac{1}{2}\mathbb{E}\left[\sup_{\tau\leq r\leq t}\|u^{n+1}(r\wedge\tau_{n})-u^{n}(r\wedge\tau_{n})\|^{2}\right]+\mathcal{Q}_6(n)\int_{\tau}^{t}\mathbb{E}\left[\sup_{\tau\leq r\leq s}\|u^{n+1}(r\wedge\tau_{n})-u^{n}(r\wedge\tau_{n})\|^{2}\right]ds,
	\end{align}
}
and
	{\small
\begin{align}\label{LSWs2.34}
	\begin{split}
		&~~2\varepsilon_0 \mathbb{E}\left[\sup_{\tau\leq r\leq t\wedge\tau_{n}}\left|\sum_{k=1}^\infty \int_{\tau}^{r}\left(v^{n+1}(s)-v^{n}(s)\right)\left(\sigma^{n+1}_{k}\big(v^{n+1}(s)\big)
		-\sigma_{k}^{n+1}\big(v^{n}(s)\big)\right)dW_k(s)\right|\right]\\
		&\leq \frac{1}{2}\mathbb{E}\left[\sup_{\tau\leq r\leq t}\|v^{n+1}(r\wedge\tau_{n})-v^{n}(r\wedge\tau_{n})\|^{2}\right]+\mathcal{Q}_7(n)\int_{\tau}^{t}\mathbb{E}\left[\sup_{\tau\leq r\leq s}\|v^{n+1}(r\wedge\tau_{n})-v^{n}(r\wedge\tau_{n})\|^{2}\right]ds.
	\end{split}
\end{align}
}

Substituting \eqref{LSWs2.33} and \eqref{LSWs2.34} into \eqref{LSWs2.32}, we have that there exists $\mathcal{Q}_8(n)>0$ such that
\begin{align*}
		&~~\mathbb{E}\left[\sup_{\tau\leq r\leq t}\left(\|u^{n+1}(r\wedge\tau_{n})-u^{n}(r\wedge\tau_{n})\|^2
		+\|v^{n+1}(r\wedge\tau_{n})-v^{n}(r\wedge\tau_{n})\|^2\right)\right]\\
		&\leq \mathcal{Q}_8(n)\int_{\tau}^{t}\mathbb{E}\left[\sup_{\tau\leq r\leq s}\left(\|u^{n+1}(r\wedge\tau_{n})-u^{n}(r\wedge\tau_{n})\|^2+\|v^{n+1}(r\wedge\tau_{n})-v^{n}(r\wedge\tau_{n})\|^2\right)\right]ds,
\end{align*}
which implies that
\begin{align}\label{LSWs2.35}
	\begin{split}
		&~~\mathbb{E}\left[\sup_{\tau\leq r\leq t}\left(\|u^{n+1}_{\tau_{n}}(r)-u^{n}_{\tau_{n}}(r)\|^2
		+\|v^{n+1}_{\tau_{n}}(r)-v^{n}_{\tau_{n}}(r)\|^2\right)\right]\\
		&\leq \mathcal{Q}_8(n)\int_{\tau}^{t}\mathbb{E}\left[\sup_{\tau\leq r\leq s}\left(\|u^{n+1}_{\tau_{n}}(r)-u^{n}_{\tau_{n}}(r)\|^2+\|v^{n+1}_{\tau_{n}}(r)-v^{n}_{\tau_{n}}(r)\|^2\right)\right]ds.
	\end{split}
\end{align}
Applying Gronwall's inequality to \eqref{LSWs2.35}, we can obtain
\begin{align*}
	\mathbb{E}\left[\sup_{\tau\leq r\leq t}\left(\|\varphi^{n+1}_{\tau_{n}}(r)-\varphi^{n}_{\tau_{n}}(r)\|^2
	\right)\right]=0,\quad\forall t\geq \tau,
\end{align*}
from which we have $\varphi^{n+1}_{\tau_{n}}(t)=\varphi^{n}_{\tau_{n}}(t)$ for any $t\geq \tau$ almost surely, i.e., $u^{n+1}_{\tau_{n}}(t)=u^{n}_{\tau_{n}}(t)$ and $v^{n+1}_{\tau_{n}}(t)=v^{n}_{\tau_{n}}(t)$ for any $t\geq \tau$ almost surely. Therefore, by \eqref{2.37} we have $\tau_{n+1}\geq\tau_{n}$ almost surely, which shows the desired result \eqref{2.38}.
This proof is finished.
\end{proof}

\begin{remark}
	As a preliminary, we notice that system \eqref{2.appro2} is known to have a unique solution in $L^{2}(\Omega, C([\tau, \infty), \ell^{2}_c\times \ell^2))$. The estimates provided in Lemma \ref{lemma2.3} subsequently ensure the existence and uniqueness of this solution in the more regular space $L^{4}(\Omega, C([\tau, \infty), \ell^{2}_c))\times L^{2}(\Omega, C([\tau, \infty), \ell^{2}))$. This foundation allows us to establish the uniqueness of the solution for system \ref{2.002} within this high-order setting.
\end{remark}

Next, we first establish uniform estimates for solutions $\varphi^n$ of the approximate system \eqref{2.appro2}. Based on these estimates, we then prove that the associated stopping times $\tau_n$ tend to infinity as $n\rightarrow \infty$.
\begin{lemma}\label{lemma2.3}
	Suppose that $(H_1)-(H_3)$ hold, and let $\varphi^n=(u^n(t),v^n(t))^\mathrm{T}$ be the solution of system \eqref{2.appro2}.  Then, for each $T>0$, $\varphi^{n}$ satisfies the following estimates:
	\begin{align}\label{LSWs2.36}
		\mathbb{E}\left[\sup_{\tau\leq r\leq \tau+T}\left(\|u^{n}(r)\|^{4}+\|u^{n}(r)\|^{2}\right)\right]\leq \rho_0,
	\end{align}
	where 
	\begin{align}\label{LSWs2.37}
		\begin{split}
			\rho_0=&\textbf{c}e^{\textbf{c}(1+T)^2}\bigg(\mathbb{ E}\left[\|u_\tau\|^4+\|v_\tau\|^2\right]\\
			&+(1+T)\int_{\tau}^{t} \mathbb{ E}\left[\|f(r)\|^4+\|g(r)\|^4+\|b(r)\|^{4}+ \|\gamma(r)\|^{4}\right]dr+(1+T)^2\bigg),
		\end{split}
	\end{align}
	with $\textbf{c}>0$ being a constant independent of $\tau$, $\varphi_{\tau}$, $n$ and $T$. In particular, let $\tau_{n}$ be the stopping time given by \eqref{2.37} and $\tau$ be defined by
	$
		\tau=\lim\limits_{n\rightarrow\infty}\tau_{n}=\sup\limits_{n\in \mathbb{N}}\tau_{n},
	$
	then we have
	\begin{align}\label{LSWs2.38}
		\tau=\infty,\quad\textrm{a.s.}
	\end{align}
\end{lemma}
\begin{proof}
	{\bf Step 1.} We first prove \eqref{LSWs2.36}. 
	Applying Ito's formula for the process $\|u(t)\|^p$, by \eqref{2.appro2} and taking the real part we obtain that a.s.
	\begin{align}\label{LSWs2.39}
		\begin{split}
			&\|u^{n}(t)\|^p+\alpha p\int_\tau^t \|u^{n}(r)\|^pdr=\|u_\tau\|^p+p\int_{\tau}^{t}\|u^{n}(r)\|^{p-2}\cdot\left(f(r),u^n(r)\right)dr\\
			&+\underbrace{p\int_{\tau}^{t}\|u^{n}(r)\|^{p-2}\cdot\textbf{Im}\left(F^n\left(u^{n}(r),v^{n}(r)\right),u^{n}(r)\right)dr}_{=0}\\
			&+p\varepsilon \sum_{k=1}^\infty \int_{\tau}^{t} \|u^{n}(r)\|^{p-2}\textbf{Im}\left(u^{n}(r),\left(h_{k}^{n}\left(u^{n}(r)\right)+b_k(r)\right)dW_k(r)\right)\\
			&+\frac{p}{2}\varepsilon^2 \int_{\tau}^{t}\|u^{n}(r)\|^{p-2}\sum_{k=1}^\infty\|h_{k}^{n}\left(u^{n}(r)\right)+b_k(r)\|^{2}dr\\
			&+\frac{p(p-2)}{2}\varepsilon^2 \int_{\tau}^{t}\|u^{n}(r)\|^{p-4}\sum_{k=1}^\infty\left(u^{n}(r),h_{k}^{n}\left(u^{n}(r)\right)+b_k(r)\right)^{2}dr.
		\end{split}
	\end{align}
	Using Ito's formula for the process $\|v(t)\|^2$, by \eqref{2.appro2} we get that a.s.
	\begin{align}\label{LSWs2.40}
		\begin{split}
			&\|v^{n}(t)\|^2
			+2\beta \int_{\tau}^{t}\|v^{n}(r)\|^2dr=\|v_\tau\|^2+2\int_{\tau}^{t}\left(g(r),v^n(r)\right)dr\\
			&-2\int_{\tau}^{t}\left(G^n\left(u^{n}(r)\right),v^{n}(r)\right)dr+\varepsilon^2\sum_{k=1}^\infty \int_{\tau}^{t}\|\sigma_{k}^{n}\big(v^{n}(r)\big)+\gamma_{k}(r)\|^{2}dr\\
			&+2\varepsilon \sum_{k=1}^\infty \int_{\tau}^{t}\left(v^{n}(r),
			\left(\sigma_{k}^{n}\big(v^{n}(r)\big)+\gamma_{k}(r)\right)dW_k(r)\right).
		\end{split}
	\end{align}
	Combining with \eqref{LSWs2.39} and \eqref{LSWs2.40} yields
	\begin{align}\label{LSWs2.41}
			&\|u^{n}(t)\|^4+\|v^{n}(t)\|^2+4\alpha \int_\tau^t \|u^{n}(r)\|^4dr
			+2\beta \int_{\tau}^{t}\|v^{n}(r)\|^2dr\notag \\
			&\leq \|u_\tau\|^4+\|v_\tau\|^2+4\int_{\tau}^{t}\|u^{n}(r)\|^{2}\left|\left(f(r),u^n(r)\right)\right|dr+2\int_{\tau}^{t}\left|\left(g(r),v^n(r)\right)\right|dr\notag \\
			&+2\int_{\tau}^{t}\left|\left(G^n\left(u^{n}(r)\right),v^{n}(r)\right)\right|dr+6\varepsilon^2 \int_{\tau}^{t}\|u^{n}(r)\|^{2}\sum_{k=1}^\infty\left(\|h_{k}^{n}\left(u^{n}(r)\right)+b_k(r)\|^{2}\right)dr\notag \\
			&+\varepsilon^2\sum_{k=1}^\infty \int_{\tau}^{t}\|\sigma_{k}^{n}\big(v^{n}(r)\big)+\gamma_{k}(r)\|^{2}dr+2\varepsilon \left|\sum_{k=1}^\infty \int_{\tau}^{t}\left(v^{n}(r),
			\left(\sigma_{k}^{n}\big(v^{n}(r)\big)+\gamma_{k}(r)\right)dW_k(r)\right)\right|\notag \\
			&+4\varepsilon \left|\sum_{k=1}^\infty \int_{\tau}^{t} \|u^{n}(r)\|^{2}\left(u^{n}(r),\left(h_{k}^{n}\left(u^{n}(r)\right)+b_k(r)\right)dW_k(r)\right)\right|.
	\end{align}
	
	We deal with the third to seventh terms on the  right-hand side of \eqref{LSWs2.41}, by H\"{o}lder's inequality and Young's inequality we can infer that a.s.
	\begin{align}\label{LSWs2.42}
		\begin{split}
			4\int_{\tau}^{t}\|u^{n}(r)\|^{2}\left|\left(f(r),u^n(r)\right)\right|&dr \leq \alpha\int_{\tau}^{t}\|u^{n}(r)\|^{4}dr+ \textbf{c}\int_\tau^t \|f(r)\|^4dr,\\
			2\int_{\tau}^{t}\left|\left(g(r),v^n(r)\right)\right|dr&\leq \frac{\beta}{2}\int_{\tau}^{t}\|v^{n}(r)\|^{2}dr+\textbf{c} \int_\tau^t \|g(r)\|^2dr,\\
			2\int_{\tau}^{t}\left|\left(G^n\left(u^{n}(r)\right),v^{n}(r)\right)\right|&dr\leq \textbf{c}\int_{\tau}^{t}\|u^{n}(r)\|^{4}dr+\frac{\beta}{2}\int_{\tau}^{t}\|v^{n}(r)\|^{2}dr,
		\end{split}
	\end{align}
	and by \eqref{TeFGkoko} we have
	\begin{align}\label{LSWs2.43}
		\begin{split}
			&6\varepsilon^2 \int_{\tau}^{t}\|u^{n}(r)\|^{2}\sum_{k=1}^\infty\left(\|h_{k}^{n}\left(u^{n}(r)\right)+b_k(r)\|^{2}\right)dr+2\varepsilon^2\sum_{k=1}^\infty \int_{\tau}^{t}\|\sigma_{k}^{n}\big(v^{n}(r)\big)+\gamma_{k}(r)\|^{2}dr\\
			\leq\,& 12\varepsilon_0^2\int_{\tau}^{t}\sum_{k=1}^\infty\left(\|u^{n}(r)\|^{2}\|b_k(r)\|^{2}+\|\gamma_{k}(r)\|^{2}\right)dr\\
			&+2\varepsilon_0^2 \int_{\tau}^{t}\sum_{k=1}^\infty\left(\|u^{n}(r)\|^{2}\|h_{k}^{n}\left(u^{n}(r)\right)\|^2+\|\sigma_{k}^{n}\big(v^{n}(r)\big)\|^{2}\right)dr\\
			\leq\,& \alpha\int_{\tau}^{t}\|u^{n}(r)\|^{4}dr+\textbf{c}\int_{\tau}^{t} \left(\sum_{k=1}^\infty \|b_k(r)\|^{2}\right)^2dr+ 12\varepsilon_0^2\int_{\tau}^{t}\sum_{k=1}^\infty \|\gamma_{k}(r)\|^{2}dr\\
			&+6\varepsilon_0^2 \|\delta\|^2(t-\tau)+6\varepsilon_0^2 \|\delta\|^2\int_{\tau}^{t}\|u^{n}(r)\|^{4}dr+4\varepsilon_0^2 \|\delta\|^2\int_{\tau}^{t}\|v^{n}(r)\|^{2}dr\\
			\leq\,& \alpha\int_{\tau}^{t}\|u^{n}(r)\|^{4}dr+\textbf{c}\int_{\tau}^{t} \left(\sum_{k=1}^\infty \|b_k(r)\|^{2}\right)^2+ \left(\sum_{k=1}^\infty \|\gamma_{k}(r)\|^{2}\right)^2dr\\
			&+\textbf{c}(t-\tau)+\textbf{c}\int_{\tau}^{t}\|u^{n}(r)\|^{4}dr+\textbf{c}\int_{\tau}^{t}\|v^{n}(r)\|^{2}dr.
		\end{split}
	\end{align}

    According to \eqref{LSWs2.41}-\eqref{LSWs2.43}, we conclude that for all $t\in [\tau,\tau+T]$,
    \begin{align}\label{LSWs2.44}
    	\begin{split}
    		&\mathbb{ E}\left[\sup_{\tau \le r\le t} \left(\|u^{n}(r)\|^4+\|v^{n}(r)\|^2\right)\right]\\
    		&\leq \mathbb{ E}\left[\|u_\tau\|^4+\|v_\tau\|^2\right]+\textbf{c}\int_{\tau}^{t} \left[\|f(r)\|^4+\|g(r)\|^4+\|b(r)\|^{4}+ \|\gamma(r)\|^{4}\right]dr\\
    		&+\textbf{c}(t-\tau)+\textbf{c}\int_{\tau}^{t}\mathbb{ E}\left[\sup_{\tau \le r\le s}\left(\|u^{n}(r)\|^{4}+\|v^{n}(r)\|^{2}\right)\right]ds\\
    		&+2\varepsilon \mathbb{ E}\left[\sup_{\tau \le r\le t}\left|\sum_{k=1}^\infty \int_{\tau}^{r}\left(v^{n}(s),
    		\left(\sigma_{k}^{n}\big(v^{n}(s)\big)+\gamma_{k}(s)\right)dW_k(s)\right)\right|\right]\\
    		&+4\varepsilon \mathbb{ E}\left[\sup_{\tau \le r\le t}\left|\sum_{k=1}^\infty \int_{\tau}^{r} \|u^{n}(s)\|^{2}\left(u^{n}(s),\left(h_{k}^{n}\left(u^{n}(s)\right)+b_k(s)\right)dW_k(s)\right)\right|\right].
    	\end{split}
    \end{align}
    For the last two terms on the right-hand side of \eqref{LSWs2.44}, by the BDG inequality and \eqref{TeFGkoko}, we can get
    \begin{align}\label{LSWs2.45}
    		&2\varepsilon \mathbb{ E}\left[\sup_{\tau \le r\le t}\left|\sum_{k=1}^\infty \int_{\tau}^{r}\left(v^{n}(s),
    		\left(\sigma_{k}^{n}\big(v^{n}(s)\big)+\gamma_{k}(s)\right)dW_k(s)\right)\right|\right]\notag \\
    		\leq &2\varepsilon_0 \mathbb{ E}\left[\left( \int_{\tau}^{t}\sum_{k=1}^\infty\|v^{n}(r)\|^2\|\sigma_{k}^{n}\big(v^{n}(r)\big)+\gamma_{k}(r)\|^2dr\right)^{\frac{1}{2}}\right] \notag \\
    		\leq &\frac{1}{2}\int_{\tau}^{t} \mathbb{E}\left[\sup_{\tau \le r\le s}\|v^{n}(r)\|^2\right]ds+\textbf{c}\mathbb{ E}\left[\sum_{k=1}^\infty \int_{\tau}^{t}\|\sigma_{k}^{n}\big(v^{n}(r)\big)+\gamma_{k}(r)\|^{2}dr\right]\notag \\
    		\leq &\frac{1}{2}\int_{\tau}^{t} \mathbb{E}\left[\sup_{\tau \le r\le s}\|v^{n}(r)\|^2\right]ds+\textbf{c}\int_{\tau}^{t} \mathbb{ E}\left[\|\gamma(r)\|^{4}\right]dr+\textbf{c}(t-\tau)+\textbf{c}\int_{\tau}^{t}\mathbb{ E}\left[\|v^{n}(r)\|^{2}\right]dr,
    \end{align}
    and
    \begin{align}\label{LSWs2.46}
    		&4\varepsilon \mathbb{ E}\left[\sup_{\tau \le r\le t}\left|\sum_{k=1}^\infty \int_{\tau}^{s} \|u^{n}(s)\|^{2}\left(u^{n}(s),\left(h_{k}^{n}\left(u^{n}(s)\right)+b_k(s)\right)dW_k(r)\right)\right|\right]\notag \\
    		\leq &4\varepsilon_0 \mathbb{ E}\left[\left( \int_{\tau}^{t}\sum_{k=1}^\infty\|u^{n}(r)\|^6\|h_{k}^{n}\big(v^{n}(r)\big)+b_{k}(r)\|^2dr\right)^{\frac{1}{2}}\right] \notag \\
    		\leq &4\varepsilon_0 \mathbb{ E}\left[\left( \sup_{\tau \le r\le s}\|u^{n}(r)\|^3\right)\left( \int_{\tau}^{t}\sum_{k=1}^\infty\|h_{k}^{n}\big(v^{n}(s)\big)+b_{k}(s)\|^2ds\right)^{\frac{1}{2}}\right] \notag \\
    		\leq &4\varepsilon_0 \mathbb{ E}\left[\left( \sup_{\tau \le r\le s}\|u^{n}(r)\|^3\right)\left( \int_{\tau}^{t}\sum_{k=1}^\infty\|h_{k}^{n}\big(v^{n}(s)\big)+b_{k}(s)\|^2ds\right)^{\frac{1}{2}}\right] \notag \\
    		\leq &\frac{1}{2}\int_{\tau}^{t}\mathbb{ E}\left[\sup_{\tau \le r\le s}\|u^{n}(r)\|^4\right]dr+\textbf{c}(t-\tau)\mathbb{E}\left[ \left(\int_{\tau}^{t}\sum_{k=1}^\infty\|h_{k}^{n}\big(u^{n}(r)\big)+b_{k}(r)\|^{2}dr\right)^2\right]\notag \\
    		\leq &\frac{1}{2}\int_{\tau}^{t}\mathbb{ E}\left[\sup_{\tau \le r\le s}\|u^{n}(r)\|^4\right]dr+\textbf{c}(t-\tau)\mathbb{E}\left[\int_{\tau}^{t}\|u^{n}(r)\|^{4}dr\right]\notag \\
    		&+\textbf{c}(t-\tau)^2+\textbf{c}(t-\tau)\int_{\tau}^{t}\mathbb{E}\left[\|b(r)\|^{4}\right]dr.
    \end{align}
    
    Note that the above all $\textbf{c}$ is independent of $\tau$, $\varphi_{\tau}$, $n$ and $T$. Thus, by \eqref{LSWs2.44}-\eqref{LSWs2.46} we obtain 
    for all $t\in [\tau,\tau+T]$,
    \begin{align}\label{LSWs2.47}
    	\begin{split}
    		&~~\mathbb{ E}\left[\left(\|u(r)\|^4+\|u(r)\|^2+\|v(r)\|^2\right)\right]\\
    		&\leq \mathbb{ E}\left[\|u_\tau\|^4+\|u_\tau\|^2+\|v_\tau\|^2\right]+\textbf{c}\int_{\tau}^{t} \mathbb{ E}\left[\|f(r)\|^4+\|g(r)\|^4+\|b(r)\|^{4}+ \|\gamma(r)\|^{4}\right]dr\\
    		&+\textbf{c}+\textbf{c}\int_{\tau}^{t}\mathbb{ E}\left[\left(\|u(r)\|^{4}+\|u(r)\|^{2}+\|v(r)\|^{2}\right)\right]ds.
    	\end{split}
    \end{align}
    Applying Gronwall's inequality to \eqref{LSWs2.47}, we have
    \begin{align}\label{LSWs2.48}
    	\mathbb{ E}\left[\sup_{\tau \le r\le t} \left(\|u^{n}(r)\|^4+\|v^{n}(r)\|^2\right)\right]\leq 
    	\widetilde{\textbf{c}}e^{\textbf{c}(1+T)^2},
    \end{align}
    where 
    \begin{align*}
    	\widetilde{\textbf{c}}=2\mathbb{ E}\left[\|u_\tau\|^4+\|v_\tau\|^2\right]+\textbf{c}(1+T)\int_{\tau}^{t} \mathbb{ E}\left[\|f(r)\|^4+\|g(r)\|^4+\|b(r)\|^{4}+ \|\gamma(r)\|^{4}\right]dr+\textbf{c}(1+T)^2.
    \end{align*}
    It follows from \eqref{LSWs2.48} that \eqref{LSWs2.36} holds. 
    
    {\bf Step 2.} We now prove \eqref{LSWs2.38}. For an arbitrary $T\in \mathbb{N}$, by \eqref{2.37} we have
    $$
    \{\tau_{n}<\tau+T\}\subseteq\left\{\sup_{\tau\leq r\leq \tau+T}\left(\|u_{n}(r)\|+\|v_{n}(r)\|\right)\geq n\right\}.
    $$
    Then, by Chebychev's inequality, Young's inequality and \eqref{LSWs2.36} we get
    \begin{align}\label{LSWs2.49}
    	\begin{split}
    		\mathbb{P}\{\tau_{n}<\tau+T\}\leq&\,\mathbb{P}\left\{\sup_{\tau\leq r\leq \tau+T}\left(\|u_{n}(r)\|+\|v_{n}(r)\|\right)\geq n\right\}\\
    		\leq&\,\frac{2}{n^{2}}\mathbb{E}\left[\sup_{\tau\leq s\leq \tau+T}\left(\|u_{n}(s)\|^{4}+\|v_{n}(s)\|^{2}\right)\right]+\frac{1}{2n^{2}}\\
    		\leq&\,\frac{4\rho_0+1}{2n^{2}},
    	\end{split}
    \end{align}
    where $\rho_0$, independent of $n$, is the same number as in \eqref{LSWs2.36}. From \eqref{LSWs2.49}, we obtain
    \begin{align}\label{LSWs2.50}
    	\sum_{n=1}^{\infty}\mathbb{P}\{\tau_{n}<\tau+T\}\leq \frac{4\rho_0+1}{2}\sum_{n=1}^{\infty}\frac{1}{n^{2}}<\infty.
    \end{align}
    Setting $\Omega_{T}=\bigcap\limits_{l=1}^{\infty}\bigcap\limits_{n=l}^{\infty}\{\tau_{n}<\tau+T\}$.
    Combining \eqref{LSWs2.50} with the Borel-Cantelli lemma yields
    $$
    \mathbb{P}(\Omega_{T})=\mathbb{P}\left(\bigcap\limits_{l=1}^{\infty}\bigcap\limits_{n=l}^{\infty}\{\tau_{n}<\tau+T\}\right)=0,
    $$
    which shows that there exists a subset $\Omega_{T}$ of $\Omega$ such that $\mathbb{P}(\Omega_{T})=0$ and
    for each $\omega\in \Omega\backslash\Omega_{T}$, there exists $n_{0}=n_{0}(\omega)>0$ such that
    $\tau_{n}(\omega)\geq\tau+T$ for all $n\geq n_{0}$. Thus, we get 
    $$\tau(\omega)\geq\tau+T, ~~~ \forall \omega\in \Omega\backslash\Omega_{T}.
    $$
    
    Let $\Omega_{0}=\bigcup\limits_{T=1}^{\infty}\Omega_{T}$. Then  $\mathbb{P}(\Omega_{0})=0$, and  for all $\omega\in \Omega\backslash\Omega_{0}$ and all $T\in \mathbb{N}$, we have $\tau(\omega)\geq\tau+T$. Hence, we deduce that $\tau(\omega)=\infty$
    for all $\omega\in \Omega\backslash\Omega_{0}$, from which we get that \eqref{LSWs2.38} holds. This completes the proof.
\end{proof}

In what follows, we address the existence and uniqueness of solutions to system \eqref{2.002}.
\begin{theorem}\label{th2.7}
	Suppose that $(H_1)-(H_3)$ hold, and let
	$\varphi_{\tau}\in L^{4}(\Omega,\ell_c^{2})\times L^{2}(\Omega,\ell^{2})$ be $\mathscr{F}_\tau$-measurable for every $\tau\in \mathbb{R}$. Then, for each $T>0$, the non-autonomous stochastic system \eqref{2.002} possesses a unique solution $\varphi(t,\tau,\varphi_{\tau})=\left(u(t,\tau,u_\tau),v(t,\tau,v_\tau)\right)^\mathrm{T}$
	in the sense of Definition \ref{Definition2.1}. Moreover, the solution $\varphi$ satisfies
	\begin{align}\label{LSWs251}
	&\mathbb{E}\left[\|u\|^{4}_{C([\tau,\tau+T],\ell_c^{2})}+\|v\|^{2}_{C([\tau,\tau+T],\ell^{2})}\right]\\
	\leq& \rho_1 e^{\rho_1T}\left(T+\mathbb{E}\left[\|u_{\tau}\|^{4}+\|v_{\tau}\|^{2}\right]+\int_{\tau}^{\tau+T}\mathbb{E}\Big(\|g(t)\|^{4}+\|f(t)\|^4+\|b(t)\|^{4}+\|\gamma(t)\|^4\Big)dt\right),\notag
	\end{align}
	where $\rho_1>0$ is a constant independent of $u_{\tau}$, $\tau$ and $T$.
\end{theorem}
\begin{proof}
	By Lemmas \ref{lemma2.2} and \ref{lemma2.3}, we find that there exists a subset $\Omega_{1}$ of $\Omega$ with $\mathbb{P}(\Omega_{1}) = 0$ such that, for all $n \in \mathbb{N}$, $\omega \in \Omega \setminus \Omega_{1}$, and $n\geq n_0$, it holds
	$\tau(\omega)=\lim\limits_{n\rightarrow\infty}\tau_{n}(\omega)=\infty$ and  $\varphi^{n+1}_{\tau_{n}}(t,\omega)=\varphi^{n}_{\tau_{n}}(t,\omega)$. It follows that for each $\omega\in \Omega_{1}$ and $t\geq\tau$, there exists $n_{0}=n_{0}(t,\omega)\geq 1$ such that
	\begin{align}\label{LSWs2.52}
		\tau_{n}(\omega)>t~~~\textrm{and}~~~\varphi^{n}(t,\omega)=\varphi^{n_{0}}(t,\omega),\quad\forall\,n\geq n_{0}.
	\end{align}
	Define a mapping $\varphi:\mathbb{R}^\tau\times\Omega\rightarrow\ell^{2}_c\times \ell^{2}$ by
	\begin{align}\label{LSWs2.53}
		\varphi(t,\omega)=
		\left\{\begin{array}{ll}
			\varphi^{n}(t,\omega),\quad \omega\in\Omega_{1},t\in[\tau,\tau_{n}(\omega)],\\[4pt]
			\varphi_{\tau}(\omega),\quad\omega\in\Omega\backslash\Omega_{1},t\in\mathbb{R}^\tau.
		\end{array} \right.
	\end{align}
Then, it is easy to check that the mapping $\varphi$ defined in \eqref{LSWs2.53} is well-defined. Moreover, since $\varphi^{n}$ is a continuous $\ell^{2}_c\times \ell^{2}$-valued process, it follows from \eqref{LSWs2.53} that $\varphi$ is also an almost surely continuous $\ell^{2}_c\times \ell^{2}$-valued process in time $t$. By \eqref{LSWs2.53}, we obtain that for every fixed $t\geq \tau$,
	\begin{align}\label{LSWs2.54}
		\lim_{n\rightarrow\infty}\varphi^{n}(t,\omega)=\varphi(t,\omega),\quad\forall\,\omega\in \Omega_{1}.
	\end{align}
	
	Observe that the $\mathcal{F}_t$-adaptedness of $\varphi$ inherits directly from that of $\varphi^n$ through \eqref{LSWs2.54}. Furthermore, applying Fatou's lemma in conjunction with Lemma \ref{lemma2.2} and \eqref{LSWs2.54}, we obtain the for any $T>0$,
	\begin{align}\label{LSWs2.55}
		\mathbb{E}\left[\|u\|^{4}_{C([\tau,\tau+T],\ell_c^{2})}+\|v\|^{2}_{C([\tau,\tau+T],\ell^{2})}\right]\leq \rho_0,
	\end{align}
	where $\rho_0$ is from Lemma \ref{lemma2.3}. 
	
	Next, we prove that $\varphi=(u,v)^\mathrm{T}$ is a solution of system \eqref{2.002}. Recall that $\varphi^n=(u^n,v^n)^\mathrm{T}$ is the solution of system \eqref{2.appro2} and hence it satisfies
	\begin{align}\label{LSWs2.56}
		\begin{split}
			&\varphi^n(t\wedge\tau_{n})=\varphi_\tau+\int_\tau^{t\wedge\tau_{n}}\left(-\mathfrak{A}\varphi^n(r)+\mathscr{G}^n(\varphi^n(r),r)\right)dr+ \varepsilon\sum\limits_{k=1}^{\infty}\int_\tau^{t\wedge\tau_{n}}\mathscr{H}^n_k(\varphi^n(r),r)dW_k(r).
		\end{split}
	\end{align}
	By \eqref{LSWs2.53} we know that $\varphi^n(t\wedge\tau_{n})=\varphi(t\wedge\tau_{n})$ almost surely. Thus, we can obtain from the definitions of $\mathscr{G}^n$ and $\mathscr{H}^n_k$ that a.s.
	\begin{align*}
		\mathscr{G}^n(\varphi^n(r),r)=\mathscr{G}(\varphi(r),r)\text{~~~and~~~} \mathscr{H}_k^n(\varphi^n(r),r)=\mathscr{H}_k(\varphi(r),r) \text{~~~for } r\in [\tau,\tau_{n}],
	\end{align*}
     which, together with \eqref{LSWs2.56}, can deduce that a.s.
     	\begin{align}\label{LSWs2.57}
     	\begin{split}
     		&\varphi(t\wedge\tau_{n})=\varphi_\tau+\int_\tau^{t\wedge\tau_{n}}\left(-\mathfrak{A}\varphi(r)+\mathscr{G}(\varphi(r),r)\right)dr+ \varepsilon\sum\limits_{k=1}^{\infty}\int_\tau^{t\wedge\tau_{n}}\mathscr{H}_k(\varphi(r),r)dW_k(r).
     	\end{split}
     \end{align}
    Thanks to $\lim\limits_{n\rightarrow\infty}\tau_{n}=\infty$ almost surely, it follows from \eqref{LSWs2.57} that for all $t\geq \tau$,
     \begin{align*}
     	\begin{split}
     		\varphi(t)
     		=\varphi_\tau+\int_\tau^{t}\left(-\mathfrak{A}\varphi(r)+\mathscr{G}(\varphi(r),r)\right)dr+ \varepsilon\sum\limits_{k=1}^{\infty}\int_\tau^{t}\mathscr{H}_k(\varphi(r),r)dW_k(r)~~ \mathbb{P}\text{-a.s.},
     	\end{split}
     \end{align*}
    which, along with \eqref{LSWs2.57}, implies that $\varphi$ is a solution of system \eqref{2.002} in the sense of Definition \ref{Definition2.1}.
    
	Finally, we turn to the uniqueness of solutions of system \eqref{2.002}. Let $\varphi^1=(u^1,v^1)^\mathrm{T}$ and $\varphi^2=(u^2,v^2)^\mathrm{T}$ be two solutions of system \eqref{2.002} in the sense of Definition \ref{Definition2.1}. We are committed to proving the following assertion, i.e., for any $T>0$, it holds
	\begin{align}\label{LSWs2.58}
		\begin{split}
			&\mathbb{P}\left(\|\varphi^1(t)-\varphi^2(t)\|_{\ell_c^2\times \ell^2}=0\text{~~for all~~} t\in [\tau,\tau+T]\right)\\
			=&\mathbb{P}\left(\|u^1(t)-u^2(t)\|+\|v^1(t)-v^2(t)\|=0\text{~~for all~~} t\in [\tau,\tau+T]\right)=0.
		\end{split}
	\end{align}
	 For each $n\in \mathbb{N}$, $\tau\in \mathbb{R}$ and $T>0$, we define the following stopping time:
	\begin{align}\label{2.78}
		T_{n}=(\tau+T)\wedge\inf\{t\geq\tau:\|\varphi^1(t)\|\geq n~~\textrm{or}~~\|\varphi^2(t)\|\geq n\}.
	\end{align}
	
	By \eqref{2.002} we obtain that a.s.,
	\begin{align}\label{LSWs2.60}
		\begin{split}
			&u^{1}(t\wedge T_{n})-u^2(t\wedge T_{n})
			+\alpha \int_{\tau}^{t\wedge T_{n}}\left(u^{1}(r)-u^2(r)\right)dr+i\int_{\tau}^{t\wedge T_{n}}\left(Au^{1}(r)-Au^2(r)\right)dr\\
			&+i\int_{\tau}^{t\wedge T_{n}}\left(F\left({u^{1}(r),v^{1}(r)}\right)-F\left(u^2(r),v^2(r)\right)\right)dr\\
			&=u^{1}(\tau)-u^2(\tau)-i\varepsilon \sum_{k=1}^\infty \int_{\tau}^{t\wedge T_{n}}\left(h_{k}\left(u^{1}(r)\right)
			-h_{k}\left(u^2(r)\right)\right)dW_k(r),
		\end{split}
	\end{align}
	and
	\begin{align}\label{LSWs2.61}
			&~v^{1}(t\wedge T_{n})-v^2(t\wedge T_{n})
			+\beta \int_{\tau}^{t\wedge T_{n}}\left(v^{1}(r)-v^2(r)\right)dr+\int_{\tau}^{t\wedge T_{n}}\left(G\left(u^{1}(r)\right)-G\left(u^2(r)\right)\right)dr\notag \\ 
			&=v^{1}(\tau)-v^2(\tau)+\varepsilon \sum_{k=1}^\infty \int_{\tau}^{t\wedge T_{n}}\left(\sigma_{k}\big(v^{1}(r)\big)
			-\sigma_{k}\big(v^2(r)\right)dW_k(r).
	\end{align}
	Using \eqref{LSWs2.60} and applying Ito's formula to the process $\|u^{1}(t\wedge T_{n})-u^2(t\wedge T_{n})\|^2+\|u^{1}(t\wedge T_{n})-u^2(t\wedge T_{n})\|^4$, by taking the real part we obtain 
	 \begin{align}\label{LSWs2.62}
	 		&~\|u^{1}(t\wedge T_{n})-u^2(t\wedge T_{n})\|^2+\|u^{1}(t\wedge T_{n})-u^2(t\wedge T_{n})\|^4\notag \\
	 		&+2\alpha \int_{\tau}^{t\wedge T_{n}}\|u^{1}(r)-u^{2}(r)\|^2dr+4\alpha \int_{\tau}^{t\wedge T_{n}}\|u^{1}(r)-u^{2}(r)\|^4dr\notag \\
	 		&\leq  \|u^{1}(\tau)-u^2(\tau)\|^2+\|u^{1}(\tau)-u^2(\tau)\|^4+\varepsilon^2\sum_{k=1}^\infty \int_{\tau}^{t\wedge T_{n}}\|h_{k}\left(u^{1}(r)\right)
	 		-h_{k}\left(u^{2}(r)\right)\|^{2}dr\notag \\
	 		&+6\varepsilon^2\sum_{k=1}^\infty \int_{\tau}^{t\wedge T_{n}}\|u^{1}(r)-u^2(r)\|^2\|h_{k}\left(u^{1}(r)\right)
	 		-h_{k}\left(u^{2}(r)\right)\|^{2}dr\notag \\
	 		&+2\int_{\tau}^{t\wedge T_{n}}\left|\left(F\left({u^{1}(r),v^{1}(r)}\right)-F\left(u^{2}(r),v^{2}(r)\right),u^{1}(r)-u^{2}(r)\right)\right|dr\notag \\
	 		&+4\int_{\tau}^{t\wedge T_{n}}\|u^{1}(r)-u^2(r)\|^2\left|\left(F\left({u^{1}(r),v^{1}(r)}\right)-F\left(u^{2}(r),v^{2}(r)\right),u^{1}(r)-u^{2}(r)\right)\right|dr\notag \\
	 		&+2\varepsilon \sum_{k=1}^\infty \int_{\tau}^{t\wedge T_{n}} \textbf{Im}\left(\left(u^{1}(r)-u^{2}(r)\right)\left(h_k\left(u^{1}(r)\right)
	 		-h_k\left(u^{2}(r)\right)\right)\right)dW_k(r)\notag \\
	 		&+4\varepsilon \sum_{k=1}^\infty \int_{\tau}^{t\wedge T_{n}} \|u^{1}(r)-u^2(r)\|^2\textbf{Im}\left(\left(u^{1}(r)-u^{2}(r)\right)\left(h_k\left(u^{1}(r)\right)
	 		-h_k\left(u^{2}(r)\right)\right)\right)dW_k(r).
	 \end{align}
	 By \eqref{LSWs2.61} and Ito's formula we can infer 
	 \begin{align}\label{LSWs2.63}
	 		&\|v^{1}(t\wedge T_{n})-v^{2}(t\wedge T_{n})\|^2
	 		+2\beta \int_{\tau}^{t\wedge T_{n}}\|v^{1}(r)-v^{2}(r)\|^2dr\notag \\
	 		&+2\int_{\tau}^{t\wedge T_{n}}\left(G\left(u^{1}(r)\right)-G\left(u^{2}(r)\right),v^{1}(r)-v^{2}(r)\right)dr\notag \\
	 		&=\|v^{1}(\tau)-v^2(\tau)\|^2+\varepsilon^2\sum_{k=1}^\infty \int_{\tau}^{t\wedge T_{n}}\|\sigma_{k}\big(v^{1}(r)\big)
	 		-\sigma_{k}\big(v^{2}(r)\big)\|^{2}dr\notag \\
	 		&+2\varepsilon \sum_{k=1}^\infty \int_{\tau}^{t\wedge T_{n}}\left(v^{1}(r)-v^{2}(r)\right)\left(\sigma_{k}\big(v^{1}(r)\big)
	 		-\sigma^{n}\big(v^{2}(r)\big)\right)dW_k(r).
	 \end{align} 
	 By H\"{o}lder's inequality, Young's inequality and \eqref{FGlojo}, it is easy to derive that there exists $\textbf{c}_5=\textbf{c}_5(n)>0$ such that 
	 \begin{align}\label{LSWs2.64}
	 		&~~2\int_{\tau}^{t\wedge T_{n}}\left|\left(F\left({u^{1}(r),v^{1}(r)}\right)-F\left(u^{2}(r),v^{2}(r)\right),u^{1}(r)-u^{2}(r)\right)\right|dr\notag \\
	 		&+4\int_{\tau}^{t\wedge T_{n}}\|u^{1}(r)-u^2(r)\|^2\left|\left(F\left({u^{1}(r),v^{1}(r)}\right)-F\left(u^{2}(r),v^{2}(r)\right),u^{1}(r)-u^{2}(r)\right)\right|dr\notag \\
	 		&\leq \int_{\tau}^{t\wedge T_{n}}\left(\|F\left({u^{1}(r),v^{1}(r)}\right)-F\left(u^{2}(r),v^{2}(r)\right)\|^2+\|u^{1}(r)-u^{2}(r)\|^2\right)dr\notag \\
	 		&+2\int_{\tau}^{t\wedge T_{n}}\|u^{1}(r)-u^2(r)\|^2\left(\|F\left({u^{1}(r),v^{1}(r)}\right)-F\left(u^{2}(r),v^{2}(r)\right)\|^2+\|u^{1}(r)-u^{2}(r)\|^2\right)dr\notag \\
	 		&\leq (\textbf{c}_3(n)+1 )\int_{\tau}^{t\wedge T_{n}}\|u^{1}(r)-u^2(r)\|^2 dr+ \textbf{c}_3(n)\left(4\textbf{c}_5+1\right)\int_{\tau}^{t\wedge T_{n}}\|v^{1}(r)-v^2(r)\|^2dr\notag \\
	 		&+2(\textbf{c}_3(n)+1 )\int_{\tau}^{t\wedge T_{n}}\|u^{1}(r)-u^2(r)\|^4 dr,
	 \end{align}
	 and 
	 \begin{align}\label{LSWs2.65}
	 	\begin{split}
	 		&\left|2\int_{\tau}^{t\wedge T_{n}}\left(G\left(u^{1}(r)\right)-G\left(u^{2}(r)\right),v^{1}(r)-v^{2}(r)\right)dr\right|\\
	 		\leq
	 		\,& \int_{\tau}^{t\wedge T_{n}} \|G\left(u^{1}(r)\right)-G\left(u^{2}(r)\right)\|^2dr
	 		+\int_{\tau}^{t\wedge T_{n}}\|v^{1}(r)-v^{2}(r)\|^2dr\\
	 		\leq
	 		\,& \textbf{c}_4(n)\int_{\tau}^{t\wedge T_{n}}\|u^{1}(r)-u^2(r)\|^2 dr+\int_{\tau}^{t\wedge T_{n}}\|v^{1}(r)-v^2(r)\|^2dr.
	 	\end{split}
	 \end{align}

By \eqref{opformLip}, \eqref{FGlojo}  and \eqref{LSWs2.62}-\eqref{LSWs2.65}, we obtain that there exists $$\textbf{c}_6=\textbf{c}_6(n)=\max\{\textbf{c}_1\varepsilon_0^2+\textbf{c}_3(n)+1+\textbf{c}_4(n),6\textbf{c}_1\varepsilon_0^2+2(\textbf{c}_3(n)+1 ),\textbf{c}_3(n)\left(4\textbf{c}_5+1\right)+ \textbf{c}_2\varepsilon_0^2+1\}
$$ such that the following inequality holds
	 \begin{align*}
	 		&~~\|u^{1}(t\wedge T_{n})-u^2(t\wedge T_{n})\|^2+\|u^{1}(t\wedge T_{n})-u^2(t\wedge T_{n})\|^4+\|v^{1}(t\wedge T_{n})-v^{2}(t\wedge T_{n})\|^2\\
	 		&\leq  \|u^{1}(\tau)-u^2(\tau)\|^2+\|u^{1}(\tau)-u^2(\tau)\|^4+\|v^{1}(\tau)-v^2(\tau)\|^2\\
	 		&+\textbf{c}_6 \int_{\tau}^{t\wedge T_{n}}\left(\|u^{1}(r)-u^2(r)\|^2+\|u^{1}(r)-u^2(r)\|^4+\|v^{1}(r)-v^2(r)\|^2\right)dr\\
	 		&+2\varepsilon \sum_{k=1}^\infty \int_{\tau}^{t\wedge T_{n}} \textbf{Im}\left(\left(u^{1}(r)-u^{2}(r)\right)\left(h_k\left(u^{1}(r)\right)
	 		-h_k\left(u^{2}(r)\right)\right)\right)dW_k(r)\\
	 		&+4\varepsilon \sum_{k=1}^\infty \int_{\tau}^{t\wedge T_{n}} \|u^{1}(r)-u^2(r)\|^2\textbf{Im}\left(\left(u^{1}(r)-u^{2}(r)\right)\left(h_k\left(u^{1}(r)\right)
	 		-h_k\left(u^{2}(r)\right)\right)\right)dW_k(r)\\
	 		&+2\varepsilon \sum_{k=1}^\infty \int_{\tau}^{t\wedge T_{n}}\left(v^{1}(r)-v^{2}(r)\right)\left(\sigma_{k}\big(v^{1}(r)\big)
	 		-\sigma^{n}\big(v^{2}(r)\big)\right)dW_k(r),
	 \end{align*}
	 from which we have
	  \begin{align}\label{LSWs2.66}
	 		&~~\mathbb{ E}\left[\sup_{\tau \le r\le t}\left(\|u_{T_{n}}^{1}(r)-u_{T_{n}}^2(r)\|^2+\|u_{T_{n}}^{1}(r)-u_{T_{n}}^2(r)\|^4+\|v_{T_{n}}^{1}(r)-v_{T_{n}}^{2}(r)\|^2\right)\right]\notag \\
	 		&\leq  \mathbb{ E}\left[\|u^{1}(\tau)-u^2(\tau)\|^2+\|u^{1}(\tau)-u^2(\tau)\|^4+\|v^{1}(\tau)-v^2(\tau)\|^2\right]\notag \\
	 		&+\textbf{c}_6 \int_{\tau}^{t}\mathbb{ E}\left[\sup_{\tau \le r\le s}\left(\|u_{T_{n}}^{1}(r)-u_{T_{n}}^2(r)\|^2+\|u_{T_{n}}^{1}(r)-u_{T_{n}}^2(r)\|^4+\|v_{T_{n}}^{1}(r)-v_{T_{n}}^2(r)\|^2\right)\right]ds\notag \\
	 		&+2\varepsilon \mathbb{ E}\left[\sup_{\tau \le s \le t\wedge T_n}\left|\sum_{k=1}^\infty \int_{\tau}^{s} \textbf{Im}\left(\left(u^{1}(r)-u^{2}(r)\right)\left(h_k\left(u^{1}(r)\right)
	 		-h_k\left(u^{2}(r)\right)\right)\right)dW_k(r)\right|\right]\notag \\
	 		&+4\varepsilon \mathbb{ E}\left[\sup_{\tau \le s \le t\wedge T_n}\left|\sum_{k=1}^\infty \int_{\tau}^{s} \|u^{1}(r)-u^2(r)\|^2\textbf{Im}\left(\left(u^{1}(r)-u^{2}(r)\right)\left(h_k\left(u^{1}(r)\right)
	 		-h_k\left(u^{2}(r)\right)\right)\right)dW_k(r)\right|\right]\notag \\
	 		&+2\varepsilon \mathbb{ E}\left[\sup_{\tau \le s \le t\wedge T_n}\left|\sum_{k=1}^\infty \int_{\tau}^{s}\left(v^{1}(r)-v^{2}(r)\right)\left(\sigma_{k}\big(v^{1}(r)\big)
	 		-\sigma^{n}\big(v^{2}(r)\big)\right)dW_k(r)\right|\right].
	 \end{align}

	 We deal with the last three terms on the right-hand side of \eqref{LSWs2.66}. Using \eqref{opformLip} and the BDG inequality can derive that for all $t\in [\tau,\tau+T]$ with $T>0$, it holds
	 {\small
	 \begin{align*}
	 	&~~2\varepsilon \mathbb{ E}\left[\sup_{\tau \le s \le t\wedge T_n}\left|\sum_{k=1}^\infty \int_{\tau}^{s} \textbf{Im}\left(\left(u^{1}(r)-u^{2}(r)\right)\left(h_k\left(u^{1}(r)\right)
	 	-h_k\left(u^{2}(r)\right)\right)\right)dW_k(r)\right|\right]\\
	 	&+4\varepsilon \mathbb{ E}\left[\sup_{\tau \le s \le t\wedge T_n}\left|\sum_{k=1}^\infty \int_{\tau}^{s} \|u^{1}(r)-u^2(r)\|^2\textbf{Im}\left(\left(u^{1}(r)-u^{2}(r)\right)\left(h_k\left(u^{1}(r)\right)
	 	-h_k\left(u^{2}(r)\right)\right)\right)dW_k(r)\right|\right]\\
	 	&\leq 8\sqrt{2}\varepsilon_0 \mathbb{ E}\left[ \left(\int_{\tau}^{t\wedge T_{n}}\|u^{1}(r)-u^{2}(r)\|^2 \sum_{k=1}^\infty\|h_k\left(u^{1}(r)\right)
	 	-h_k\left(u^{2}(r)\right)\|^2dr\right)^{\frac{1}{2}}\right]\\
	 	&+16\sqrt{2}\varepsilon_0 \mathbb{ E}\left[ \left(\int_{\tau}^{t\wedge T_{n}}\|u^{1}(r)-u^{2}(r)\|^6 \sum_{k=1}^\infty\|h_k\left(u^{1}(r)\right)
	 	-h_k\left(u^{2}(r)\right)\|^2dr\right)^{\frac{1}{2}}\right]\\
	 	&\leq 8\sqrt{2}\varepsilon_0 \mathbb{ E}\left[\sup_{\tau \le r\le t}\|u_{T_{n}}^{1}(r)-u_{T_{n}}^2(r)\|\left(\int_{\tau}^{t\wedge T_{n}} \sum_{k=1}^\infty\|h_k\left(u^{1}(r)\right)
	 	-h_k\left(u^{2}(r)\right)\|^2dr \right)^{\frac{1}{2}}\right]\\
	 	&+16\sqrt{2}\varepsilon_0 \mathbb{ E}\left[ \sup_{\tau \le r\le t}\|u_{T_{n}}^{1}(r)-u_{T_{n}}^2(r)\|^3 \left(\int_{\tau}^{t\wedge T_{n}} \sum_{k=1}^\infty\|h_k\left(u^{1}(r)\right)
	 	-h_k\left(u^{2}(r)\right)\|^2dr\right)^{\frac{1}{2}}\right]\\
	 	&\leq\frac{1}{2}\mathbb{ E}\left[\sup_{\tau \le r\le t}\left(\|u_{T_{n}}^{1}(r)-u_{T_{n}}^2(r)\|^2+\|u_{T_{n}}^{1}(r)-u_{T_{n}}^2(r)\|^4\right)\right]+64\varepsilon_0^2\textbf{c}_1	\int_{\tau}^{t}\mathbb{ E}\left[\sup_{\tau \le r\le s} \|u_{T_{n}}^{1}(r)-u_{T_{n}}^2(r)\|^2  \right]ds \\
	 	&+\left(\frac{2}{3}\right)^{-3}{\left(16\varepsilon_0\right)^4 \textbf{c}_1^2} T \int_{\tau}^{t}\mathbb{ E}\left[\sup_{\tau \le r\le s} \|u_{T_{n}}^{1}(r)-u_{T_{n}}^2(r)\|^4  \right]ds,	
	 \end{align*}
	 }
	 and
	 \begin{align*}
	 	&~~2\varepsilon \mathbb{ E}\left[\sup_{\tau \le s \le t\wedge T_n}\left|\sum_{k=1}^\infty \int_{\tau}^{s}\left(v^{1}(r)-v^{2}(r)\right)\left(\sigma_{k}\big(v^{1}(r)\big)
	 	-\sigma^{n}\big(v^{2}(r)\big)\right)dW_k(r)\right|\right]\\
	 	&\leq 8\sqrt{2}\varepsilon_0 \mathbb{ E}\left[\sup_{\tau \le r\le t}\|v_{T_{n}}^{1}(r)-v_{T_{n}}^2(r)\|\left(\int_{\tau}^{t\wedge T_{n}} \sum_{k=1}^\infty\|\sigma_k\left(v^{1}(r)\right)
	 	-\sigma_k\left(v^{2}(r)\right)\|^2dr \right)^{\frac{1}{2}}\right]\\
	 	&\leq \frac{1}{2}\mathbb{ E}\left[\sup_{\tau \le r\le t}\left(\|v_{T_{n}}^{1}(r)-v_{T_{n}}^2(r)\|^2\right)\right]+64\varepsilon_0^2\textbf{c}_2	\int_{\tau}^{t}\mathbb{ E}\left[\sup_{\tau \le r\le s}\|v_{T_{n}}^{1}(r)-v_{T_{n}}^2(r)\|^2  \right]ds.
	 \end{align*}
	 
	 Therefore, we can obtain that there exists $\textbf{c}_7=\textbf{c}_6+64\varepsilon_0^2(\textbf{c}_1+\textbf{c}_2)+\left(\frac{2}{3}\right)^{-3}{\left(16\varepsilon_0\right)^4 \textbf{c}_1^2}$ such that for all $t\in [\tau,\tau+T]$,
	 \begin{align}\label{LSWs2.67}
	 	\begin{split}
	 		&~~\mathbb{ E}\left[\sup_{\tau \le r\le t}\left(\|u_{T_{n}}^{1}(r)-u_{T_{n}}^2(r)\|^2+\|u_{T_{n}}^{1}(r)-u_{T_{n}}^2(r)\|^4+\|v_{T_{n}}^{1}(r)-v_{T_{n}}^{2}(r)\|^2\right)\right]\\
	 		&\leq  2\mathbb{ E}\left[\|u^{1}(\tau)-u^2(\tau)\|^2+\|u^{1}(\tau)-u^2(\tau)\|^4+\|v^{1}(\tau)-v^2(\tau)\|^2\right]+2\textbf{c}_7(1+T)\\
	 		&\cdot\int_{\tau}^{t}\mathbb{ E}\left[\sup_{\tau \le r\le s}\left(\|u_{T_{n}}^{1}(r)-u_{T_{n}}^2(r)\|^2+\|u_{T_{n}}^{1}(r)-u_{T_{n}}^2(r)\|^4+\|v_{T_{n}}^{1}(r)-v_{T_{n}}^2(r)\|^2\right)\right]ds.
	 	\end{split}
	 \end{align}
	 By applying Gronwall's lemma to \eqref{LSWs2.67}, we have
	 \begin{align}\label{LSWs2.68}
	 	\begin{split}
	 		&~~\mathbb{ E}\left[\sup_{\tau \le r\le {\tau+T}}\left(\|u_{T_{n}}^{1}(r)-u_{T_{n}}^2(r)\|^2+\|u_{T_{n}}^{1}(r)-u_{T_{n}}^2(r)\|^4+\|v_{T_{n}}^{1}(r)-v_{T_{n}}^{2}(r)\|^2\right)\right]\\
	 		&\leq 2e^{2\textbf{c}_7(1+T)T}\mathbb{ E}\left[\|u^{1}(\tau)-u^2(\tau)\|^2+\|u^{1}(\tau)-u^2(\tau)\|^4+\|v^{1}(\tau)-v^2(\tau)\|^2\right].
	 	\end{split}
	 \end{align}
	 
	 Owing to $\varphi^1(\tau)=\varphi^2(\tau)$ in $L^4(\Omega,\ell_c^{2})\times L^2(\Omega,\ell^{2})$,  we can derive from \eqref{LSWs2.68} that
	 \begin{align*}
	 	\mathbb{E}\left[\sup_{\tau \le r\le {\tau+T}}\left(\|u_{T_{n}}^{1}(r)-u_{T_{n}}^2(r)\|^4+\|v_{T_{n}}^{1}(r)-v_{T_{n}}^{2}(r)\|^2\right)\right]=0,
	 \end{align*}
	 from which we have $\|u_{T_{n}}^{1}(r)-u_{T_{n}}^2(r)\|+\|v_{T_{n}}^{1}(r)-v_{T_{n}}^{2}(r)\|=0$ for all $t\in[\tau,\tau+T]$ almost surely. We find $T_{n}=\tau+T$ for large enough $n$ by the continuity of $\varphi^1$ and $\varphi^{2}$ in $t$.
	 
	 Consequently, we conclude that  $\|u^{1}(t)-u^{2}(t)\|+\|v^{1}(t)-v^{2}(t)\|=0$ for all $t\in[\tau,\tau+T]$ almost surely; that is, $\|\varphi^{1}(t)-\varphi^{2}(t)\|=0$ for all $t\in[\tau,\tau+T]$ almost surely. This establishes \eqref{LSWs2.58}. By the arbitrariness of $T$, it follows from \eqref{LSWs2.68} that
	 \begin{align*}
	 	\mathbb{P}(\|\varphi^{1}(t)-\varphi_{2}(t)\|=0\text{~~for all~~}t\geq\tau)=1,
	 \end{align*}
	 which implies the uniqueness of solutions. In particular, simlar to the argument \eqref{LSWs2.36}, we can also verify \eqref{LSWs251}. This proof is finished.
\end{proof}

\section{Existence of weak pullback attractors}
In this section, we focus on studying the existence and uniqueness of weak pullback mean random attractors of system (\ref{2.002}) in $ L^{4}(\Omega,\ell_c^{2})\times L^{2}(\Omega,\ell^{2})$ over $ (\Omega,\mathcal{F},\{\mathcal{F}_t\}_{t\in\mathbb{R}},\mathbb{P}) $.  
To this end, we suppose that
\begin{align}\label{LSWs4.1}
\alpha-\frac{18\lambda^2}{\beta}>0.
\end{align}
%

From now on, we will begin by defining a mean random dynamical system for system (\ref{2.002}), which will serve as the basis for investigating the existence and uniqueness of weak $ \mathcal{D} $-pullback mean random attractors. For every $ \tau \in \mathbb{R} $ and $ t\in \mathbb{R}^{+} $, we define the mapping $ \Phi(t,\tau):\, L^{4}(\Omega,\ell_c^{2})\times L^{2}(\Omega,\ell^{2}) \rightarrow  L^{4}(\Omega,\ell_c^{2})\times L^{2}(\Omega,\ell^{2})$ by 
\begin{equation}\label{3.1}
	\Phi(t,\tau,\varphi_{\tau})=\varphi(t+\tau,\tau,\varphi_{\tau}), \, \,\, \, \forall \varphi_{\tau} \in L^{4}(\Omega,\ell_c^{2})\times L^{2}(\Omega,\ell^{2}),
\end{equation}
where $ \varphi(t+\tau,\tau,\varphi_{\tau}) $ denotes the solution of system (\ref{2.002}) with initial data $ \varphi_{\tau} $.


It is easy to see that $ \Phi(0,\tau) $ is the identity operator on $ L^{4}(\Omega,\ell_c^{2})\times L^{2}(\Omega,\ell^{2})$. Additionally, by the uniqueness of solution of (\ref{2.002}) we find that for any $ \tau \in \mathbb{R} $ and $s, t \in \mathbb{R}^{+}$,
\begin{equation*}
	\Phi(t+s,\tau,\varphi_{\tau})=\Phi(t,s+\tau,\Phi(s,\tau,\varphi_{\tau})),
\end{equation*}
i.e., $\Phi(t+s,\tau)=\Phi(t,s+\tau)\,\circ\, \Phi(s,\tau)$.
Therefore,  $\Phi $ is called a mean random dynamical system associated with (\ref{2.002}) on $ L^{4}(\Omega,\ell_c^{2})\times L^{2}(\Omega,\ell^{2}) $ over $ (\Omega,\mathcal{F},\{\mathcal{F}_t\}_{t\in\mathbb{R}},\mathbb{P}) $ by \cite[Definition $2.9$]{WBX2019}. In order to avoid confusion with symbols, we write  $\Phi(t+s,\tau,\varphi_{\tau})=\Phi(t+s,\tau)\varphi_{\tau}$.

Let $ B=\{ B(\tau) \subseteq L^{4}(\Omega,\ell_c^{2})\times L^{2}(\Omega,\ell^{2}): \tau \in \mathbb{R}\} $ 
denote a family of all nonempty bounded sets of $L^{4}(\Omega,\ell_c^{2})\times L^{2}(\Omega,\ell^{2})$ such that
\begin{equation}\label{4.3}
	\lim\limits_{\tau \to -\infty} e^{\kappa\tau}\|B(\tau)\|^2_{L^{4}(\Omega,\ell_c^{2})\times L^{2}(\Omega,\ell^{2})}=0,
\end{equation}
where $$ \|B(\tau)\|_{L^{4}(\Omega,\ell_c^{2})\times L^{2}(\Omega,\ell^{2})}=\sup_{\varphi\in B(\tau)}\|\varphi\|_{L^{4}(\Omega,\ell_c^{2})\times L^{2}(\Omega,\ell^{2})} .$$ 
We denote by $\mathcal{D}$ the collection of all families of nonempty
bounded sets of $ L^{4}(\Omega,\ell_c^{2})\times L^{2}(\Omega,\ell^{2}) $, it is defined by
\begin{equation}\label{SCLDD4.3}
	\mathcal{D}=\{B= \{B(\tau)\subseteq L^{4}(\Omega,\ell_c^{2})\times L^{2}(\Omega,\ell^{2}):B(\tau) \ne \emptyset \text{ bounded},\tau \in \mathbb{R} \}: B  \text{ satisfies } (\ref{4.3}) \}.
\end{equation}
To obtain the existence of weak $ \mathcal{D} $-pullback mean random attractors of problem (\ref{2.002}), we further suppose 
\begin{align}\label{LSWs4.5}
	\int_{-\infty}^{\tau}e^{\kappa r} \mathbb{ E}\left[\|f(r)\|^4+\|g(r)\|^4+\|b(r)\|^{4}+ \|\gamma(r)\|^{4}\right]dr<\infty, \quad \forall \tau\in\mathbb{R}.
\end{align}

We derive the following uniform estimates of solution to (\ref{2.002}) in $ L^{4}(\Omega,\ell_c^{2})\times L^{2}(\Omega,\ell^{2}) $. 
\begin{lemma}\label{lemma3.1}
	Suppose that $(H_1)-(H_3)$, \eqref{LSWs4.1} and \eqref{LSWs4.5} hold. Then there exists $$
	\varepsilon_0=\max\left\{\sqrt{\frac{\alpha}{24\|\delta\|^2}},\sqrt{\frac{\beta}{48\|\delta\|^2}}\right\}$$ such that for any $\varepsilon\in [0,\varepsilon_0]$,
	 $ \tau \in \mathbb{R} $ and $ B=\{ B(t) \}_{t\in \mathbb{R}}\in \mathcal{D} $, there exists $ T=T(\tau,B)>0 $ such that for all $ t\ge T $, the solution $ \varphi(t)=(u(t),v(t))^\mathrm{T}$ of system (\ref{2.002}) satisfies
\begin{equation}\label{3.001}
	\begin{aligned}
& \mathbb{E}\left[\left\|u\left(\tau, \tau-t, u_{\tau-t}\right)\right\|^4+\left\|v\left(\tau ,\tau-t, v_{\tau-t}\right)\right\|^2\right] \\
&  \leq \mathfrak{M}_0+\mathfrak{M}_0\int_{-\infty}^\tau e^{\kappa (r-\tau)}\mathbb{ E}\left[\|f(t)\|^4+\|g(t)\|^4+\|b(r)\|^{4}+ \|\gamma(t)\|^{4}\right]dr,
	\end{aligned}
\end{equation}
where 
$\varphi_{\tau-t}=(u_{\tau-t},v_{\tau-t})^\mathrm{T} \in B(\tau-t)$, and $\mathfrak{M}_0>0$ is a constant independent of $\tau$ and $B$.
\end{lemma}
\begin{proof}
  Applying Ito's formula to the process $\|u(t)\|^p$, by \eqref{2.002} and taking the real part we obtain that a.s.
	\begin{align}\label{3.002}
		\begin{split}
			&d\|u(t)\|^p+\alpha p \|u(t)\|^pdt=p\|u(t)\|^{p-2}\textbf{Im}\left(f(t),u(t)\right)dt\\
&+p\varepsilon \sum_{k=1}^\infty  \|u(t)\|^{p-2}\textbf{Im}\left(u(t),h_{k}\left(u(t)\right)+b_k(t)\right)dW_k(r)\\
			&+\frac{p}{2}\varepsilon^2\|u(t)\|^{p-2}\sum_{k=1}^\infty\|h_{k}\left(u(t)\right)+b_k(t)\|^{2}dt\\
			&+\frac{p(p-2)}{2}\varepsilon^2 \|u(t)\|^{p-4}\sum_{k=1}^\infty\left(u(t),h_{k}\left(u(t)\right)+b_k(t)\right)^{2}dt.
		\end{split}
	\end{align}
	Using Ito's formula to the process $\|v(t)\|^2$, by \eqref{2.002} we get that a.s.
	\begin{align}\label{3.003}
		\begin{split}
			&\|v(t)\|^2
			+2\beta \|v(t)\|^2dt=2\left(g(t),v(t)\right)dt-2\left(G\left(u(t)\right),v(t)\right)dt\\
			&+\varepsilon^2\sum_{k=1}^\infty\|\sigma_{k}\big(v(t)\big)+\gamma_{k}(t)\|^{2}dt+2\varepsilon \sum_{k=1}^\infty \left(v(t),
			\sigma_{k}\big(v(t)\big)+\gamma_{k}(t)\right)dW_k(t).
		\end{split}
	\end{align}
	Combining with \eqref{3.002} and \eqref{3.003}, and by taking expectation we have
\begin{equation} \label{3.005}
	\begin{aligned}
		&~~\mathbb{ E}\left[\|u(t)\|^4+\|v(t)\|^2\right]+4\alpha \mathbb{ E}\left[\|u(t)\|^4\right]dt+2\beta \mathbb{ E}\left[\|v(t)\|^2\right]dt\\
		&\leq 4\mathbb{E}\left[\|u(t)\|^{2}\left|\left(f(t),u(t)\right)\right|\right]dt+2\mathbb{E}\left[\left|\left(g(t),v(t)\right)\right|\right]dt\\
		&+2\mathbb{E}\left[\left|\left(G\left(u(t)\right),v(t)\right)\right|\right]dt+6\varepsilon^2\mathbb{E}\left[\|u(t)\|^{2}\sum_{k=1}^\infty\left(\|h_{k}\left(u(t)\right)+b_k(t)\|^{2}\right)\right]dt\\
			&+\varepsilon^2\sum_{k=1}^\infty \mathbb{E}\left[\|\sigma_{k}\big(v(t)\big)+\gamma_{k}(t)\|^{2}\right]dt.
	\end{aligned}
\end{equation}

We now estimate the second, third, and fourth terms on the right-hand side of (\ref{3.005}). By Young's inequality and H\"{o}lder's inequality, we obtain
\begin{align}\label{3.006}
		\begin{split}
			4\mathbb{E}\left[\|u(t)\|^{2}\left|\left(f(t),u(t)\right)\right|\right] &\leq \alpha\mathbb{E}\left[\|u(t)\|^{4}\right]+ \frac{27}{\alpha} \mathbb{E}\left[\|f(t)\|^4\right],\\
			2\mathbb{E}\left[\left|\left(g(t),v(t)\right)\right|\right]&\le \frac{\beta}{2}\mathbb{ E}\left[\|v(t)\|^2\right]+\frac{2}{\beta}\mathbb{ E}\left[\|g(t)\|^2\right],\\
			2\mathbb{E}\left[\left|\left(G\left(u(t)\right),v(t)\right)\right|\right]&\leq \frac{\beta}{2}\mathbb{E}\left[\|v(t)\|^{2}\right]dr+\frac{16\lambda^2}{\beta}\mathbb{E}\left[\|u(t)\|^{4}\right].
		\end{split}
	\end{align}
For the fifth and sixth terms on the right-hand side of (\ref{3.005}). By virtue of assumption ${(H_3)}$, Young's inequality and H\"{o}lder's inequality, one gets
\begin{align}\label{3.007}
			&6\varepsilon^2 \mathbb{E}\left[\|u(t)\|^{2}\sum_{k=1}^\infty\left(\|h_{k}\left(u(t)\right)+b_k(t)\|^{2}\right)\right]+\varepsilon^2\sum_{k=1}^\infty \mathbb{E}\left[\|\sigma_{k}\big(v(t)\big)+\gamma_{k}(t)\|^{2}\right]\notag \\
			\leq\,& 12\varepsilon^2\mathbb{E}\left[\sum_{k=1}^\infty\left(\|u(t)\|^{2}\|b_k(t)\|^{2}+\|\gamma_{k}(t)\|^{2}\right)\right]+12\varepsilon^2 \mathbb{E}\left[\sum_{k=1}^\infty\left(\|u(t)\|^{2}\|h_{k}\left(u(t)\right)\|^2+\|\sigma_{k}\big(v(t)\big)\|^{2}\right)\right]\notag \\
			\leq\,& \left(\alpha+24\varepsilon_0^2 \|\delta\|^2\right)\mathbb{E}\left[\|u(t)\|^{4}\right]+24\varepsilon_0^2 \|\delta\|^2\mathbb{E}\left[\|v(t)\|^{2}\right]\notag \\
			&+\frac{36\varepsilon_0^4}{\alpha}\mathbb{E}\left[ \|b(t)\|^{4}\right]+ 12\varepsilon_0^2\mathbb{E}\left[\|\gamma(t)\|^{2} \right]+48\varepsilon_0^2 \|\delta\|^2\notag \\
			\leq\,& \left(\alpha+24\varepsilon_0^2 \|\delta\|^2\right)\mathbb{E}\left[\|u(t)\|^{4}\right]+24\varepsilon_0^2 \|\delta\|^2\mathbb{E}\left[\|v(t)\|^{2}\right]\notag \\
			&+\frac{36\varepsilon_0^4}{\alpha}\mathbb{E}\left[ \|b(t)\|^{4}\right]+ 12\varepsilon_0^2\mathbb{E}\left[\|\gamma(t)\|^{4} \right]+6\varepsilon_0^2(1+8\|\delta\|^2).
	\end{align}
  Combining with (\ref{3.005})-(\ref{3.007}), let $\kappa=\min\left\{\alpha-\frac{18\lambda^2}{\beta},\frac{\beta}{2}\right\}$, we see that
 \begin{align}\label{3.008}
    	\begin{split}
    		&~~\frac{d}{dt}\mathbb{ E}\left[\|u(t)\|^4+\|v(t)\|^2\right]+\kappa\mathbb{ E}\left[\|u(t)\|^4+\|v(t)\|^2\right] \\
    		&\leq \widetilde{\kappa} \mathbb{ E}\left[\|f(t)\|^4+\|g(t)\|^4+\|b(r)\|^{4}+ \|\gamma(t)\|^{4}\right]+\widetilde{\kappa},
    	\end{split}
    \end{align}
    where $\widetilde{\kappa}=\max\left\{\frac{27}{\alpha},\frac{2}{\beta},\frac{9\beta^2}{576\alpha \|\delta\|^4},\frac{\beta}{4\|\delta\|^2},\frac{\beta(1+8\|\delta\|^2)}{8\|\delta\|^2}\right\}$.
    
Multiplying \eqref{3.008} by $e^{\kappa t}$, and then integrating from $\tau-t(t>0)$ to $\tau$, we deduce that
\begin{align}\label{3.009}
	\begin{split}
		&~\mathbb{E}\left[\left\|u(\tau, \tau-t, u_{\tau-t})\right\|^4+\left\|v(\tau, \tau-t, v_{\tau-t})\right\|^2\right]\leq e^{-\kappa t}\mathbb{E}\left[\|u_{\tau-t}\|^4+\left\|v_{\tau-t}\right\|^2\right]\\
		&+\widetilde{\kappa} e^{-\kappa \tau}\int_{\tau-t}^\tau e^{\kappa r}\mathbb{ E}\left[\|f(t)\|^4+\|g(t)\|^4+\|b(r)\|^{4}+ \|\gamma(t)\|^{4}\right]dr+\frac{\widetilde{\kappa}}{\kappa}.
	\end{split}
\end{align}
Thanks to $\varphi_{\tau-t}\in B(\tau-t)$ and $ B=\{B(t)\}_{t\in\mathbb{R}} \in \mathscr{D}$, we know
\begin{equation*}
	\lim\limits_{t\to \infty}e^{-\kappa  t}\left(\mathbb{E}\left[\left\|u_{\tau-t}\right\|^4+\left\|v_{\tau-t}\right\|^2\right]\right)\leq \lim\limits_{t\to \infty}e^{-\kappa  t}\left(\mathbb{E}\left[\left\|B({\tau-t})\right\|_{L^{4}(\Omega,\ell_c^{2})\times L^{2}(\Omega,\ell^{2})}\right]\right) = 0.
\end{equation*} 
Thus, there exists $ T = T(\tau,B)>0 $ such that
\begin{equation}\label{3.011}
	e^{-\kappa t}\left(\mathbb{E}\left[\left\|u_{\tau-t}\right\|^4+\left\|v_{\tau-t}\right\|^2\right]\right) \le \frac{\widetilde{\kappa}}{\kappa},\quad   \forall t\geq T,
\end{equation}
which, together with \eqref{3.009} and \eqref{3.011}, conclude the desired result (\ref{3.001}). This proof is finished.
\end{proof}
	
According to Lemma \ref{lemma3.1}, we shall directly give the existence of weakly compact $ \mathcal{D}$-pullback bounded absoring set.
\begin{lemma} \label{lemma 3.2}
Suppose that $(H_1)-(H_3)$, \eqref{LSWs4.1} and \eqref{LSWs4.5} hold. Then the mean random dynamical system $ \Phi $ generated by system (\ref{2.002}) has a unique weakly compact $ \mathcal{D}$-pullback bounded absorbing set $ \mathcal{K}=\{ \mathcal{K}(\tau):\tau \in \mathbb{R}\} \in \mathscr{D}$ in $ L^{4}(\Omega,\ell_c^{2})\times L^{2}(\Omega,\ell^{2})$, here for any given $ \tau \in \mathbb{R}$, $ \mathcal{K}(\tau) $ is denoted as follows 
\begin{equation}\label{key}
\begin{aligned}
	\mathcal{K}(\tau)&=\{\varphi \in L^{4}(\Omega,\ell_c^{2})\times L^{2}(\Omega,\ell^{2}) :\mathbb{E}\left[\|u(t)\|^4+\|v(t)\|^2\right]\le \mathfrak{R}_0(\tau) \},
\end{aligned}
\end{equation}
where $ \mathfrak{R}_0(\tau)=\mathfrak{M}_0+\mathfrak{M}_0\int_{-\infty}^\tau e^{\kappa (r-\tau)}\mathbb{ E}\left[\|f(t)\|^4+\|g(t)\|^4+\|b(r)\|^{4}+ \|\gamma(t)\|^{4}\right]dr $.
\end{lemma}

\begin{proof}
Since $ \mathcal{K}(\tau) $ is a bounded closed convex subset of the reflexive product Banach space  $ L^{4}(\Omega,\ell_c^{2})\times L^{2}(\Omega,\ell^{2})$, then it is obvious that $ \mathcal{K}(\tau) $ is weakly compact in  $ L^{4}(\Omega,\ell_c^{2})\times L^{2}(\Omega,\ell^{2})$. In particular, for any $ \tau\in\mathbb{R} $ and  $ B=\{B(t)\}_{t\in\mathbb{R}} \in \mathcal{D}$, by Lemma \ref{lemma3.1}  there exists $ T = T(\tau,B)>0 $ such that 
$$
	\Phi(\tau,\tau-t,B(\tau-t))=\varphi(\tau,\tau-t,B(\tau-t)) \subseteq \mathcal{K}(\tau),\quad \forall t\geq T.
$$
We finally verify $ \mathcal{K} \in \mathcal{D} $. From (\ref{3.001}), we can get
$$
	0\le \lim\limits_{\tau \to -\infty}e^{\kappa \tau}\|\mathcal{K}(\tau)\|^2_{L^4(\Omega,\mathscr{F}_{\tau}, \ell_c^2)\times L^2(\Omega,\mathscr{F}_{\tau}, \ell^2)} \le \lim\limits_{\tau \to -\infty}e^{\kappa \tau}\mathfrak{R}_0(\tau)=0,
$$	
which implies  that
$ \mathcal{K} $ satisfies (\ref{4.3}). 
This proof is finished.
\end{proof}

Ultimately, we present a main theorem establishing the existence and uniqueness of a weak $\mathcal{D}$-pullback mean random attractor for the mean random dynamical system $ \Phi $.
\begin{theorem}
Suppose that $(H_1)-(H_3)$, \eqref{LSWs4.1} and \eqref{LSWs4.5} hold. Then the mean random dynamical system $ \Phi$ generated by the system \eqref{2.002} possesses a unique $ \mathcal{D}$-pullback mean random attractor $ \mathcal{A} =\{ \mathcal{A}(\tau):\tau\in\mathbb{R}  \}\in\mathcal{D}$ in $L^{4}(\Omega,\ell_c^{2})\times L^{2}(\Omega,\ell^{2}) $ over $ (\Omega, \mathcal{F}, \{\mathcal{F}_t\}_{t\in\mathbb{R}},\mathbb{P} )$. In particular, for every $ \tau\in \mathbb{R}$, $ \mathcal{A}(\tau)$ can be represented as :
$$
	\mathcal{A}(\tau) = \bigcap_{s\ge 0}\overline{\bigcup_{t\ge s}\Phi(t,\tau-t)\mathcal{K}(\tau-t)}^w,
$$
where the closure is taken with respect to the weak topology of $ L^{4}(\Omega,\ell_c^{2})\times L^{2}(\Omega,\ell^{2}) $.
	
\end{theorem}
\begin{proof}
Combining with Lemmas \ref{lemma3.1} and \ref{lemma 3.2}, by \cite[Theorem $2.7$]{WBX2019} we can immediately get the existence and uniqueness of weak $\mathscr{D}$-pullback mean random attractors $ \mathcal{A} \in \mathcal{D}$ of $ \Phi $. This completes the proof.
\end{proof}

\section{Existence of invariant measures}

In this section, we shall establish the existence of invariant measures of autonomous stochastic discrete  long-wave–short-wave resonance equation \eqref{2.002} driven by nonlinear noise. To this end, we shall proceed with the argument in three stages: $(i)$ deriving uniform estimates for solutions to problem \eqref{2.002} with initial time $\tau=0$; $(ii)$ establishing the tightness of the family of probability distributions of the solutions in $\ell_c^2\times \ell^2$; $(iii)$ applying the Krylov-Bogolyubov method to conclude the existence of an invariant measure. We assume that the external forcing terms $f$, $g$, $ b_{k} $, $\gamma_{k} $ are independent of time $t$ and the sample $\omega \in \Omega$ in this section. Under this time-independent and sample-independent assumptions, we present the following assumption:
\begin{align}\label{LSWs5.1}
	\|f\|^2+\|g\|^2+\sum\limits_{k=1}^{\infty}\|b_k\|^2 +\sum\limits_{k=1}^{\infty}\|\gamma_k\|^2< \infty.
\end{align}

In order to clearly represent the dependence of solutions on the noise intensity $ \varepsilon \in [0,\varepsilon_0] $, we denote by $ \varphi^{\varepsilon}(t,0,\varphi_{0})=(u^{\varepsilon}(t,0,u_{0}),v^{\varepsilon}(t,0,v_{0}))^\mathrm{T} $ the solution of system (\ref{2.002}) with initial data $ \varphi_{0}=(u_{0},v_{0})^\mathrm{T} $ at initial time $0$. Next, we shall first give the uniform estimates of solution $ \varphi^{\varepsilon}(t,0,\varphi_{0})$ of problem (\ref{2.002}) in $L^4(\Omega,\ell_c^2)\times L^2(\Omega,\ell^2)$ for all $t\geq 0$
for checking the tightness of probability distributions of solutions and the existence of invariant measures.
\begin{lemma}\label{lemma 5.1}
Suppose that the assumptions $(H_2)$, $(H_3)$, \eqref{LSWs4.1} and \eqref{LSWs5.1} hold. Then for all $\varepsilon \in \left[0,\varepsilon_0\right]$ with $\varepsilon_0=\max\left\{\sqrt{\frac{\alpha}{24\|\delta\|^2}},\sqrt{\frac{\beta}{48\|\delta\|^2}}\right\}$, the solution $\varphi^{\varepsilon}(t,0,\varphi_{0})$ of system \eqref{2.002} with initial data $\varphi_{0}$ at inital time $\tau=0$ satisfies 
\begin{align}\label{5.1}
	\begin{split}
		~\sup_{t\geq 0}&\sup_{\varepsilon \in [0,\varepsilon_0]}\mathbb{ E}\left[ \|u^{\varepsilon}(t,0,u_{0})\|^4 +\|v^{\varepsilon}(t,0,u_{0})\|^2\right]\\
		& \le  \mathbb{ E}[\|u_0\|^4+\|v_0\|^2]+\mathfrak{M}_1\left(1+\|f\|^4+\|g\|^4+\|b\|^4+\|\gamma\|^4\right),
	\end{split}
\end{align}
where $\mathfrak{M}_1>0$ is a constant independent of $\varepsilon$, $\varphi_{0}$, $t$.
\end{lemma}
\begin{proof}
Since $f$, $g$, $b_k$, $\gamma_{k}$ are time-independent deterministic functions. Then by Gronwall's inequality we can obtain from (\ref{3.008}) that for all $t\geq 0$,
\begin{equation*} 
		\mathbb{ E}\left[\|u^\varepsilon(t)\|^4+\|v^\varepsilon(t)\|^2\right]\le e^{-\kappa t}\mathbb{ E}\left[\|u_0\|^4+\|v_0\|^2\right]+\frac{\widetilde{\kappa}}{\kappa}\left(\|f\|^4+\|g\|^4+\|b\|^{4}+ \|\gamma\|^{4}+1\right),
\end{equation*} 
from which we can obtain the desired conclusion.
This proof is finished.
\end{proof}

Next, we shall prove the tightness of the family of probability distributions of
solutions on $\ell_c^2\times \ell^2$, which can be achieved by
deriving the following uniform tail-estimates of solutions for system \eqref{2.002}.
\begin{lemma}\label{lemma 5.2}
Suppose that the assumptions $(H_2)$, $(H_3)$, \eqref{LSWs4.1} and \eqref{LSWs5.1} hold. Then for every compact subset $K(:=\widetilde{K}\times \widehat{K})$ of $\ell_c^2 \times \ell^2$ \footnote{$\widetilde{K}$ and $\widehat{K}$ are compact subsets of $\ell_c^2$ and $\ell^2$, respectively.}, $\varepsilon \in \left[0,\varepsilon_0\right]$ with $\varepsilon_0=\max\left\{\sqrt{\frac{\alpha}{24\|\delta\|^2}},\sqrt{\frac{\beta}{48\|\delta\|^2}}\right\}$ and $\zeta >0 $,  there exists a positive number $ N=N(K,\zeta) \in \mathbb{N} $ such that the solution $\varphi^\varepsilon(t;0,\varphi_{0})$ of system \eqref{2.002} with $ \varphi_{0} \in K $ satisfies
\begin{equation}\label{5.3}
\sup_{n\geq N}\sup_{t\geq 0}\sup_{\varepsilon \in \left[0,\varepsilon_0\right]} \sup_{\varphi_{0} \in K}	\left(\left(\sum\limits_{|m|\geq n}\mathbb{ E}\left[|u^\varepsilon_m(t,0,u_{0})|^2\right]\right)^2+\sum\limits_{|m|\geq n}\mathbb{ E}\left[|v^\varepsilon_m(t,0,v_{0})|^2\right]\right) <\zeta,
\end{equation}
for all $ n\geq N $ and $ t\geq 0 $.
\end{lemma}
\begin{proof}
Let  $ \rho : \mathbb{R} \to \mathbb{R} $ be a smooth function such that $0\leq \rho(s)  \leq 1$ for all $s\in \mathbb{R}$, and 
\begin{equation}\label{smoothfun}
	\rho(s) =\left\{\begin{array}{rcl}
		0, & \mbox{for} & |s|\leq 1,\\
		1, & \mbox{for} & |s|\geq 2.
	\end{array}\right.
\end{equation}
It is easy to verify from \eqref{smoothfun} that there is a constant $ C_0 >0$ such that $ |\rho'(s)|\leq C_0 $ for every $ s\in\mathbb{R}$.
For each $n\in\mathbb{N}$ and $\varphi = (\varphi_m)_{m\in\mathbb{Z}}$, we set 
\begin{equation*}
	\rho_n=\left(\rho\left(\frac{m}{n}\right)\right)_{m\in\mathbb{Z}}\quad  \text{and} \quad \,\, \rho_n\varphi=\left( \rho\left(\frac{m}{n}\right)\varphi_m \right)_{m\in\mathbb{Z}}.
\end{equation*}
We also apply analogous notation to the other terms. By \eqref{2.002} we get
\begin{equation}\label{5.4}
d(\rho_n\varphi^{\varepsilon}(t))+\mathfrak{A}(\rho_n\varphi^{\varepsilon}(t) )dt=	\rho_n\mathscr{G}(\varphi^{\varepsilon}(t),t)dt +\varepsilon\sum\limits_{k=1}^{\infty}\rho_n\mathscr{H}_k(\varphi^{\varepsilon}(t),t)dW_k(t). 
\end{equation}

Similar to \eqref{3.002}, combining with Ito's formula and $\eqref{5.4}$, and by taking the real part of the resulting expression we have
\begin{equation}
	\begin{aligned}\label{5.5}
		d\|\rho_nu^\varepsilon(t)\|^p\leq & - p\alpha\|\rho_nu^\varepsilon(t)\|^pdt+p\textbf{Im}(\|\rho_nu^\varepsilon(t)\|^{p-2}(f(t), \rho_n^2u^\varepsilon(t))) d t \\
		&+p\varepsilon \sum_{k=1}^{\infty}\|\rho_nu^\varepsilon(t)\|^{p-2}\textbf{Im}\left( h_k(u^\varepsilon(t))+b_k, \rho_n^2z^\varepsilon(t)\right) d W_k(t)\\
        &+\frac{p(p-1)}{2}\varepsilon^2\|\rho_nu^\varepsilon(t)\|^{p-2} \sum_{k=1}^{\infty}\left\|\rho_nh_k(u^\varepsilon(t))+\rho_nb_k(t)\right\|^2 d t,
	\end{aligned}
\end{equation}
and
\begin{equation}
	\begin{aligned}\label{5.6}
		&d\|\rho_n^2v^\varepsilon(t)\|^2\leq -2\beta \|\rho_n^2v^\varepsilon(t)\|^2dt+2|(G\left(u^\varepsilon(t)\right),\rho_n^4v^\varepsilon(t))|dt+2(g(t), \rho_n^4v^\varepsilon(t)) d t\\
		&+2 \varepsilon \sum_{k=1}^{\infty}\left( \sigma_k(v^\varepsilon(t))+\gamma_k, \rho_n^4v^\varepsilon(t)\right) d W_k(t) 
		+\varepsilon^2 \sum_{k=1}^{\infty}\left\|\rho_n^2 \sigma_k(v^\varepsilon(t))+\rho_n^2 \gamma_k\right\|^2 d t.
	\end{aligned}
\end{equation}
Combining with (\ref{5.5})-(\ref{5.6}) and then taking expectation, we obtain
\begin{equation}
	\begin{aligned}\label{5.7}
		&~~\frac{d}{dt}\mathbb{ E}\left[\|\rho_nu^\varepsilon(t)\|^4+\|\rho_n^2v^\varepsilon(t)\|^2 \right]\leq -4\alpha\mathbb{ E}[\|\rho_nu^\varepsilon(t)\|^4]-2\beta\mathbb{ E}[\|\rho_n^2v^\varepsilon(t)\|^2]\\
		&+2\mathbb{ E}\left[\left|G(u^\varepsilon(t),\rho_n^4v^\varepsilon(t))\right| \right]+4\mathbb{ E}[\|\rho_nu^\varepsilon(t)\|^2\textbf{Im}(f(t), \rho_n^2u^\varepsilon(t))] +2\mathbb{ E}[(g, \rho_n^4v^\varepsilon(t))] \\
		&+6\varepsilon^2\sum_{k=1}^{\infty}\mathbb{ E}\left[\|\rho_nu^\varepsilon(t)\|^2 \left\|\rho_nh_k(u^\varepsilon(t))+\rho_nb_k(t)\right\|^2\right]+\varepsilon^2 \sum_{k=1}^{\infty}\mathbb{ E}\left[\left\|\rho_n^2 \sigma_k(v^\varepsilon(t))+\rho_n^2 \gamma_k\right\|^2\right].
	\end{aligned}
\end{equation}

For the third term on the right-hand side of (\ref{5.7}), we use Young's and H\"{o}lder's inequalities to obtain that
\begin{align}\label{5.9}
	&~~2\mathbb{ E}\left[\left|(G(u^\varepsilon(t),\rho_n^4v^\varepsilon(t))\right| \right]\notag\\ 
    &\le \frac{2\lambda^2}{\beta}\mathbb{ E}\left[\left\| \rho_n^2 B(|u^\varepsilon(t)|^2)\right\|^2\right]+\frac{\beta}{2}\mathbb{ E}\left[\left\|
 \rho_n^2v^\varepsilon(t)\right\|^2\right]\notag\\ 
	&=\frac{2\lambda^2}{\beta}\mathbb{ E}\left[\sum_{m\in\mathbb{Z}}\left(\rho^2(\frac{m}{n})(|u_{m+1}^\varepsilon(t)|^2-|u_{m}^\varepsilon(t)|^2)\right)^2\right]+\frac{\beta}{2}\mathbb{ E}\left[\left\|
 \rho_n^2v^\varepsilon(t)\right\|^2\right]\notag\\ 
	&\le \frac{4\lambda^2}{\beta}\mathbb{E}\left[\sum_{m\in\mathbb{Z}}\rho^4(\frac{m}{n})(|u_{m+1}^\varepsilon(t)|^4+|u_{m}^\varepsilon(t)|^4)\right]+\frac{\beta}{2}\mathbb{ E}\left[\left\|
 \rho_n^2v^\varepsilon(t)\right\|^2\right]\notag\\ 
 &\le \frac{4\lambda^2}{\beta}\mathbb{E}\left[\sum_{m\in\mathbb{Z}}\rho^4(\frac{m}{n})|u_{m+1}^\varepsilon(t)|^4\right]+\frac{4\lambda^2}{\beta}\mathbb{E}\left[\sum_{m\in\mathbb{Z}}\rho^4(\frac{m}{n})|u_{m}^\varepsilon(t)|^4\right]+\frac{\beta}{2}\mathbb{ E}\left[\left\| \rho_n^2v^\varepsilon(t)\right\|^2\right]\notag\\ 
 &\le \frac{4\lambda^2}{\beta}\mathbb{E}\left[\left(\sum_{m\in\mathbb{Z}}\rho^2(\frac{m}{n})|u_{m+1}^\varepsilon(t)|^2\right)^2\right]+\frac{4\lambda^2}{\beta}\mathbb{E}\left[\|\rho_n u^\varepsilon\|^4\right]+\frac{\beta}{2}\mathbb{ E}\left[\left\| \rho_n^2v^\varepsilon(t)\right\|^2\right]\notag\\ 
 &\le \frac{4\lambda^2}{\beta}\mathbb{E}\left[\left(2\sum_{m\in\mathbb{Z}}\rho^2(\frac{m+1}{n})|u_{m+1}^\varepsilon(t)|^2+2\sum_{m\in\mathbb{Z}}\left (\rho(\frac{m}{n})-\rho(\frac{m+1}{n})\right)^2|u_{m+1}^\varepsilon(t)|^2\right)^2\right]\notag\\ 
 &+\frac{4\lambda^2}{\beta}\mathbb{E}\left[\|\rho_n u^\varepsilon\|^4\right]+\frac{\beta}{2}\mathbb{ E}\left[\left\| \rho_n^2v^\varepsilon(t)\right\|^2\right]\notag\\ 
  &\le \frac{8\lambda^2}{\beta}\mathbb{E}\left[\left(2\sum_{m\in\mathbb{Z}}\left (\rho(\frac{m}{n})-\rho(\frac{m+1}{n})\right)^2|u_{m+1}^\varepsilon(t)|^2\right)^2\right]+\frac{36\lambda^2}{\beta}\mathbb{E}\left[\|\rho_n u^\varepsilon\|^4\right]+\frac{\beta}{2}\mathbb{ E}\left[\left\| \rho_n^2v^\varepsilon(t)\right\|^2\right]\notag\\ 
 &\le \frac{32\lambda^2 C_0^4}{n^4\beta}\mathbb{E}\left[\left\| u^\varepsilon(t)\right\|^2\right]+\frac{12\lambda^2}{\beta}\mathbb{E}\left[\|\rho_n u^\varepsilon\|^4\right]+\frac{\beta}{2}\mathbb{ E}\left[\left\| \rho_n^2v^\varepsilon(t)\right\|^2\right].
\end{align}

For the fourth and fifth terms on the right-hand side of (\ref{5.7}), we have
\begin{align}\label{5.10}
	\begin{split}
	4\mathbb{ E}[\|\rho_nu^\varepsilon(t)\|^2\textbf{Im}(f(t), \rho_n^2u^\varepsilon(t))] \le
	 \alpha \mathbb{E}\left[\left\|\rho_n u^\varepsilon (t)\right\|^4\right]+\textbf{c} \mathbb{E}[\|\rho_nf\|^4],
		\end{split}
\end{align}
\begin{equation}\label{5.11}
	\begin{aligned}
		 2\mathbb{ E}[|(g, \rho_n^4v^\varepsilon(t))|]  \le
		\frac{\beta}{2} \mathbb{E}[\left\|\rho_n^2 v^\varepsilon (t)\right\|^2]+\textbf{c} \mathbb{E}[\|\rho_n^2g\|^2].
	\end{aligned}
\end{equation}

For the sixth term on the right-hand side of (\ref{5.7}), by Young's inequality, H\"{o}lder's inequality as well as the properties of smooth function $ \rho $ defined in \eqref{smoothfun}, we can derive that
\begin{align}\label{5.12}
	&6\varepsilon^2\sum_{k=1}^{\infty}\mathbb{ E}\left[\|\rho_nu^\varepsilon(t)\|^2 \left\|\rho_nh_k(u^\varepsilon(t))+\rho_nb_k(t)\right\|^2\right]\notag \\
=&6\varepsilon^2 \mathbb{ E}\left[\|\rho_nu^\varepsilon(t)\|^2\sum_{k=1}^{\infty}\sum_{m\in\mathbb{Z} } \rho^2(\frac{m}{n} )(h_{k,m}(u_m^\varepsilon(t))+b_{k,m})^2\right]\notag \\
\le& 12\varepsilon^2 \mathbb{ E}\left[\|\rho_nu^\varepsilon(t)\|^2\sum_{k=1}^{\infty}\sum_{m\in\mathbb{Z} } \rho^2(\frac{m}{n} )((h_{k,m}(u_m^\varepsilon(t)))^2+|b_{k,m}|^2)\right]\notag \\
\le& 12\varepsilon^2 \mathbb{ E}\left[\|\rho_nu^\varepsilon(t)\|^2\sum_{k=1}^{\infty}\sum_{m\in\mathbb{Z} }\rho^2(\frac{m}{n} )(h_{k,m}(u_m^\varepsilon(t)))^2\right]+12\varepsilon^2 \mathbb{ E}\left[\|\rho_nu^\varepsilon(t)\|^2\sum_{k=1}^{\infty}\sum_{m\in\mathbb{Z} }\rho^2(\frac{m}{n} )|b_{k,m}|^2\right]\notag \\
\le& 12\varepsilon^2\mathbb{ E}\left[\|\rho_nu^\varepsilon(t)\|^2\sum_{k=1}^{\infty}\sum_{m\in\mathbb{Z} }\rho^2(\frac{m}{n} )(\delta _{k,m}(1+|u^\varepsilon_m(t)|))^2\right]+12\varepsilon^2 \mathbb{ E}\left[\|\rho_nu^\varepsilon(t)\|^2\|\rho_n b\|^2\right]\notag \\
\le& 24\varepsilon_0^2\mathbb{ E}\left[\|\rho_nu^\varepsilon(t)\|^2\sum_{k=1}^{\infty}\sum_{m\in\mathbb{Z} }\rho^2(\frac{m}{n} )|\delta  _{k,m}|^2\right]+24\varepsilon_0^2\mathbb{ E}\left[\|\rho_nu^\varepsilon(t)\|^2\sum_{k=1}^{\infty}\sum_{m\in\mathbb{Z} }\rho^2(\frac{m}{n} )|\delta _{k,m}|^2|u^\varepsilon _m(t)|^2\right]\notag \\
&+12\varepsilon_0^2 \mathbb{ E}\left[\|\rho_nu^\varepsilon(t)\|^2\|\rho_n b\|^2\right]\notag \\
\le&  \alpha \mathbb{ E}\left[\| \rho_n u^\varepsilon \|^4\right]+\textbf{c} \mathbb{ E}\left[\|\rho_n \delta  \|^4\right]+\textbf{c}  \mathbb{ E}\left[\|\rho_n b\|^4\right].
	\end{align}
A similar argument applied to the seventh term on the right-hand side of (\ref{5.12}) yields
\begin{equation}
	\begin{aligned}\label{5.13}
	&\varepsilon^2 \sum_{k=1}^{\infty}\mathbb{ E}\left[\left\|\rho_n^2 \sigma_k(v^\varepsilon(t))+\rho_n^2 \gamma_k\right\|^2\right]\le \frac{\beta}{2} \mathbb{ E}\left[\| \rho_n^2 v^\varepsilon \|^2\right]+\textbf{c} \mathbb{ E}\left[\|\rho_n^2 \delta  \|^4\right]+\textbf{c}  \mathbb{ E}\left[\|\rho_n^2 \gamma\|^4\right].
	\end{aligned}
\end{equation}

Substituting (\ref{5.9})-(\ref{5.13}) into (\ref{5.7}), because $f, g, \delta_k, b_k, \gamma_{k}$ are independent of $t$ and $\omega$, then we can get that 
\begin{equation}
	\begin{aligned}\label{5.14}
		&~~\frac{d}{dt}\mathbb{ E}\left[\|\rho_nu^\varepsilon(t)\|^4+\|\rho_n^2 v^\varepsilon(t)\|^2 \right]+2(\alpha-\frac{18\lambda^2}{\beta})\mathbb{ E}\left[\|\rho_nu^\varepsilon(t)\|^4\right]+\beta\mathbb{ E}\left[\|\rho_nv^\varepsilon(t)\|^2\right]\\
		&\leq \frac{16\lambda^2 C_0^4}{n^4\beta}\mathbb{E}\left[\left\| u^\varepsilon(t)\right\|^2\right]+\textbf{c}\left(\left\|\rho_nf\right\|^4+\left\|\rho_n^2g\right\|^2+\|\rho_n \delta  \|^4+\|\rho_n b\|^4+\|\rho_n^2 \gamma\|^4\right)\\
		&\leq \frac{16\lambda^2 C_0^4}{n^4\beta}\mathbb{E}\left[\left\| u^\varepsilon(t)\right\|^2\right]+\textbf{c}\left(\left\|\rho_nf\right\|^4+\left\|\rho_ng\right\|^2+\|\rho_n \delta  \|^4+\|\rho_n b\|^4+\|\rho_n \gamma\|^4\right).
	\end{aligned}
\end{equation}

We know from Lemma \ref{lemma3.1} that there exists $C_1=C_1(K)>0$ independent of $n$ such that for all $\varphi_{0}\in K$ and $t\geq 0$, it holds $ \mathbb{E}\left[\left\| u^\varepsilon(t)\right\|^2\right]\leq {C_{1}}$. Therefore, for any $ \zeta>0 $, there exists $ N_1:=N_1(\zeta) >0$ such that for any $ n\geq N_1 $, 
\begin{equation}\label{5.16}
	\frac{16\lambda^2 C_0^4}{n^4\beta}\mathbb{E}\left[\left\| u^\varepsilon(t)\right\|^2\right]\leq \frac{\kappa \zeta}{12}, 
\end{equation} 
Moreover, by $(H_3)$ and \eqref{LSWs5.1}, we can deduce that for any $\zeta>0 $ there exists $ N_2:=N_2(\zeta) $ such that for any $n\geq N_2 $,
\begin{equation}\label{5.17}
  \textbf{c}\left\|\rho_nf\right\|^4= \textbf{c}\left(\sum_{|m|\geq n} |f_m|^2\right)^2\leq \frac{\kappa \zeta}{12},\quad 
  \textbf{c}\left\|\rho_n g \right\|^4= \textbf{c}\left(\sum_{|m|\geq n} |g_m|^2\right)^2\leq \frac{\kappa \zeta}{12},
\end{equation}
and
\begin{align}\label{5.sfsf17}
	\begin{split}
		  &~~\textbf{c}\left(\|\rho_n \delta  \|^4+\|\rho_n b  \|^4+\|\rho_n \gamma \|^4\right)\\
		  &=\textbf{c}\left(\sum_{|m|\geq n}\sum_{k=1}^\infty  |\delta_{k, m}|^2\right)^2+\textbf{c}\left(\sum_{|m|\geq n}\sum_{k=1}^\infty  |b_{k, m}|^2\right)^2+\textbf{c}\left(\sum_{|m|\geq n}\sum_{k=1}^\infty  |\gamma_{k, m}|^2\right)^2\\
		  &\leq \frac{\kappa \zeta}{4}.
	\end{split}
\end{align}
By (\ref{5.14})-(\ref{5.sfsf17}) and applying the Gronwall's inequality, we can get that for any $t\geq 0$
\begin{align}\label{5.18}
		\begin{split}
			\mathbb{ E}\left[\|\rho_nu^\varepsilon(t,0,u_0)\|^4+\|\rho_n^2 v^\varepsilon(t,0,v_0)\|^2 \right] &\le e^{-\kappa t}\mathbb{ E}\left[\|\rho_nu(0)\|^4+\|\rho_n^2 v(0)\|^2 \right] +\frac{ \zeta}{2}\\
			&\le \mathbb{ E}\left[\|\rho_nu(0)\|^4+\|\rho_n^2 v(0)\|^2 \right] +\frac{\zeta}{2}.
		\end{split}
\end{align} 

Since $ K  $ is a compact subset of $\ell_c^2 \times \ell^2$ and $ \varphi_{0}=(u_0,v_0) \in K $, thus we can get that as $n\rightarrow \infty$,
\begin{align*}
	&~~\sup_{\varphi_{0}\in K}\sup_{t\geq 0}\mathbb{ E}\left[\|\rho_nu_0\|^4+\|\rho_n^2 v_0\|^2 \right]\\
	&\leq \sup_{\varphi_{0}\in K}\sup_{t\geq 0}\mathbb{ E}\left[\left(\sum_{|m|\geq n}  |u_{0, m}|^2\right)^2+\sum_{|m|\geq n}  |v_{0, m}|^2\right]\rightarrow 0,
\end{align*}
where the last step follows from the uniform tail estimate and the fact that $\left(\sum_{|m|\geq n}  |u_{0, m}|^2\right)^2$ dominates $\sum_{|m|\geq n}  |u_{0, m}|^2$ for large $n$ when the tail is small.
That is to say, there exists $N=N_3(K,\zeta)>0$ such that for any $n\geq N_3$ and $ \varphi_{0}\in K $,
$$
\mathbb{ E}\left[\|\rho_nu(0)\|^4+\|\rho_n^2 v(0)\|^2 \right]\leq  \frac{\zeta}{2},
$$
which, together with (\ref{5.18}), conclude that  there exists $ N=\max\{N_1,N_2,N_3\} $ such that for all $t\geq 0$, $ n\geq N $ and $ \varphi_{0}\in K$,
\begin{equation}\label{5.19}
	\begin{aligned}
	&~~\mathbb{ E}\left[\left(\sum_{|m|\geq 2n}  |u^\varepsilon_{ m}(t,0,u_0)|^2\right)^2+\sum_{|m|\geq 2n}  |v^\varepsilon_{k}(t,0,v_0)|^2 \right]\\
	&\leq \mathbb{ E}\left[\|\rho_nu^\varepsilon(t,0,u_0)\|^4+\|\rho_n^2 v^\varepsilon(t,0,v_0)\|^2 \right]
\leq  {\zeta}.
	\end{aligned}
\end{equation}
Thus, the conclusion \eqref{5.3} holds. This proof is finished.
\end{proof}

We denote the probability distribution of solution $\varphi(t,0,u_{0})$ of the
stochastic discrete  long-wave-short-wave resonance system \eqref{2.002} by $\mathfrak{L}(\varphi(t,0,u_{0}))$. It follows from Lemmas \ref{lemma 5.1} and \ref{lemma 5.2} that the family of distributions for the solutions to system \eqref{2.002} (with $\tau=0$) is tight.

\begin{lemma}\label{lemma SFf5.3}
	Suppose that the assumptions $(H_2)$, $(H_3)$, \eqref{LSWs4.1} and \eqref{LSWs5.1} hold. Then for every compact subset $K(:=\widetilde{K}\times \widehat{K})$ in $\ell_c^2 \times \ell^2$ and $\varepsilon \in \left[0,\varepsilon_0\right]$ with $\varepsilon_0=\max\left\{\sqrt{\frac{\alpha}{24\|\delta\|^2}},\sqrt{\frac{\beta}{48\|\delta\|^2}}\right\}$, the family $\{\mathfrak{L}(\varphi(t,0,\varphi_{0})):t\geq 0,\varphi_{0}\in K\}$ of the distributions of the solutions of system \eqref{2.002} is tight in $\ell_c^2 \times \ell^2$. More precisely, for every $\zeta> 0$, there exists a compact subset $\mathfrak{K}_{\varepsilon}^\zeta$ of $\ell_c^2 \times \ell^2$ such that for all $\varphi_{0}\in K$,
	\begin{align}\label{4.29}
		\mathbb{P}\left(\{\omega\in\Omega:\varphi(t,0,\varphi_{0})\notin \mathfrak{K}_{\varepsilon}^\zeta\}\right)< \zeta,\quad  \forall t\geq 0.
	\end{align}
\end{lemma}
\begin{proof}
	For given $n\in \mathbb{N}$, let $\mathfrak{I}_{[-n,n]}$ a the characteristic function of $[-n,n]$. If $ \varphi(t,0,\varphi_{0}) $
	is the solution of stochastic discrete  long-wave-short-wave resonance system \eqref{2.002} with $\tau=0$, then we write
	$$
	\widetilde{\varphi}_{n}^{\varepsilon}(t,\varphi_{0})=(\mathfrak{I}_{[-n,n]}(m)\varphi_{m}^{\varepsilon}(t,0,u_{0}))_{m\in \mathbb{Z}}\quad \textrm{and}\quad \widehat{\varphi}_{n}^{\varepsilon}(t,\varphi_{0})=((1-\mathfrak{I}_{[-n,n]}(m))\varphi_{m}^{\varepsilon}(t,0,\varphi_{0}))_{m\in \mathbb{Z}}.
	$$
	It follows that for all $n\in \mathbb{N}$ and $t\geq 0$, 
	\begin{align}\label{4.30}
		 \varphi^{\varepsilon}(t,0,\varphi_{0}) =\widetilde{\varphi}_{n}^{\varepsilon}(t,u_{0})+\widehat{\varphi}_{n}^{\varepsilon}(t,u_{0}).
	\end{align}
	
	By Lemma \ref{lemma 5.2}, for every $\zeta>0$ and $k\in \mathbb{N}$, there exists an integer $n_{k}$ depending on $K$, $\zeta$ and $k$ such that for all $t\geq0$ and $\varphi_{0}\in K$,
	\begin{align}\label{4.31}
		\mathbb{E}\left[\left(\sum_{|m|\geq n_{k}}|u_{m}(t,0,u_{0})|^{2}\right)^2\right]<\frac{\zeta^2}{2^{8k+4}}\text{~~and~~}\mathbb{E}\left[\sum_{|m|\geq n_{k}}|v_{m}(t,0,v_{0})|^{2}\right]<\frac{\zeta}{2^{4k+2}},
	\end{align}
	which means that for all $t\geq0 $ and $\varphi_{0}\in K$,
	\begin{align}\label{4.dfsdf32}
		\mathbb{E}\left[\|\widehat{u}_{n_k}^{\varepsilon}(t,u_{0})\|^{4}\right]<\frac{\zeta^2}{2^{8k+4}}
		\text{~~and~~}
		\mathbb{E}\left[\|\widehat{v}_{n_k}^{\varepsilon}(t,v_{0})\|^{2}\right]<\frac{\zeta}{2^{4k+2}}.
	\end{align}
	Since $\varphi_{0}$ belongs to the compact subset $K$ in $\ell^{2}_c\times \ell^{2}$, which along with Lemma \ref{lemma 5.1} can deduce that there exists $\textbf{c} = \textbf{c}(K)>0$ such that for all $t\geq0 $ and $\varphi_{0}\in K$,
	\begin{align}\label{4.dfsdf33}
		\mathbb{E}\left[\|u(t,0,u_{0}) \|^{4}+\|v(t,0,v_{0}) \|^{2}\right]\leq \textbf{c}^{2}.
	\end{align}
	
	For every $m\in \mathbb{N}$, we define
	\begin{align}\label{4.dfsdf34}
		\mathfrak{Y}_{m}^{\varepsilon}=\left\{\varpi=(\varpi_{m})_{m\in \mathbb{Z}}\in \ell^{2}_c\times \ell^2:\varpi_{m}=0\enspace\textrm{for}\;|m|>n_{k}\enspace\textrm{and}
		\enspace\|\varpi\|_{\ell^{2}_c\times \ell^2}\leq\frac{2^{k+1}(1+\textbf{c})}{\sqrt{\zeta}}\right\},
	\end{align}
	 and 
	\begin{align}\label{4.dfsdf35}
		\mathfrak{K}_{m}^{\varepsilon}=\left\{\varpi\in \ell^{2}_c\times \ell^2:\|\varpi-w\|_{\ell^{2}_c\times \ell^2}\leq\frac{1}{2^{k}}\;\textrm{for some}\;w\in \mathfrak{Y}_{m}^{\varepsilon}\right\},
	\end{align}
	where $\textbf{c}$ is the same constant as in \eqref{4.dfsdf33}.
	Then by \eqref{4.30}, \eqref{4.dfsdf34} and \eqref{4.dfsdf35}, we obtain that
	\begin{align}\label{4.dfsdf36}
		&\left\{\omega\in\Omega: \varphi^{\varepsilon}(t,0,\varphi_{0}) \notin\mathfrak{K}_{m}^{\varepsilon}\right\}\notag\\
		\subseteq&\left\{\omega\in\Omega:\widetilde{\varphi}^{\varepsilon}_{n_{k}}(t,\varphi_{0})\notin\mathfrak{Y}_{m}^{\varepsilon}\right\}\cup
		\left\{\omega\in\Omega: \varphi(t,0,\varphi_{0}) \notin\mathfrak{K}_{m}^{\varepsilon}\;\;\textrm{and}\;\;\widetilde{\varphi}^{\varepsilon}_{n_{k}}(t,u_{0})\in\mathfrak{Y}_{m}^{\varepsilon}\right\}\notag\\
		\subseteq&\left\{\omega\in\Omega:\widetilde{\varphi}^{\varepsilon}_{n_{k}}(t,\varphi_{0})\notin\mathfrak{Y}_{m}^{\varepsilon}\right\}\cup
		\left\{\omega\in\Omega:\|\widehat{\varphi}^{\varepsilon}_{n_{k}}(t,\varphi_{0})\|_{\ell_c^{2} \times \ell^{2}}>\frac{1}{2^{k}}\right\}.
	\end{align}
	By \eqref{4.dfsdf33}, \eqref{4.dfsdf34} and Chebychev's inequality we get that for all $t\geq0$ and $u_{0}\in K$,
	\begin{align}\label{4.dfsdf37}
		\begin{split}
			&\mathbb{P}(\{\omega\in\Omega:\widetilde{\varphi}_{n_{k}}^{\varepsilon}(t,\varphi_{0})\notin\mathfrak{Y}_{m}^{\varepsilon}\})\leq\mathbb{P}
			\left(\left(\omega\in\Omega:\|\widetilde{\varphi}_{n_{k}}^{\varepsilon}(t,\varphi_{0})\|_{\ell^{2}_c\times \ell^2}>\frac{2^{k+1}(1+\textbf{c})}{\sqrt{
					\zeta}}\right)\right)\\
			&\leq\frac{\zeta}{2^{2(k+1)}(1+\textbf{c})^{2}}\mathbb{E}\left[\|\widetilde{\varphi}_{n_{k}}^{\varepsilon}(t,\varphi_{0})\|_{\ell_c^{2} \times \ell^{2}}^{2}\right]	\leq\frac{\zeta}{2^{2k+1}(1+\textbf{c})^{2}}\mathbb{E}\left[\|\varphi^{\varepsilon}(t,0,\varphi_{0})\|_{\ell_c^{2} \times \ell^{2}}^{2}\right]\\
			&\leq\frac{\zeta}{2^{2k}(1+\textbf{c})^{2}}\mathbb{E}\left[1+\|u^{\varepsilon}(t,0,u_{0})\|^{4}+\|v^{\varepsilon}(t,0,v_{0})\|^{2}\right]\leq\frac{\zeta}{2^{2k}},
		\end{split}
	\end{align}
	here $\|\varphi^{\varepsilon}(t,0,\varphi_{0})\|=\|u^{\varepsilon}(t,0,u_{0})\|+\|v^{\varepsilon}(t,0,v_{0})\|$,
	Combining with \eqref{4.dfsdf32} and Chebychev's inequality. we obtain
	\begin{align}\label{4.dfsdf38}
		\begin{split}
			&\mathbb{P}\left(\left\{\omega\in\Omega:\|\widehat{\varphi}_{n_{k}}^{\varepsilon}(t,\varphi_{0})\|_{\ell_c^{2} \times \ell^{2}}>\frac{1}{2^{k}}\right\}\right)\leq2^{2k}\mathbb{ E}\left[\|\widehat{\varphi}^{\varepsilon}_{n_{k}}(t,\varphi_{0})\|_{\ell_c^{2} \times \ell^{2}}^{2}\right]\\
			&\leq 2^{2k+1}\sqrt{\mathbb{ E}\left[\|\widehat{u}^{\varepsilon}_{n_{k}}(t,u_{0})\|^{4}\right]}
			+2^{2k+1}\mathbb{ E}\left[\|\widehat{v}^{\varepsilon}_{n_{k}}(t,v_{0})\|^{2}\right]
			\leq\frac{\zeta}{2^{2k}}.
		\end{split}
	\end{align}
	By \eqref{4.dfsdf36}-\eqref{4.dfsdf38} we have that for any $t\geq0$ and $\varphi_{0}\in K$,
	\begin{align}\label{4.dfsdf39}
		\mathbb{P}\{\omega\in\Omega: \varphi^{\varepsilon}(t,0,\varphi_{0}) \notin\mathfrak{K}_{m}^{\varepsilon}\}\leq\frac{\zeta}{2^{2k-1}}.
	\end{align}
	
	Let $\mathfrak{K}_\zeta^{\varepsilon}=\cap_{m=1}^{\infty}\mathfrak{K}^{\varepsilon}_{m}$. Then $\mathfrak{K}_\zeta^{\varepsilon}$ is compact in $\ell_c^{2}\times \ell^{2}$, since it is closed and totally bounded. Therefore, by \eqref{4.dfsdf39} we know that for any $t\geq0$ and $\varphi_{0}\in K$,
	\begin{align}\label{4.dfsdf40}
		\mathbb{P}\{\omega\in\Omega: \varphi^{\varepsilon}(t,0,\varphi_{0}) \notin\mathfrak{K}_\zeta^{\varepsilon}\}\leq
		\sum_{k=1}^{\infty}\frac{\zeta}{2^{2k-1}}<\zeta.
	\end{align}
	This proof is finished.
\end{proof}

In the sequel, we shall continue to study the existence of invariant measures of stochastic discrete long-wave-short-wave resonance equation (\ref{2.002}) with nonlinear noise on $\ell_c^{2}\times \ell^{2}$. Let $\Xi: \ell_c^{2} \times \ell^{2} \rightarrow \mathbb{R}$ be a bounded Borel function, we define 
\begin{equation}\label{5.21}
	\left(P_{s, t}^\varepsilon \Xi\right)(\phi)=\mathbb{E}\left[\Xi\left(\varphi^\varepsilon(t, s, \phi)\right)\right], \quad \forall\, 0\le s\le t,\,\, \phi \in \ell_c^{2} \times \ell^{2}.
\end{equation}
The family $\left\{P_{s, t}^\varepsilon\right\}_{0 \leq s \leq t}$ with parameter $\varepsilon$ is called the transition semigroup (or operator) associated with system (\ref{2.002}).  In particular, let  $\mathcal{B}\left(\ell_c^{2} \times \ell^{2}\right)$ denote the Borel $\sigma$-algebra of $\ell_c^{2} \times \ell^{2}$, then for any $\Gamma \in \mathcal{B}\left(\ell_c^{2} \times \ell^{2}\right)$ and $0\le s\le t$, we can define 
\begin{equation}\label{5.22}
P^\varepsilon(s, \phi; t, \Gamma)=\left(P^\varepsilon_{s,t}\chi_{\Gamma}\right)(\phi)=\mathbb{P}\left(\left\{\omega \in \Omega: \varphi^\varepsilon(t, s, \phi) \in \Gamma\right\}\right), \quad \phi \in \ell_c^{2} \times \ell^{2},
\end{equation}
where $\chi_{\Gamma}$ is a characteristic function of $\Gamma$. Note that $P^\varepsilon(s, \phi, t,\cdot)$ is the probability distribution of the solution $\varphi^\varepsilon(t,s,\phi)$, namely, 
$P^\varepsilon(s, \phi, t,\cdot)=\mathfrak{L}(\varphi^\varepsilon(t,s,\phi))$.
Furthermore, we give the following definition of invariant measures, see \cite{D-Prato1996} for more details. A probability measure $\mu$ on $(\ell_c^{2} \times \ell^{2}, \mathcal{B}(\ell_c^{2} \times \ell^{2}))$ is called an invariant measure if
\begin{equation}\label{5.23}
	\int_{\ell_c^{2} \times \ell^{2}}\left(P_t^\varepsilon \varLambda\right)(\phi) \mu(d \phi)=\int_{\ell_c^{2} \times \ell^{2}}\varLambda(\phi) \mu(d \phi),\quad  \forall t \geq 0,\,\, \varLambda \in C_b(X),
\end{equation}
where $ C_b(X) $ is the collection of all bounded continuous functions, and the Markov semigroup $P_t^\varepsilon$ is the
abbreviation of $P_{0, t}^\varepsilon$, it is defined as follows:
$$
P_t^\varepsilon\varLambda(\phi)=\int_H \varLambda(\xi)P(t,\phi,d\xi).
$$

We now demonstrate the Feller property of $P^\varepsilon_{s,t}$ for $0\leq s\leq t$, which is needed for proving the existence of invariant measures.
\begin{lemma}\label{lemma5.4}
	Suppose that the assumptions $(H_2)$, $(H_3)$, \eqref{LSWs4.1} and \eqref{LSWs5.1} hold. If $\Xi: \ell_c^{2} \times \ell^{2} \rightarrow \mathbb{R}$ is bounded and continuous, then for every $0\leq s\leq t$, the function $P_{s,t} \Xi:\ell_c^{2} \times \ell^{2}\rightarrow \mathbb{R}$ is also bounded and continuous.
\end{lemma}
\begin{proof}
	Let $\{z_n=(\widetilde{z}_n,\widehat{z}_n)\}_{n=1}^\infty$ be a sequence of $\ell_c^{2} \times \ell^{2}$ \footnote{$\{\widetilde{z}_n\}_{n=1}^\infty$ and $\{\widehat{z}_n\}_{n=1}^\infty$ are sequences of $\ell_c^2$ and $\ell^2$, respectively.}, and $z=(\widetilde{z},\widehat{z})\in \ell_c^{2} \times \ell^{2}$ such that $z_n \rightarrow z$  in $\ell_c^{2} \times \ell^{2}$ (i.e., $\widetilde{z}_n \to \widetilde{z}$ and $\widehat{z}_n \to \widehat{z}$ in $\ell_c^{2}$ and $ \ell^{2}$, respectively) as $n\to \infty$. We now proceed to prove the following limit
	\begin{align}\label{lemma5.4(5.35)}
		\lim_{n\rightarrow\infty} \mathbb{ E}\left[\Xi(\varphi(t_0,s_0,z_n))\right]=\mathbb{ E}\left[\Xi(\varphi(t_0,s_0,z))\right], \quad \forall t_0\geq s_0 \geq 0.
	\end{align}
	Since the set $\{z,z_n\}_{n=1}^\infty$ is compact in $\ell_c^{2} \times \ell^{2}$, then by using Lemma \ref{lemma SFf5.3} we can obtain that the family $\big\{\mathfrak{L}(\varphi(t_0,s_0,z)),\mathfrak{L}(\varphi(t_0,s_0,z_n))\big\}_{n=1}^\infty$ of the distributions of $\varphi(t_0,s_0,z)$ and $\varphi(t_0,s_0,z_n)$ is tight in $\ell_c^{2} \times \ell^{2}$. Thus, for every $\zeta>0$, there exists a compact subset $\mathfrak{K}_\zeta^{\varepsilon}$ of $\ell_c^{2} \times \ell^{2}$ such that for all $n\in \mathbb{N}$,
	\begin{align}\label{lemma5.4(5.36)}
		P^\varepsilon(s_0,z; t_0,\mathfrak{K}_\zeta^{\varepsilon})>1-\frac{\zeta}{4} \text{~~and~~}
		P^\varepsilon(s_0,z_n; t_0,\mathfrak{K}_\zeta^{\varepsilon})>1-\frac{\zeta}{4}.
	\end{align}
	
	Accoridng to the continuity of $\Xi$, it is easy to find that $\Xi$ is uniformly continuous in $\mathfrak{K}_\zeta^{\varepsilon}$. Then there exists $\widehat{\delta}>0$ such that for all ${\varphi}_1, {\varphi}_2\in \mathfrak{K}_\zeta^{\varepsilon}$ with $\|{\varphi}_1-{\varphi}_2\|<\widehat{\delta}$, there holds
	\begin{align}\label{lemma5.4(5.37)}
		\left|\Xi({\varphi}_1)-\Xi({\varphi}_2)\right|<\zeta.
	\end{align}
	Thanks to $z_n\to z$ in $\ell_c^{2} \times \ell^{2}$, 
	by Theorem \ref{th2.7} we obtain that there exists $\widehat{\textbf{c}}_1=\widehat{\textbf{c}}_1(s_0,t_0)>0$ independent of $n$ and $\varepsilon$ such that 
		\begin{align*}
		\mathbb{ E}\left[	\sup_{t\in [s_0,t_0]} \left(\|u(t,s_0,\widetilde{z}_n)\|^4+\|v(t,s_0,\widehat{z}_n)\|^2\right)\right]+\mathbb{ E}\left[	\sup_{t\in [s_0,t_0]} \left(\|u(t,s_0,\widetilde{z})\|^4+\|v(t,s_0,\widehat{z})\|^2\right)\right]\leq \widehat{\textbf{c}}_1.
	\end{align*}
	This together with Chebyshev's inequality can derive that for every $\mathcal{R}>0$,
	\begin{align}\label{lemma5.4(5.38)}
		\begin{split}
				&\mathbb{P}\left(\left\{\omega\in \Omega: \sup_{t\in [s_0,t_0]}\|\varphi(t,s_0,z_n)\|_{\ell_c^{2} \times \ell^{2}}>\mathcal{R}\right\}\right)\leq \frac{2\left(\sqrt{\widehat{\textbf{c}}_1}+\widehat{\textbf{c}}_1\right)}{\mathcal{R}^2},\\
				&\mathbb{P}\left(\left\{\omega\in \Omega: \sup_{t\in [s_0,t_0]}\|\varphi(t,s_0,z)\|_{\ell_c^{2} \times \ell^{2}}>\mathcal{R}\right\}\right)\leq \frac{2\left(\sqrt{\widehat{\textbf{c}}_1}+\widehat{\textbf{c}}_1\right)}{\mathcal{R}^2}.
			\end{split}
	\end{align}
	From \eqref{lemma5.4(5.38)}, we have that there exists $\widehat{\mathcal{R}}=\widehat{\mathcal{R}}(\zeta)>0$ such that  for all $n\in \mathbb{N}$,
	\begin{align}\label{lemma5.4(5.39)}
		\begin{split}
			&\mathbb{P}\left(\left\{\omega\in \Omega: \sup_{t\in [s_0,t_0]}\|\varphi(t,s_0,z_n)\|_{\ell_c^{2} \times \ell^{2}}>\widehat{\mathcal{R}}\right\}\right)< \frac{\zeta}{4},\\
			&\mathbb{P}\left(\left\{\omega\in \Omega: \sup_{t\in [s_0,t_0]}\|\varphi(t,s_0,z)\|_{\ell_c^{2} \times \ell^{2}}>\widehat{\mathcal{R}}\right\}\right)< \frac{\zeta}{4}.
		\end{split}
	\end{align}
	
	For $n\in \mathbb{N}$, let us define the following stopping time:
	\begin{align*}
		\tau_n =\inf \left\{t\geq s_0: \|\varphi(t,s_0,z_n)\|_{\ell_c^{2} \times \ell^{2}}>\widehat{\mathcal{R}}\right\}
		\text{~~and~~}
		\tau =\inf \left\{t\geq s_0: \|\varphi(t,s_0,z)\|_{\ell_c^{2} \times \ell^{2}}>\widehat{\mathcal{R}}\right\}.
	\end{align*}
	It follows from \eqref{opformLip}, \eqref{Flojo}, \eqref{Glojo}, and the argument in \eqref{LSWs2.68} that there exists $\widehat{\textbf{c}}_2=\widehat{\textbf{c}}_2(s_0,t_0,\widehat{\mathcal{R}})>0$ such that for all $n\in \mathbb{N}$,
	\begin{align}\label{lemma5.4(5.40)}
			&~~\mathbb{ E}\bigg[\sup_{t\in [s_0,t_0]} \big(\|u(t\wedge \tau_n\wedge \tau ,s_0,\widetilde{z}_n)-u(t\wedge \tau_n\wedge \tau ,s_0,\widetilde{z})\|^4\\
			&+\|v(t\wedge \tau_n\wedge \tau ,s_0,\widehat{z}_n)-v(t\wedge \tau_n\wedge \tau ,s_0,\widehat{z})\|^2\big)\bigg]\leq \widehat{\textbf{c}}_2\left(\|\widetilde{z}_n-\widetilde{z}\|^2+\|\widetilde{z}_n-\widetilde{z}\|^4+\|\widehat{z}_n-\widehat{z}\|^2\right),\notag
	\end{align}
	which implies that
	\begin{align}\label{lemma5.4(5.41)}
			&~~\mathbb{P}\left(\left\{\omega\in \Omega: \|\varphi(t_0\wedge \tau_n\wedge \tau ,s_0,{z}_n)-\varphi(t_0\wedge \tau_n\wedge \tau ,s_0,{z})\|_{\ell_c^{2} \times \ell^{2}}\geq \iota\right\}\right)\\
			&\leq \frac{2}{\iota^2}\left(\sqrt{\widehat{\textbf{c}}_2}\left(\|\widetilde{z}_n-\widetilde{z}\|^2+\|\widetilde{z}_n-\widetilde{z}\|^4+\|\widehat{z}_n-\widehat{z}\|^2\right)^{1/2}+{\widehat{\textbf{c}}_2}\left(\|\widetilde{z}_n-\widetilde{z}\|^2+\|\widetilde{z}_n-\widetilde{z}\|^4+\|\widehat{z}_n-\widehat{z}\|^2\right)\right).\notag
	\end{align}
	
	Given $n\in \mathbb{N}$, we define the following set
	\begin{align}\label{lemma5.4(5.42)}
		\begin{split}
			&\Omega_{\zeta,n}^{\varepsilon,1}=\left\{\omega \in \Omega: \varphi(t_0,s_0,{z}_n)\in \mathfrak{K}_\zeta^{\varepsilon} \text{~~~and~~} \sup_{t\in [s_0,t_0]}\|\varphi(t_0,s_0,{z}_n)\|_{\ell_c^{2} \times \ell^{2}}\leq\widehat{\mathcal{R}}\right\},\\
			&~~\Omega_{\zeta}^{\varepsilon,2}=\left\{\omega \in \Omega: \varphi(t_0,s_0,{z})\in \mathfrak{K}_\zeta^{\varepsilon} \text{~~~and~~} \sup_{t\in [s_0,t_0]}\|\varphi(t_0,s_0,{z})\|_{\ell_c^{2} \times \ell^{2}}\leq \widehat{\mathcal{R}}\right\}.
		\end{split}
	\end{align}
	Let $\Omega_{\zeta,n}^{\varepsilon}=\Omega_{\zeta,n}^{\varepsilon,1}\cap \Omega_{\zeta}^{\varepsilon,2}$, then by \eqref{lemma5.4(5.36)} and \eqref{lemma5.4(5.39)} we can deduce
	\begin{align*}
		\mathbb{P}\left(\Omega\backslash \Omega_{\zeta,n}^{\varepsilon}\right)\leq 	\mathbb{P}\left(\Omega\backslash \Omega_{\zeta,n}^{\varepsilon,1}\right)+	\mathbb{P}\left(\Omega\backslash \Omega_{\zeta}^{\varepsilon,2}\right)<\zeta,
	\end{align*}
	from which we have
	\begin{align}\label{lemma5.4(5.43)}
		\mathbb{P}\left(\Omega_{\zeta,n}^{\varepsilon}\right)>1-\zeta.
	\end{align}
	Owing to $\tau_n(\omega), \tau_n(\omega)\geq t_0$ for all $\omega \in \Omega_{\zeta,n}^{\varepsilon}$. Then we can get that for for all $\omega \in \Omega_{\zeta,n}^{\varepsilon}$, 
	\begin{align}\label{lemma5.4(5.44)}
			\varphi(t_0\wedge \tau_n\wedge \tau ,s_0,{z}_n)=\varphi(t_0,s_0,{z}_n) \text{~~and~~} \varphi(t_0\wedge \tau_n\wedge \tau ,s_0,{z})=\varphi(t_0\wedge \tau_n\wedge \tau ,s_0,{z}).
	\end{align}
	By \eqref{lemma5.4(5.44)} we have
	\begin{align}\label{lemma5.4(5.45)}
		\begin{split}
			&~~~~\mathbb{P}\left(\left\{\omega\in \Omega_{\zeta,n}^{\varepsilon}: \|\varphi(t_0 ,s_0,{z}_n)-\varphi(t_0 ,s_0,{z})\|_{\ell_c^{2} \times \ell^{2}}\geq \iota\right\}\right)\\
			&\leq \mathbb{P}\left(\left\{\omega\in \Omega: \|\varphi(t_0\wedge \tau_n\wedge \tau ,s_0,{z}_n)-\varphi(t_0\wedge \tau_n\wedge \tau ,s_0,{z})\|_{\ell_c^{2} \times \ell^{2}}\geq \iota\right\}\right),
		\end{split}
	\end{align}
	Combining with \eqref{lemma5.4(5.41)} and \eqref{lemma5.4(5.45)}, we infer that for all $n\in \mathbb{N}$,
	\begin{align}\label{lemma5.4(5.46)}
		&~~~~\mathbb{P}\left(\left\{\omega\in \Omega_{\zeta,n}^{\varepsilon}: \|\varphi(t_0 ,s_0,{z}_n)-\varphi(t_0 ,s_0,{z})\|_{\ell_c^{2} \times \ell^{2}}\geq \iota\right\}\right)\\
		&\leq \frac{2}{\iota^2}\left(\sqrt{\widehat{\textbf{c}}_2}\left(\|\widetilde{z}_n-\widetilde{z}\|^2+\|\widetilde{z}_n-\widetilde{z}\|^4+\|\widehat{z}_n-\widehat{z}\|^2\right)^{1/2}+{\widehat{\textbf{c}}_2}\left(\|\widetilde{z}_n-\widetilde{z}\|^2+\|\widetilde{z}_n-\widetilde{z}\|^4+\|\widehat{z}_n-\widehat{z}\|^2\right)\right).\notag
	\end{align}
	Since $\Xi$ is bounded, then there exists $\widehat{\textbf{c}}_3=\widehat{\textbf{c}}_3(\varphi)>0$ such that
	\begin{align}\label{lemma5.4(5.47)}
		|\Xi(\varphi)|\leq \widehat{\textbf{c}}_3, \quad \forall \varphi\in \ell_c^2\times \ell^2.
	\end{align}
	
	By \eqref{lemma5.4(5.46)} and \eqref{lemma5.4(5.47)} and \eqref{lemma5.4(5.37)} we have
	\begin{align}\label{lemma5.4(5.48)}
		\begin{split}
			&~~\int_{\Omega_{\zeta,n}^{\varepsilon}} \left|\Xi(\varphi(t_0 ,s_0,{z}_n))-\Xi(\varphi(t_0 ,s_0,{z}))\right|d\mathbb{P}\\
			&\leq  \int_{\Omega_{\zeta,n}^{\varepsilon}\cap \{\omega\in \Omega:\|\varphi(t_0 ,s_0,{z}_n)-\varphi(t_0,s_0,{z})\|_{\ell_c^{2} \times \ell^{2}}\geq \iota\}} \left|\Xi(\varphi(t_0 ,s_0,{z}_n))-\Xi(\varphi(t_0 ,s_0,{z}))\right|d\mathbb{P}\\
			&+ \int_{\Omega_{\zeta,n}^{\varepsilon}\cap \{\omega\in \Omega:\|\varphi(t_0 ,s_0,{z}_n)-\varphi(t_0,s_0,{z})\|_{\ell_c^{2} \times \ell^{2}}<\iota\}} \left|\Xi(\varphi(t_0 ,s_0,{z}_n))-\Xi(\varphi(t_0 ,s_0,{z}))\right|d\mathbb{P}\\
			&\leq \frac{4\widehat{\textbf{c}}_3}{\iota^2}\left(\sqrt{\widehat{\textbf{c}}_2}+{\widehat{\textbf{c}}_2}\right)\bigg(\left(\|\widetilde{z}_n-\widetilde{z}\|^2+\|\widetilde{z}_n-\widetilde{z}\|^4+\|\widehat{z}_n-\widehat{z}\|^2\right)^{1/2}\\
			&+\|\widetilde{z}_n-\widetilde{z}\|^2+\|\widetilde{z}_n-\widetilde{z}\|^4+\|\widehat{z}_n-\widehat{z}\|^2\bigg)+\zeta.
		\end{split}
	\end{align}
	Since as $n\rightarrow \infty$, $\widetilde{z}_n \to \widetilde{z}$ in $\ell_c^{2}$ and $\widehat{z}_n \to \widehat{z}$ in $ \ell^{2}$, respectively. Then by \eqref{lemma5.4(5.48)} we get
	\begin{align*}
		\limsup_{n\rightarrow \infty}\int_{\Omega_{\zeta,n}^{\varepsilon}} \left|\Xi(\varphi(t_0 ,s_0,{z}_n))-\Xi(\varphi(t_0 ,s_0,{z}))\right|d\mathbb{P}\leq \zeta.
	\end{align*}
	Furthermore, it follows from \eqref{lemma5.4(5.43)} and \eqref{lemma5.4(5.47)} that
	\begin{align*}
		\int_{\Omega\backslash \Omega_{\zeta,n}^{\varepsilon}} \left|\Xi(\varphi(t_0 ,s_0,{z}_n))-\Xi(\varphi(t_0 ,s_0,{z}))\right|d\mathbb{P}\leq 2\widehat{\textbf{c}}_3\zeta.
	\end{align*}
	Therefore, we have
	$
	\limsup_{n\rightarrow \infty}\int_{\Omega} \left|\Xi(\varphi(t_0 ,s_0,{z}_n))-\Xi(\varphi(t_0 ,s_0,{z}))\right|d\mathbb{P}\leq (1+2\widehat{\textbf{c}}_3)\zeta,
	$
	which along with the arbitrariness of $\zeta$ can obtain
	$$
	\limsup_{n\rightarrow \infty}\int_{\Omega} \left|\Xi(\varphi(t_0 ,s_0,{z}_n))-\Xi(\varphi(t_0 ,s_0,{z}))\right|d\mathbb{P}= 0.
	$$
    This completes the proof.
\end{proof}

At the end of this section, we give the existence theorem of invariant measures. This can be proved by using Krylov-Bogolyubov method.
\begin{theorem}\label{therorem5.4}
Suppose that the assumptions $(H_2)$, $(H_3)$, \eqref{LSWs4.1} and \eqref{LSWs5.1} hold.  Then for every $\varepsilon \in \left[0,\varepsilon_0\right]$ with $\varepsilon_0=\max\left\{\sqrt{\frac{\alpha}{24\|\delta\|^2}},\sqrt{\frac{\beta}{48\|\delta\|^2}}\right\}$, the stochastic discrete long-wave-short-wave resonance equation (\ref{2.002}) possesses an invariant measure on $ \ell_c^{2} \times \ell^{2}$.
\end{theorem}
\begin{proof}
    We organize the proof into four steps, following a standard methodology.
	
	\emph{\bf Step 1. Feller property.} It follows from Lemma \ref{lemma5.4} that the transition semigroup $P^{\varepsilon}_{s,t}$ is Feller for any $0\leq s \leq t$.  
	
	\emph{\bf Step 2. Markov property.} By the uniqueness of solutions to system (\ref{2.002}) and the reasoning in Section 9 of \cite{D-Prato1996}, we can obtain that for every bounded and continuous $\Xi: \ell_c^{2} \times \ell^{2} \rightarrow \mathbb{R}$, the solution $\varphi(t,s,\varphi_0)$ of (\ref{2.002}) is a $\ell_c^{2} \times \ell^{2}$-valued Markov process, that is, for any $0\leq s \leq r\leq t$ and $\varphi_{0} \in L^4(\Omega, \ell_c^2)\times L^2(\Omega,\ell^2)$,  $
		\mathbb{E}\left[\Xi(\varphi(t,r,\varphi_0))|\mathcal{F}_r\right]=\mathbb{E}\left[\Xi(\varphi(t,r,z))\right]|_{z=\varphi_0}$ $\mathbb{P}$-a.s.,
	i.e.,
	$$
		\mathbb{E}\left[\Xi(\varphi(t,r,\varphi_0))|\mathcal{F}_r\right]=(P^\varepsilon_{r,t}\Xi)(z)|_{z=\varphi_0},\quad \mathbb{P}-a.s.,
$$
	from which we have
	\begin{equation*}
		\mathbb{E}\left[\Xi(\varphi(t,s,\varphi_0))|\mathcal{F}_r\right]=(P^\varepsilon_{r,t}\Xi)(z)|_{z=\varphi(r,s,\varphi_0)},\quad \mathbb{P}-a.s.
	\end{equation*}
This implies $\ell_c^{2} \times \ell^{2}$-valued Markov property of solution $\varphi(t,s,\varphi_0)$ of system (\ref{2.002}).
	
	\emph{\bf Step 3. Chapman-Kolmogorov equation.} As an immediate consequence of \emph{\bf Step 2}, we can get that for any bounded Borel function $\Xi\in \ell_c^{2} \times \ell^{2} \rightarrow \mathbb{R}$ and $0\leq s\leq r\leq t$,
	\begin{equation}\label{svelbeq3.38}
		(P^\varepsilon_{s,t}\Xi)(\varphi_0)=(P^\varepsilon_{s,r}(P^\varepsilon_{r,t}\Xi))(\varphi_0),\quad \mathbb{P}-a.s.,
	\end{equation}
	where $\varphi_{0} \in L^4(\Omega, \ell_c^2)\times L^2(\Omega,\ell^2)$. In particular,
	and for every $\Gamma \in \mathcal{B}\left(\ell_c^{2} \times \ell^{2}\right)$, the following Chapman-Kolmogorov equation 
	\begin{equation}\label{svelbeq3.39}
		\mathbb{P}(s,\varphi_0;t,\Gamma)=\int_{\ell_c^{2} \times \ell^{2}}P^\varepsilon(s,\varphi_0;r,dy)P^\varepsilon(r,y;t,\Gamma),\quad \forall \, 0\leq s\leq r\leq t
	\end{equation}
	holds for any $\varphi_{0} \in L^4(\Omega, \ell_c^2)\times L^2(\Omega,\ell^2)$.
	
	\emph{\bf Step 4. Translation property.} According to the argument of \cite{D-Prato1996}, the operator $P^\varepsilon_{s,t}$ satisfies the following translation property 
	\begin{equation}\label{svelbeq3.40}
		P^\varepsilon(s,\varphi_0;t,\cdot)=P^\varepsilon(0,\varphi_0;t-s,\cdot)
	\end{equation}
	holds for any $\varphi_{0} \in L^4(\Omega, \ell_c^2)\times L^2(\Omega,\ell^2)$ and $0\leq s\leq r\leq t$.
	
	\emph{\bf Step 5. Existence of invariant measures.} We consider the solution $\varphi(t,0,\varphi_0)$ of system (\ref{2.002}) with initial data $\varphi_0\in L^4(\Omega, \ell_c^2)\times L^2(\Omega,\ell^2)$, and let
	\begin{equation}\label{svelbeq3.41}
		{\mu}_n(\cdot)=\frac{1}{n}\int_0^n P^\varepsilon(0,\varphi_0;t,\cdot)dt, \quad \forall n\in \mathbb{N}.
	\end{equation}
	By virtue of Lemma \ref{lemma SFf5.3} we get that the family $\{{\mu}_n\}_{n=1}^\infty$ is tight on $\ell_c^{2} \times \ell^{2}$. Thus, there exists a probability measure ${\mu}\in \ell_c^{2} \times \ell^{2}$ such that,  up to a subsequence, ${\mu}_{n_k}\rightarrow \tilde{\mu}$ as $k\rightarrow \infty$. This along with \eqref{svelbeq3.41} 
	 can get that for any $s,t\geq 0$ and $\varLambda\in C_b(H)$,
	\begin{align*}
			\int_{\ell_c^{2} \times \ell^{2}} \varLambda  (x)d \mu (x) \hfill
			&= \mathop {\lim }\limits_{n \to \infty } \frac{1}{n}\int_0^n {\left( {\int_{\ell_c^{2} \times \ell^{2}} \varLambda  (x)P^\varepsilon(0,\varphi_0;s,dx)} \right)} ds \hfill \\
			&= \mathop {\lim }\limits_{n \to \infty } \frac{1}{n}\int_{ - s}^{n - s} {\left( {\int_{\ell_c^{2} \times \ell^{2}} \varLambda (x)P^\varepsilon(0,\varphi_0;s + t,dx)} \right)} ds\hfill \\
			&= {\mathop {\lim }\limits_{n \to \infty } \frac{1}{n}\int_{ - s}^0 {\left( {\int_{\ell_c^{2} \times \ell^{2}} \varLambda  (x)P^\varepsilon(0,\varphi_0;s + t,dx)} \right)} ds} \hfill \\
			&+ \mathop {\lim }\limits_{n \to \infty } \frac{1}{n}\int_0^n {\left( {\int_{\ell_c^{2} \times \ell^{2}} \varLambda  (x)P^\varepsilon(0,\varphi_0;s + t,dx)} \right)} ds \hfill \\
			&+ {\mathop {\lim }\limits_{n \to \infty } \frac{1}{n}\int_n^{n - s} {\left( {\int_{\ell_c^{2} \times \ell^{2}} \varLambda  (x)P^\varepsilon(0,\varphi_0;s + t,dx)} \right)} ds} ,
	\end{align*}
 		which, together with \eqref{svelbeq3.40}, can deduce
		\begin{align*}	
			\int_{\ell_c^{2} \times \ell^{2}} \varLambda  (x)d \mu (x) &= \mathop {\lim }\limits_{n \to \infty } \frac{1}{n}\int_{ 0}^n {\left( {\int_{\ell_c^{2} \times \ell^{2}} \varLambda  (x)P^\varepsilon(0,\varphi_0;s + t,dx)} \right)} ds \hfill \\
			&= \mathop {\lim }\limits_{n \to \infty } \frac{1}{n}\int_{ 0}^n {\left( {\int_{\ell_c^{2} \times \ell^{2}} \varLambda  (x)\int_{\ell_c^{2} \times \ell^{2}} {P^\varepsilon(s,y ;s + t,dx)P^\varepsilon(0,\varphi_0;s,dy )} } \right)} ds \hfill \\
			&= \mathop {\lim }\limits_{n \to \infty } \frac{1}{n}\int_{ 0}^n {\left( {\int_{\ell_c^{2} \times \ell^{2}} {\int_{\ell_c^{2} \times \ell^{2}} \varLambda (x)P^\varepsilon(0,y ;t,dx)P^\varepsilon(0,{\varphi_0};s,dy ) } } \right)} ds \hfill \\
			&= \int_{\ell_c^{2} \times \ell^{2}} {\left( {\int_{\ell_c^{2} \times \ell^{2}} {\varLambda (x)P^\varepsilon(0,y ;t,dx)} } \right)} d\mu (y ) \hfill\\
			&= \int_{\ell_c^{2} \times \ell^{2}} {\left( {\int_{\ell_c^{2} \times \ell^{2}} {\varLambda (y)P^\varepsilon(0,x;t,dy)} } \right)} d \mu (x).
	\end{align*}
	This completes the proof.
\end{proof}

\section{Limiting behavior of invariant measures with respect to the noise intensity}

This section investigates the limiting behavior between the invariant measures of the stochastic discrete long-wave-short-wave resonance equation (\ref{2.002}) and its deterministic counterpart with $\varepsilon=0$, under the assumption that the external forcing terms $f$, $g$, $b_k$, and $\gamma_k$ are independent of time $t$ and the sample $\omega \in \Omega$. To this end, we first establish the following the convergence result in probability of the solutions, which will be used in the proof of our main theorem. 
\begin{lemma}\label{lemma6.1}
Suppose that the assumptions $(H_2)$, $(H_3)$, \eqref{LSWs4.1} and \eqref{LSWs5.1} hold. Then for any bounded subset $\mathfrak{B}(:=\widetilde{\mathfrak{B}}\times \widehat{\mathfrak{B}})$ of $\ell_c^{2} \times \ell^{2}$ \footnote{$\widetilde{\mathfrak{B}}$ and $\widehat{\mathfrak{B}}$ are bounded subsets of $\ell_c^2$ and $\ell^2$, respectively.},  $ T>0 ,$ $ \eta>0, $ and $ \varepsilon_1,\varepsilon_2\in \left[0,\varepsilon_0\right]$ with $\varepsilon_0=\max\left\{\sqrt{\frac{\alpha}{24\|\delta\|^2}},\sqrt{\frac{\beta}{48\|\delta\|^2}}\right\}$, there holds
\begin{equation}\label{6.1}
	\lim _{\varepsilon_1 \rightarrow \varepsilon_2} \sup _{\varphi_0 \in \mathfrak{B}} \mathbb{P}\left( \left\{ \omega\in\Omega:  \sup_{0 \le t\le T}  \left\|\varphi^{\varepsilon_1}\left(t,0, \varphi_0\right)-\varphi^{\varepsilon_2}\left(t,0, \varphi_0 \right)\right\|_{\ell_c^2\times \ell^2} \geq \eta       \right\} \right)=0,
\end{equation}
where $\varphi^{\varepsilon_1}\left(t, 0,\varphi_0 \right) =(u^{\varepsilon_1}\left(t,0, u_{0 }\right),v^{\varepsilon_1}\left(t ,0, v_{0 }\right))^\mathrm{T}  $   and   $\varphi^{\varepsilon_2}\left(t,0, \varphi_{0}\right) =(u^{\varepsilon_2}\left(t, 0, u_{0}\right),v^{\varepsilon_2}\left(t, 0,v_{0}\right)) ^\mathrm{T} $ are two solutions of system \eqref{2.002}	with initial value $\varphi_{0}=(u_{0},v_{0})^\mathrm{T}$.
\end{lemma}
\begin{proof}
For simplicity's sake, we use $ \varphi^{\varepsilon_1}(t)=(u^{\varepsilon_1}(t),v^{\varepsilon_1}(t))^\mathrm{T}$  and $ \varphi^{\varepsilon_2}(t)=(u^{\varepsilon_2}(t),v^{\varepsilon_2}(t))^\mathrm{T}$ to denote $\varphi^{\varepsilon_1}\left(t, 0,\varphi_0 \right) =(u^{\varepsilon_1}\left(t,0, u_{0 }\right),v^{\varepsilon_1}\left(t ,0, v_{0 }\right))^\mathrm{T}  $   and   $\varphi^{\varepsilon_2}\left(t,0, \varphi_{0}\right) =(u^{\varepsilon_2}\left(t, 0, u_{0}\right),v^{\varepsilon_2}\left(t, 0,v_{0}\right)) ^\mathrm{T} $, respectively. 
For any $n\in\mathbb{N}$ and $ T>0 $, we define the following stopping time $\mathfrak{T}^n$:
\begin{equation}\label{6.2}
	\mathfrak{T}^n= \inf\{ t\ge 0:   \|\varphi^{\varepsilon_1}(t)\|\ge n \} \wedge  \inf\{ t\ge0:   \|\varphi^{\varepsilon_2}(t)\|\ge n \}\wedge T.
\end{equation}	
We can get from the equation $(\ref{2.002})$ the following indenties
\begin{align}\label{6.3}
			&u^{\varepsilon_1}(t\wedge \mathfrak{T}^n)-u^{\varepsilon_2}(t\wedge \mathfrak{T}^n)
			+\alpha \int_{0}^{t\wedge \mathfrak{T}^n}\left(u^{\varepsilon_1}(r)-u^{\varepsilon_2}(r)\right)dr+i\int_{0}^{t\wedge \mathfrak{T}^n}\left(Au^{\varepsilon_1}(r)-Au^{\varepsilon_2}(r)\right)dr\\ \notag
			&+i\int_{0}^{t\wedge \mathfrak{T}^n}\left(F\left({u^{\varepsilon_1}(r),v^{\varepsilon_1}(r)}\right)-F\left(u^{\varepsilon_2}(r),v^{\varepsilon_2}(r)\right)\right)dr\\ \notag
			&=-i \varepsilon_2\sum_{k=1}^\infty \int_{0}^{t\wedge \mathfrak{T}^n}\left(h_{k}\left(u^{\varepsilon_1}(r)\right)
			-h_{k}\left(u^{\varepsilon_2}(r)\right)\right)dW_k(r)-i(\varepsilon_1-\varepsilon_2) \sum_{k=1}^\infty \int_{0}^{t\wedge \mathfrak{T}^n}\left(h_{k}\left(u^{\varepsilon_1}(r)\right)+b_k\right)dW_k(r),  \notag
	\end{align}
	and
	\begin{align}\label{6.4}
			&v^{\varepsilon_1}(t\wedge \mathfrak{T}^n)-v^{\varepsilon_2}(t\wedge \mathfrak{T}^n)
			+\beta \int_{0}^{t\wedge \mathfrak{T}^n}\left(v^{\varepsilon_1}(r)-v^{\varepsilon_2}(r)\right)dr+\int_{0}^{t\wedge \mathfrak{T}^n}\left(G\left(u^{\varepsilon_1}(r)\right)-G\left(u^{\varepsilon_2}(r)\right)\right)dr\\
			&=\varepsilon_2 \sum_{k=1}^\infty \int_{0}^{t\wedge \mathfrak{T}^n}\left(\sigma_{k}\big(v^{\varepsilon_1}(r)\big)
			-\sigma_{k}\big(v^{\varepsilon_2}(r)\big)\right)dW_k(r)+(\varepsilon_1-\varepsilon_2) \sum_{k=1}^\infty \int_{0}^{t\wedge \mathfrak{T}^n}\left(\sigma_{k}\big(v^{\varepsilon_1}(r)\big)
			+\gamma_{k}\right)dW_k(r). \notag
	\end{align}
	Applying Ito's formula and combining with \eqref{6.3}, by taking the real part we have that a.s., 
	\begin{align}\label{6.5}
			&\|u^{\varepsilon_1}(t\wedge \mathfrak{T}^n)-u^{\varepsilon_2}(t\wedge\mathfrak{T}^n)\|^2+\|u^{\varepsilon_1}(t\wedge \mathfrak{T}^n)-u^{\varepsilon_2}(t\wedge\mathfrak{T}^n)\|^4\notag \\
			&+2\alpha \int_{0}^{t\wedge \mathfrak{T}^n}\|u^{\varepsilon_1}(r)-u^{\varepsilon_2}(r)\|^2dr+4\alpha \int_{0}^{t\wedge \mathfrak{T}^n}\|u^{\varepsilon_1}(r)-u^{\varepsilon_2}(r)\|^2dr\notag \\
			&\leq 2\int_{0}^{t\wedge \mathfrak{T}^n}\left|\left(F\left({u^{\varepsilon_1}(r),v^{\varepsilon_1}(r)}\right)-F\left(u^{\varepsilon_2}(r),v^{\varepsilon_2}(r)\right),u^{\varepsilon_1}(r)-u^{\varepsilon_2}(r)\right)\right|dr\notag \\
			&+4\int_{0}^{t\wedge \mathfrak{T}^n}\|u^{\varepsilon_1}(r)-u^{\varepsilon_2}(r)\|^2\left|\left(F\left({u^{\varepsilon_1}(r),v^{\varepsilon_1}(r)}\right)-F\left(u^{\varepsilon_2}(r),v^{\varepsilon_2}(r)\right),u^{\varepsilon_1}(r)-u^{\varepsilon_2}(r)\right)\right|dr\notag \\
			&+\sum_{k=1}^\infty \int_{0}^{t\wedge \mathfrak{T}^n}\|\varepsilon_2\left(h_{k}\left(u^{\varepsilon_1}(r)\right)
			-h_{k}\left(u^{\varepsilon_2}(r)\right)
			+\left(\varepsilon_1-\varepsilon_2\right)\left(h_{k}\left(u^{\varepsilon_1}(r)\right)+b_k\right)\right)\|^{2}dr\notag \\
			&+6\sum_{k=1}^\infty \int_{0}^{t\wedge \mathfrak{T}^n}\|u^{\varepsilon_1}(r)-u^{\varepsilon_2}(r)\|^2\|\varepsilon_2\left(h_{k}\left(u^{\varepsilon_1}(r)\right)
			-h_{k}\left(u^{\varepsilon_2}(r)\right)
			+\left(\varepsilon_1-\varepsilon_2\right)\left(h_{k}\left(u^{\varepsilon_1}(r)\right)+b_k\right)\right)\|^{2}dr\notag \\
			&+2 \sum_{k=1}^\infty \int_{0}^{t\wedge \mathfrak{T}^n}\textbf{Im}\big(u^{\varepsilon_1}(r)-u^{\varepsilon_2}(r),\varepsilon_2\left(h_{k}\left(u^{\varepsilon_1}(r)\right)
			-h_{k}\left(u^{\varepsilon_2}(r)\right)\right)\notag \\
			&\qquad \qquad  +\left(\varepsilon_1-\varepsilon_2\right)\left(h_{k}\left(u^{\varepsilon_1}(r)\right)+b_k\right)\big)dW_k(r)\notag \\
			&+4 \sum_{k=1}^\infty \int_{0}^{t\wedge \mathfrak{T}^n}\|u^{\varepsilon_1}(r)-u^{\varepsilon_2}(r)\|^2\textbf{Im}\big(u^{\varepsilon_1}(r)-u^{\varepsilon_2}(r),\varepsilon_2\left(h_{k}\left(u^{\varepsilon_1}(r)\right)
			-h_{k}\left(u^{\varepsilon_2}(r)\right)\right)\notag \\
			&\qquad \qquad  +\left(\varepsilon_1-\varepsilon_2\right)\left(h_{k}\left(u^{\varepsilon_1}(r)\right)+b_k\right)\big)dW_k(r).
	\end{align}
	By \eqref{6.4} and Ito's formula we can also derive that a.s.,
	\begin{align}\label{6.6}
			&\|v^{\varepsilon_1}(t\wedge \mathfrak{T}^n)-v^{\varepsilon_2}(t\wedge\tau_{n})\|^2
			+2\beta \int_{0}^{t\wedge \mathfrak{T}^n}\|v^{\varepsilon_1}(r)-v^{\varepsilon_2}(r)\|^2dr\notag \\
			&\leq 2\int_{0}^{t\wedge \mathfrak{T}^n}\left|\left(G\left(u^{\varepsilon_1}(r)\right)-G\left(u^{n}(r)\right),v^{\varepsilon_1}(r)-v^{n}(r)\right)\right|dr\notag \\
			&+\sum_{k=1}^\infty \int_{0}^{t\wedge \mathfrak{T}^n}\|\varepsilon_2\left(\sigma_{k}\big(v^{\varepsilon_1}(r)\big)
			-\sigma_{k}\big(v^{\varepsilon_2}(r)\big)\right)+(\varepsilon_1-\varepsilon_2)\left(\sigma_{k}\big(v^{\varepsilon_1}(r)\big)+\gamma_{k}\right)\|^{2}dr\notag \\
			&+2\sum_{k=1}^\infty \int_{0}^{t\wedge \mathfrak{T}^n}\big(v^{\varepsilon_1}(r)-v^{\varepsilon_2}(r),\varepsilon_2\left(\sigma_{k}\big(v^{\varepsilon_1}(r)\big)
			-\sigma_{k}\big(v^{\varepsilon_2}(r)\big)\right)\notag \\
			&\qquad \qquad  +\left(\varepsilon_1-\varepsilon_2\right)\left(\sigma_{k}\left(v^{\varepsilon_1}(r)\right)+\gamma_k\right)\big)dW_k(r).
	\end{align}
	By summing (\ref{6.5})-(\ref{6.6}) we deduce that
	\begin{align}\label{6.7}
		&\|u^{\varepsilon_1}(t\wedge \mathfrak{T}^n)-u^{\varepsilon_2}(t\wedge\mathfrak{T}^n)\|^2+\|u^{\varepsilon_1}(t\wedge \mathfrak{T}^n)-u^{\varepsilon_2}(t\wedge\mathfrak{T}^n)\|^4+\|v^{\varepsilon_1}(t\wedge \mathfrak{T}^n)-v^{\varepsilon_2}(t\wedge\mathfrak{T}^n)\|^2 \notag \\ 
		&\leq 2\int_{0}^{t\wedge \mathfrak{T}^n}\left|\left(F\left({u^{\varepsilon_1}(r),v^{\varepsilon_1}(r)}\right)-F\left(u^{\varepsilon_2}(r),v^{\varepsilon_2}(r)\right),u^{\varepsilon_1}(r)-u^{\varepsilon_2}(r)\right)\right|dr\notag \\
		&+4\int_{0}^{t\wedge \mathfrak{T}^n}\|u^{\varepsilon_1}(r)-u^{\varepsilon_2}(r)\|^2\left|\left(F\left({u^{\varepsilon_1}(r),v^{\varepsilon_1}(r)}\right)-F\left(u^{\varepsilon_2}(r),v^{\varepsilon_2}(r)\right),u^{\varepsilon_1}(r)-u^{\varepsilon_2}(r)\right)\right|dr\notag \\ 
		&+2\int_{0}^{t\wedge \mathfrak{T}^n}\left|\left(G\left(u^{\varepsilon_1}(r)\right)-G\left(u^{n}(r)\right),v^{\varepsilon_1}(r)-v^{n}(r)\right)\right|dr\notag \\ 
		&+\sum_{k=1}^\infty \int_{0}^{t\wedge \mathfrak{T}^n}\|\varepsilon_2\left(h_{k}\left(u^{\varepsilon_1}(r)\right)
		-h_{k}\left(u^{\varepsilon_2}(r)\right)
		+\left(\varepsilon_1-\varepsilon_2\right)\left(h_{k}\left(u^{\varepsilon_1}(r)\right)+b_k\right)\right)\|^{2}dr\notag \\ 
		&+6\sum_{k=1}^\infty \int_{0}^{t\wedge \mathfrak{T}^n}\|u^{\varepsilon_1}(r)-u^{\varepsilon_2}(r)\|^2\|\varepsilon_2\left(h_{k}\left(u^{\varepsilon_1}(r)\right)
		-h_{k}\left(u^{\varepsilon_2}(r)\right)
		+\left(\varepsilon_1-\varepsilon_2\right)\left(h_{k}\left(u^{\varepsilon_1}(r)\right)+b_k\right)\right)\|^{2}dr\notag \\
		&+\sum_{k=1}^\infty \int_{0}^{t\wedge \mathfrak{T}^n}\|\varepsilon_2\left(\sigma_{k}\big(v^{\varepsilon_1}(r)\big)
		-\sigma_{k}\big(v^{\varepsilon_2}(r)\big)\right)+(\varepsilon_1-\varepsilon_2)\left(\sigma_{k}\big(v^{\varepsilon_1}(r)\big)+\gamma_{k}\right)\|^{2}dr\notag \\ 
		&+2 \sum_{k=1}^\infty \int_{0}^{t\wedge \mathfrak{T}^n}\textbf{Im}\big(u^{\varepsilon_1}(r)-u^{\varepsilon_2}(r),\varepsilon_2\left(h_{k}\left(u^{\varepsilon_1}(r)\right)
		-h_{k}\left(u^{\varepsilon_2}(r)\right)\right)\notag \\ 
		&\qquad \qquad  +\left(\varepsilon_1-\varepsilon_2\right)\left(h_{k}\left(u^{\varepsilon_1}(r)\right)+b_k\right)\big)dW_k(r)\notag \\ 
		&+4 \sum_{k=1}^\infty \int_{0}^{t\wedge \mathfrak{T}^n}\|u^{\varepsilon_1}(r)-u^{\varepsilon_2}(r)\|^2\textbf{Im}\big(u^{\varepsilon_1}(r)-u^{\varepsilon_2}(r),\varepsilon_2\left(h_{k}\left(u^{\varepsilon_1}(r)\right)
		-h_{k}\left(u^{\varepsilon_2}(r)\right)\right)\notag \\
		&\qquad \qquad  +\left(\varepsilon_1-\varepsilon_2\right)\left(h_{k}\left(u^{\varepsilon_1}(r)\right)+b_k\right)\big)dW_k(r)\notag \\
		&+2\sum_{k=1}^\infty \int_{0}^{t\wedge \mathfrak{T}^n}\big(v^{\varepsilon_1}(r)-v^{\varepsilon_2}(r),\varepsilon_2\left(\sigma_{k}\big(v^{\varepsilon_1}(r)\big)
		-\sigma_{k}\big(v^{\varepsilon_2}(r)\big)\right)\notag \\ 
		&\qquad \qquad  +\left(\varepsilon_1-\varepsilon_2\right)\left(\sigma_{k}\left(v^{\varepsilon_1}(r)\right)+\gamma_k\right)\big)dW_k(r)\notag \\ 
		&:=\sum_{l=1}^{9}\mathfrak{J}_l(t\wedge \mathfrak{T}^n).	
	\end{align}
	
Analogous to the proofs of \eqref{LSWs2.64} and \eqref{LSWs2.65}, the first and third terms on the right-hand side of \eqref{6.7} can be handled.  Namely, by \eqref{Flojo} and \eqref{Glojo} we can derive that for any bounded subsetset $\mathfrak{B}$ of $\ell_c^{2} \times \ell^{2}$, there exists a constant  $\textbf{c}_8=\textbf{c}_8(\mathfrak{B})>0$ independent of $\varepsilon_1$ and $\varepsilon_2$ such that  for all $t\in [0,T]$ with $T>0$,
\begin{align}\label{6.8}
			&\sum_{l=1}^{3}\mathfrak{J}_l(t\wedge \mathfrak{T}^n)\leq \textbf{c}_8\int_{0}^{t\wedge \mathfrak{T}^n}\left(\|u^{\varepsilon_1}(r)-u^{\varepsilon_2}(r)\|^2+\|u^{\varepsilon_1}(r)-u^{\varepsilon_2}(r)\|^4+\|v^{\varepsilon_1}(r)-v^{\varepsilon_2}(r)\|^2\right)dr.
	\end{align}

For the fourth and sixth terms on the right-hand side of (\ref{6.7}), by \eqref{opformLip} we can infer that there exists a constant  $\textbf{c}_9=\textbf{c}_9(\mathfrak{B})>0$ independent of $\varepsilon_1$ and $\varepsilon_2$ such that  for all $t\in [0,T]$ with $T>0$,
	\begin{align}\label{6.10}
	&\sum_{l=4}^{6}\mathfrak{J}_l(t\wedge \mathfrak{T}^n)\leq 2\varepsilon_2^2\int_{0}^{t\wedge \mathfrak{T}^n} \sum_{k=1}^\infty  \left(\|h_{k}\left(u^{\varepsilon_1}(r)\right)
	-h_{k}\left(u^{\varepsilon_2}(r)\right)\|^2+\|\sigma_{k}\big(v^{\varepsilon_1}(r)\big)
	-\sigma_{k}\big(v^{\varepsilon_2}(r)\big)\|^2\right)dr\notag \\
	&+2(\varepsilon_1-\varepsilon_2)^2 \int_{0}^{t\wedge \mathfrak{T}^n} \sum_{k=1}^\infty \left(\|h_{k}\left(u^{\varepsilon_1}(r)\right)+b_k\|^2+\|\sigma_{k}\big(v^{\varepsilon_1}(r)\big)+\gamma_{k}\|^2\right)dr\notag \\
	&+12\varepsilon_2^2\int_{0}^{t\wedge \mathfrak{T}^n} \|u^{\varepsilon_1}(r)-u^{\varepsilon_2}(r)\|^2\sum_{k=1}^\infty  \|h_{k}\left(u^{\varepsilon_1}(r)\right)
	-h_{k}\left(u^{\varepsilon_2}(r)\right)\|^2dr\notag \\
	&+12(\varepsilon_1-\varepsilon_2)^2 \int_{0}^{t\wedge \mathfrak{T}^n} \|u^{\varepsilon_1}(r)-u^{\varepsilon_2}(r)\|^2\sum_{k=1}^\infty \|h_{k}\left(u^{\varepsilon_1}(r)\right)+b_k\|^2dr\notag \\
	&\leq \textbf{c}_9\int_{0}^{t\wedge \mathfrak{T}^n}\left(\|u^{\varepsilon_1}(r)-u^{\varepsilon_2}(r)\|^2+\|u^{\varepsilon_1}(r)-u^{\varepsilon_2}(r)\|^4+\|v^{\varepsilon_1}(r)-v^{\varepsilon_2}(r)\|^2\right)dr\notag \\
	&+|\varepsilon_1-\varepsilon_2|^2\textbf{c}_9T\left(1+\|b\|^4+\|\gamma\|^4\right).
	\end{align}

	By \eqref{6.7}-\eqref{6.10}, we obtain that there exists $\textbf{c}_{10}=\textbf{c}_8+\textbf{c}_9$ such that for all $t\in [0,T]$ with $T>0$,
	\begin{align}\label{6.13}
		&~~\mathbb{E}\left[\sup_{0\leq r\leq t}\|u^{\varepsilon_1}(r\wedge \mathfrak{T}^n)-u^{\varepsilon_2}(r\wedge \mathfrak{T}^n)\|^2+\|v^{\varepsilon_1}(r\wedge \mathfrak{T}^n)-v^{\varepsilon_2}(r\wedge\mathfrak{T}^n)\|^2\right]\notag \\
		&\leq \textbf{c}_{10}\int_{0}^{t\wedge \mathfrak{T}^n}\left(\|u^{\varepsilon_1}(r)-u^{\varepsilon_2}(r)\|^2+\|u^{\varepsilon_1}(r)-u^{\varepsilon_2}(r)\|^4+\|v^{\varepsilon_1}(r)-v^{\varepsilon_2}(r)\|^2\right)dr\notag \\
		& +|\varepsilon_1-\varepsilon_2|^2\textbf{c}_{10}T\left(1+\|b\|^4+\|\gamma\|^4\right)+ \mathbb{E}\left[\sup_{\tau\leq s\leq t\wedge \mathfrak{T}^n}\left|\sum_{l=7}^{9}\mathfrak{J}_l(s)\right|\right].
	\end{align}

	We handle the last two terms on the right-hand side of \eqref{6.13}. By \eqref{opformLip} and the BDG inequality  we obtain that there exists $\textbf{c}_{11}=\textbf{c}_{11}(\mathfrak{B})>0$ independent of $\varepsilon_1$ and $\varepsilon_2$ such that  for all $t\in [0,T]$ with $T>0$,
	\begin{align}\label{6.14}
			&\mathbb{E}\left[\sup_{\tau\leq s\leq t\wedge \mathfrak{T}^n}\left|\mathfrak{J}_7(s)+\mathfrak{J}_8(s)\right|\right]\notag \\
			&\leq 8\sqrt{2} \mathbb{E}\bigg[\bigg(\int_{0}^{t\wedge \mathfrak{T}^n}\|u^{\varepsilon_1}(r)-u^{\varepsilon_2}(r)\|^{2}
			\sum_{k=1}^{\infty}\|\varepsilon_2\left(h_{k}\left(u^{\varepsilon_1}(r)\right)
			-h_{k}\left(u^{\varepsilon_2}(r)\right)\right)\notag \\
			&\qquad\qquad +\left(\varepsilon_1-\varepsilon_2\right)\left(h_{k}\left(u^{\varepsilon_1}(r)\right)+b_k\right)\|^{2}dr
			\bigg)^{\frac{1}{2}}\bigg]\notag \\
			&+16\sqrt{2} \mathbb{E}\bigg[\bigg(\int_{0}^{t\wedge \mathfrak{T}^n}\|u^{\varepsilon_1}(r)-u^{\varepsilon_2}(r)\|^{4}
			\sum_{k=1}^{\infty}\|\varepsilon_2\left(h_{k}\left(u^{\varepsilon_1}(r)\right)
			-h_{k}\left(u^{\varepsilon_2}(r)\right)\right)\notag \\
			&\qquad\qquad+\left(\varepsilon_1-\varepsilon_2\right)\left(h_{k}\left(u^{\varepsilon_1}(r)\right)+b_k\right)\|^{2}dr
			\bigg)^{\frac{1}{2}}\bigg]\notag \\
			&\leq 16\varepsilon_2^2 \mathbb{E}\left[\left(\int_{0}^{t\wedge \mathfrak{T}^n}\|u^{\varepsilon_1}(r)-u^{\varepsilon_2}(r)\|^{2}
			\sum_{k=1}^{\infty}\|h_{k}\left(u^{\varepsilon_1}(r)\right)
			-h_{k}\left(u^{\varepsilon_2}(r)\right)\|^{2}dr
			\right)^{\frac{1}{2}}\right]\notag \\
			&+16|\varepsilon_1-\varepsilon_2|^2 \mathbb{E}\left[\left(\int_{0}^{t\wedge \mathfrak{T}^n}\|u^{\varepsilon_1}(r)-u^{\varepsilon_2}(r)\|^{2}
			\sum_{k=1}^{\infty}\|h_{k}\left(u^{\varepsilon_1}(r)\right)+b_k\|^{2}dr
			\right)^{\frac{1}{2}}\right]\notag \\
			&+32\varepsilon_2^2\mathbb{E}\left[\left(\sup_{0\leq r\leq t}\|u^{\varepsilon_1}(r\wedge \mathfrak{T}^n)-u^{\varepsilon_2}(r\wedge \mathfrak{T}^n)\|^3\right)\left(\int_{0}^{t\wedge \mathfrak{T}^n}\sum_{k=1}^{\infty}
			\|h_{k}\left(u^{\varepsilon_1}(r)\right)
			-h_{k}\left(u^{\varepsilon_2}(r)\right)\|^{2}dr
			\right)^{\frac{1}{2}}\right]\notag \\
			&+32|\varepsilon_1-\varepsilon_2|^2 \mathbb{E}\left[\left(\sup_{0\leq r\leq t}\|u^{\varepsilon_1}(r\wedge \mathfrak{T}^n)-u^{\varepsilon_2}(r\wedge \mathfrak{T}^n)\|^3\right)\left(\int_{0}^{t\wedge \mathfrak{T}^n}
			\sum_{k=1}^{\infty}\|h_{k}\left(u^{\varepsilon_1}(r)\right)+b_k\|^{2}dr
			\right)^{\frac{1}{2}}\right]\notag \\
			&\leq \textbf{c}\mathbb{E}\bigg[\left(\sup_{0\leq r\leq t}\|u^{\varepsilon_1}(r\wedge \mathfrak{T}^n)-u^{\varepsilon_2}(r\wedge \mathfrak{T}^n)\|+\sup_{0\leq r\leq t}\|u^{\varepsilon_1}(r\wedge \mathfrak{T}^n)-u^{\varepsilon_2}(r\wedge \mathfrak{T}^n)\|^3\right)\notag \\
			&\qquad \qquad \cdot\left(\int_{0}^{t\wedge \mathfrak{T}^n}\sum_{k=1}^{\infty}
			\|h_{k}\left(u^{\varepsilon_1}(r)\right)
			-h_{k}\left(u^{\varepsilon_2}(r)\right)\|^{2}dr
			\right)^{\frac{1}{2}}\bigg]\notag \\
			&+\textbf{c}|\varepsilon_1-\varepsilon_2|^2 \mathbb{E}\bigg[\left(\sup_{0\leq r\leq t}\|u^{\varepsilon_1}(r\wedge \mathfrak{T}^n)-u^{\varepsilon_2}(r\wedge \mathfrak{T}^n)\|+\sup_{0\leq r\leq t}\|u^{\varepsilon_1}(r\wedge \mathfrak{T}^n)-u^{\varepsilon_2}(r\wedge \mathfrak{T}^n)\|^3\right)\notag \\
			&\qquad \qquad \cdot\left(\int_{0}^{t\wedge \mathfrak{T}^n}
			\sum_{k=1}^{\infty}\|h_{k}\left(u^{\varepsilon_1}(r)\right)+b_k\|^{2}dr
			\right)^{\frac{1}{2}}\bigg]\notag \\
			&\leq\frac{1}{2}\mathbb{E}\left[\sup_{0\leq r\leq t}\|u^{\varepsilon_1}(r\wedge \mathfrak{T}^n)-u^{\varepsilon_2}(r\wedge \mathfrak{T}^n)\|^{2}\right]+\frac{1}{2}\mathbb{E}\left[\sup_{0\leq r\leq t}\|u^{\varepsilon_1}(r\wedge \mathfrak{T}^n)-u^{\varepsilon_2}(r\wedge \mathfrak{T}^n)\|^{4}\right]\notag \\
			&+\textbf{c}_{11}\int_{0}^{t}\mathbb{E}\left[\sup_{\tau\leq r\leq s}\left(\|u^{\varepsilon_1}(r\wedge \mathfrak{T}^n)-u^{\varepsilon_2}(r\wedge  \mathfrak{T}^n)\|^{2}+\|u^{\varepsilon_1}(r\wedge \mathfrak{T}^n)-u^{\varepsilon_2}(r\wedge  \mathfrak{T}^n)\|^{4}\right)\right]ds\notag \\
			&
			+|\varepsilon_1-\varepsilon_2|^2\textbf{c}T(1+\|b\|^4),
	\end{align}
and simlarly, we derive that there exists $\textbf{c}_{12}=\textbf{c}_{12}(\mathfrak{B})>0$ independent of $\varepsilon_1$ and $\varepsilon_2$ such that  for all $t\in [0,T]$ with $T>0$,
\begin{align}\label{6.15}
	\begin{split}
		&~~\mathbb{E}\left[\sup_{\tau\leq s\leq t\wedge \mathfrak{T}^n}\left|\mathfrak{J}_9(s)\right|\right]\leq \frac{1}{2}\mathbb{E}\left(\sup_{0\leq r\leq t}\|v^{\varepsilon_1}(r\wedge\mathfrak{T}^n)-v^{\varepsilon_2}(r\wedge\mathfrak{T}^n)\|^{2}\right)\\
		&+\textbf{c}_{12}\int_{0}^{t}\mathbb{E}\left(\sup_{\tau\leq r\leq s}\|v^{\varepsilon_1}(r\wedge\mathfrak{T}^n)-v^{\varepsilon_2}(r\wedge\mathfrak{T}^n)\|^{2}\right)ds
		+|\varepsilon_1-\varepsilon_2|^2\textbf{c}_{12}T(1+\|\gamma\|^4).
	\end{split}
\end{align}

Substituting \eqref{6.14} and \eqref{6.15} into \eqref{6.13}, we get that there exists $\textbf{c}_{13}=\textbf{c}_{10}+\textbf{c}_{11}+\textbf{c}_{12}$ such that for all $t\in [0,T]$ with $T>0$,
\begin{align}\label{6.16}
			&~~\mathbb{E}\left[\sup_{0\leq r\leq t}\left(\|u^{\varepsilon_1}(r\wedge \mathfrak{T}^n)-u^{\varepsilon_2}(r\wedge \mathfrak{T}^n)\|^4+\|\varphi^{\varepsilon_1}(r\wedge \mathfrak{T}^n)-\varphi^{\varepsilon_2}(r\wedge\mathfrak{T}^n)\|_{\ell_c^2\times \ell^2}^2\right)\right]\notag \\
		&\leq \textbf{c}_{13}\int_{0}^{t}\mathbb{E}\left[\sup_{\tau\leq r\leq s}\left(\|u^{\varepsilon_1}(r\wedge \mathfrak{T}^n)-u^{\varepsilon_2}(r\wedge \mathfrak{T}^n)\|^4+\|\varphi^{\varepsilon_1}(r\wedge \mathfrak{T}^n)-\varphi^{\varepsilon_2}(r\wedge\mathfrak{T}^n)\|_{\ell_c^2\times \ell^2}^2\right)\right]ds\notag \\
		&
		+|\varepsilon_1-\varepsilon_2|^2\textbf{c}_{13}T(1+\|b\|^4+\|\gamma\|^4).
\end{align}
By Gronwall's lemma, we obtain from \eqref{6.16} that for all $t\in [0,T]$,
\begin{align}\label{6.17}
	\begin{split}
		&~\mathbb{E}\left[\sup_{0\leq r\leq t}\left(\|u^{\varepsilon_1}(r\wedge \mathfrak{T}^n)-u^{\varepsilon_2}(r\wedge \mathfrak{T}^n)\|^4+\|v^{\varepsilon_1}(r\wedge \mathfrak{T}^n)-v^{\varepsilon_2}(r\wedge\mathfrak{T}^n)\|^2
		\right)\right]\\
		&\leq |\varepsilon_1-\varepsilon_2|^2\textbf{c}_{13}T(1+\|b\|^4+\|\gamma\|^4) e^{\textbf{c}_{13} T}.
	\end{split}
\end{align}
Note that $\mathfrak{T}^n=T$ for large enough $n$ since  $\varphi^{\varepsilon_{1}}$ and $\varphi^{\varepsilon_{2}}$ is continuous in
$t$. Then, taking the limit in \eqref{6.17} as $n \to \infty$, we obtain
\begin{align}\label{6.18}
	\begin{split}
		&~\mathbb{E}\left[\sup_{0\leq t\leq T}\left(\|u^{\varepsilon_1}(t)-u^{\varepsilon_2}(t)\|^4+\|v^{\varepsilon_1}(t)-v^{\varepsilon_2}(t)\|^2
		\right)\right]\\
		&\leq |\varepsilon_1-\varepsilon_2|^2\textbf{c}_{13}T(1+\|b\|^4+\|\gamma\|^4) e^{\textbf{c}_{13} T}.
	\end{split}
\end{align}
Therefore, by \eqref{6.18} and Chebychev's inequality we can obtain 
\begin{align*}
	&~\sup _{\varphi_0 \in \mathfrak{B}} \mathbb{P}\left( \left\{ \omega\in\Omega:  \sup_{0 \le t\le T}  \left\|\varphi^{\varepsilon_1}\left(t,0, \varphi_0\right)-\varphi^{\varepsilon_2}\left(t,0, \varphi_0 \right)\right\|_{\ell_c^2\times \ell^2} \geq \eta       \right\} \right)\\
	&\leq \frac{2}{\eta^2}\mathbb{ E}\left[\sup_{0 \le t\le T}  \left(\left\|u^{\varepsilon_1}\left(t,0, u_0\right)-u^{\varepsilon_2}\left(t,0, u_0 \right)\right\|^2+\left\|v^{\varepsilon_1}\left(t,0, v_0\right)-v^{\varepsilon_2}\left(t,0, v_0 \right)\right\|^2\right) \right]\\
	&\leq \frac{2}{\eta^2}\left(\left(\mathbb{ E}\left[\sup_{0 \le t\le T}  \left\|u^{\varepsilon_1}\left(t,0, u_0\right)-u^{\varepsilon_2}\left(t,0, u_0 \right)\right\|^4 \right]\right)^{1/2}+\mathbb{ E}\left[\sup_{0 \le t\le T}  \left\|v^{\varepsilon_1}\left(t,0, v_0\right)-v^{\varepsilon_2}\left(t,0, v_0 \right)\right\|^2 \right]\right)\\
	&\leq \frac{2}{\eta^2}|\varepsilon_1-\varepsilon_2|\left( e^{\frac{\textbf{c}_{13} T}{2}}\sqrt{\textbf{c}_{13}T(1+\|b\|^4+\|\gamma\|^4)}+|\varepsilon_1-\varepsilon_2|\textbf{c}_{13}T(1+\|b\|^4+\|\gamma\|^4) e^{\textbf{c}_{13} T}\right)\\
	&\rightarrow 0, \quad \text{as~~} \varepsilon_{1}\rightarrow \varepsilon_{2}.
\end{align*}
This completes the proof.
\end{proof}

 To investigate the limiting behavior of invariant measures of system (\ref{2.002}) as $\varepsilon$ varies, we first establish the following key result. Note that Theorem \ref{therorem5.4} guarantees that the set $ \Im^\varepsilon $ mentioned below is non-empty.

\begin{theorem}\label{theorem6.2}
Suppose that the assumptions $(H_2)$, $(H_3)$, \eqref{LSWs4.1} and \eqref{LSWs5.1} hold. Then we have the following conclusions:

${(i)}$ The union $ \cup_{\varepsilon\in[0,\varepsilon_0]}\Im^\varepsilon $ is tight on $\ell_c^2 \times \ell^2$, where $ \Im^\varepsilon $ denotes the collection of all invariant measures of system (\ref{2.002}) for every $\varepsilon\in[0,\varepsilon_0]$ with $\varepsilon_0=\max\left\{\sqrt{\frac{\alpha}{24\|\delta\|^2}},\sqrt{\frac{\beta}{48\|\delta\|^2}}\right\}$.

${(ii)}$ If $\varepsilon_n \to 0 $ and $ \mu^{\varepsilon_n}\in \Im^{\varepsilon_n} $, then there exist a subsequence $\{\varepsilon_{n_k}\}_{k=1}^\infty$ of $\{\varepsilon_{n}\}_{n=1}^\infty$ and an invariant measure $ \mu^0\in \Im^0 $  such that $ \mu^{\varepsilon_{n_k}}\to \mu^0 $ weakly $\ell_c^2 \times \ell^2$.
\end{theorem}
\begin{proof}
{\bf Proof of $(i)$.} First, we observe that all estimates in Lemmas \ref{lemma 5.1} and \ref{lemma 5.2} are uniform in $ \varepsilon\in[0,\varepsilon_0]$. Therefore, one can verify that the union $ \cup_{\varepsilon\in[0,\varepsilon_0]}\Im^\varepsilon $ is tight. Since the proof is analogous to that of Lemma \ref{lemma SFf5.3}, we omit the details.

{\bf Proof of $(ii)$.} From part ${(i)}$, we know that the set $ \{ \mu^\varepsilon: \varepsilon\in[0,\varepsilon_0]\} $ is tight. Consequently, there exist a subsequence $\{\varepsilon_{n_k}\}_{k=1}^\infty$ of $\{\varepsilon_{n}\}_{n=1}^\infty$ and a probability measure $ \mu^0 $ such that $ \mu^{\varepsilon_{n_k}}\to \mu^0 $ weakly. Since $\varepsilon_{n_k} \to 0 $, it follows from Lemma \ref{lemma6.1} and \cite{ChenD_SCM-2020} that $ \mu $ is an invariant measure of system (\ref{2.002}) with $\varepsilon=0$, yielding $ \mu^0\in \Im^0 $. This completes the proof.
\end{proof}

\begin{theorem}
	Suppose that the assumptions $(H_2)$, $(H_3)$, \eqref{LSWs4.1} and \eqref{LSWs5.1} hold. If $ \mu^\varepsilon $ are the invariant measures of the systems {\eqref{2.002}} with $\varepsilon\in[0,\varepsilon_0]$, here $\varepsilon_0=\max\left\{\sqrt{\frac{\alpha}{24\|\delta\|^2}},\sqrt{\frac{\beta}{48\|\delta\|^2}}\right\}$, then $ \mu^{\varepsilon}\to \mu^0 $ weakly as $ \varepsilon\to0 $.
\end{theorem}
\begin{proof}
By Theorem \ref{therorem5.4}, we know that for every $\varepsilon\in[0,\varepsilon_0]$ with $\varepsilon_0=\max\left\{\sqrt{\frac{\alpha}{24\|\delta\|^2}},\sqrt{\frac{\beta}{48\|\delta\|^2}}\right\}$, system {\eqref{2.002}}  admits an invariant measure. Combining this fact with Theorem \ref{theorem6.2} yields the desired conclusion.
\end{proof}

\section{Conclusions and remarks}
This work investigates the existence and uniqueness of solutions to the stochastic discrete long-wave-short-wave resonance equation, the existence of a weak pullback mean random attractor, and the existence of invariant measures together with their limiting behavior with respect to the noise intensity.  The specific structure of the equation, particularly its coupled nonlinear terms, makes the classical $L^2(\Omega,C([\tau,\tau+T]\ell_c^2\times \ell^2))$ space unsuitable as the phase space, thus requiring analysis in the high-order space $L^4(\Omega,C([\tau,\tau+T],\ell_c^2))\times L^2(\Omega,C([\tau,\tau+T],\ell^2))$. This approach presents significant challenges in obtaining long-time uniform estimates of solutions, tail estimates, and convergence analysis. To address these issues, we develop several novel analytical techniques. Furthermore, the high-order phase space requires us to reestablish the tightness of the solution distributions and the Feller property of the transition semigroup for the stochastic discrete long-wave-short-wave resonance equation.

A limitation of this study is the lack of established mixing property and ergodicity for the invariant measures. In principle, demonstrating exponential convergence of solutions for the stochastic discrete long-wave-short-wave resonance equation over extended time scales would lead to exponential mixing of the invariant measures, and consequently establish their uniqueness and ergodicity. However, proving such exponential convergence in the $L^4(\Omega,\ell_c^2) \times L^2(\Omega,\ell^2)$ setting remains beyond the reach of current analytical techniques, thereby preventing us from obtaining these stronger qualitative results at the present stage.

It is noteworthy that under the local Lipschitz condition on the noise coefficients of the stochastic discrete long-wave-short-wave resonance equation, the results established in this paper can be fully extended to the more general high-order Bochner space $L^{2p}(\Omega,\ell_c^2) \times L^{p}(\Omega,\ell^2)$ $(p\geq 2)$. The main challenges in such an extension lie in establishing the $p$th-order It\^{o} energy equalities and deriving more sophisticated and refined high-order moment estimates, which will be addressed in our future work.

\section*{Declarations}

\subsection*{Author Contributions}
All authors contributed equally to this work. Specifically, X. Pan, J. Huang, J. Wu, and J. Zhang were involved in: conceptualization and methodology design; formal analysis and theoretical proofs; validation and investigation; writing-original draft preparation; and writing-review and editing. All authors have read and agreed to the published version of the manuscript.

\subsection*{Conflict of interest}
The authors have no conflicts to disclose.

\subsection*{Availability of date and materials}
Not applicable.

\newcounter{cankan}

\end{document}